\documentclass[12pt]{report}

\usepackage{tamuconfig}

\usepackage{times}
\usepackage[T1]{fontenc}
\usepackage{xcolor}
\usepackage{epstopdf}
\usepackage{amsmath}

\usepackage{graphicx}

\DeclareGraphicsExtensions{.png}

\graphicspath{ {./graphic/} }

\usepackage[colorlinks=true, linkcolor=black, citecolor=black, urlcolor=black]{hyperref}
\usepackage[noend]{algpseudocode}
\usepackage{algorithm}
\usepackage{amsmath}
\usepackage{amssymb}
\usepackage{bm,enumerate}
\usepackage{framed}
\usepackage{subcaption}
\newcommand{\inner}[2]{\langle #1,#2\rangle}
\newtheorem{thm}{Theorem}[section]
\newtheorem{theorem}[thm]{Theorem}

\newtheorem{lemma}[thm]{Lemma}
\newtheorem{prop}[thm]{Proposition}
\newtheorem{proposition}[thm]{Proposition}

\newtheorem{corollary}[thm]{Corollary}
\newtheorem{definition}[thm]{Definition}

\newcommand{\Ik}{{\mathcal{I}}_k}
\newcommand{\argmin}{\mathrm{argmin}\,}

\newcommand{\Prox}{\mathrm{Prox}}

\newcommand{\R}{\mathbb{R}}

\newcommand{\dom}{\mathrm{dom }}

\usepackage{booktabs}
\usepackage{longtable}

\begin{document}

\renewcommand{\tamumanuscripttitle}{Block Decomposable Methods for Large-Scale Optimization Problems}

\renewcommand{\tamupapertype}{Dissertation}

\renewcommand{\tamufullname}{Leandro Farias Maia}

\renewcommand{\tamudegree}{Doctor of Philosophy}
\renewcommand{\tamuchairone}{David Huckleberry Gutman}

\renewcommand{\tamumemberone}{Alfredo Garcia}
\newcommand{\tamumembertwo}{Sergiy Butenko}
\newcommand{\tamumemberthree}{Jonathan W. Siegel}
\renewcommand{\tamudepthead}{Lewis Ntaimo}

\renewcommand{\tamugradmonth}{August}
\renewcommand{\tamugradyear}{2025}
\renewcommand{\tamudepartment}{Industrial Engineering}
\newcommand{\rev}[1]{{\textcolor{blue}{#1}}}
\newcommand{\qedflush}{\hfill \scalebox{0.7}{$\blacksquare$}}

\newcommand{\key}{\alpha}
\newcommand{\SA}{\textproc{S-ADMM}}
\newcommand{\DA}{\textproc{A-ADMM}}
\newcommand{\DP}{\textproc{DP-ADMM}}
\newcommand{\PD}{\textproc{P-ADMM}}
\newcommand{\SD}{\textproc{SDD-ADMM}}
\newcommand{\SBAPen}{\textproc{S-BAPen}}
\newcommand{\SBIPP}{\textproc{B-IPP}}
\newcommand{\ABIPP}{\textproc{AB-IPP}}
\newcommand{\revision}[1]{{\textcolor{blue}{#1}}}
\newcommand{\tx}{\tilde x}
\newcommand{\ty}{\tilde y}
\newcommand{\tT}{\tilde T}
\newcommand{\tz}{\tilde z}
\newcommand{\tu}{\tilde u}
\newcommand{\tq}{\tilde q}
\newcommand{\tv}{\tilde v}
\newcommand{\tp}{\tilde p}
\newcommand{\tf}{\tilde f}
\newcommand{\psis}{\psi_s}
\newcommand{\psin}{\psi_n}
\newcommand{\ts}{\tilde\sigma}
\newcommand{\tgamma}{\tilde \gamma}
\newcommand{\mK}{\mathrm{E}}
\newcommand{\Lc}{{\cal L}_c}
\newcommand{\Il}{{\cal I}_\ell}
\newcommand{\mQ}{\mathrm{Q}}
\newcommand{\ulam}{\underline{\lambda}}
\newcommand{\lam}{\lambda}
\newcommand{\bConv}[1]{\overline{\mbox{\rm Conv}}\,(\R^{#1})}
\newcommand{\chione}{\sigma_1}
\newcommand{\chitwo}{\sigma_2}
\newcommand{\chisum}{\sigma_1+c\sigma_2}

%
%
%
%
%


\providecommand{\tabularnewline}{\\}

\begin{titlepage}
\begin{center}
\begin{doublespace}

\MakeUppercase{  \tamumanuscripttitle}
\end{doublespace}
\vspace{4em}

A \tamupapertype

\vspace{1em}

by

\vspace{1em}

\MakeUppercase{\tamufullname}

\vspace{4em}

\begin{singlespace}

Submitted to the Graduate and Professional School of \\
Texas A\&M University \\

in partial fulfillment of the requirements for the degree of \\
\end{singlespace}

\vspace{1em}

\MakeUppercase{\tamudegree}
\par\end{center}
\begin{doublespace}

\end{doublespace}
\begin{tabular}{ll}
 & \tabularnewline
& \cr
Chair of Committee, & \tamuchairone\tabularnewline
Committee Members, & \tamumemberone\tabularnewline
 & \tamumembertwo\tabularnewline
 & \tamumemberthree\tabularnewline
Head of Department, & \tamudepthead\tabularnewline

\end{tabular}

\vspace{3em}

\begin{center}
\tamugradmonth \hspace{2pt} \tamugradyear

\vspace{3em}

Major Subject: \tamudepartment \par
\vspace{3em}
Copyright \tamugradyear \hspace{.5em}\tamufullname 
\par\end{center}
\end{titlepage}
\pagebreak{}


%
%
%
%
%

\chapter*{ABSTRACT}
\addcontentsline{toc}{chapter}{ABSTRACT} 

\pagestyle{plain} 
\pagenumbering{roman} 
\setcounter{page}{2}

\indent This dissertation explores block decomposable methods for large-scale optimization problems. It focuses on alternating direction method of multipliers (ADMM) schemes and block coordinate descent (BCD) methods. Specifically, it introduces a new proximal ADMM algorithm and proposes two BCD methods.

The first part of the research presents a new proximal ADMM algorithm. This method is adaptive to all problem parameters and solves the proximal augmented Lagrangian (AL) subproblem inexactly. This adaptiveness facilitates the highly efficient application of the algorithm to a broad swath of practical problems. The inexact solution of the proximal AL subproblem overcomes many key challenges in the practical applications of ADMM. The resultant algorithm obtains an approximate solution of an optimization problem in a number of iterations that matches the state-of-the-art complexity for the class of proximal ADMM schemes.

The second part of the research focuses on an inexact proximal mapping for the class of block proximal gradient methods. Key properties of this operator is established, facilitating the derivation of convergence rates for the proposed algorithm. Under two error decreases conditions, the algorithm matches the convergence rate of its exactly computed counterpart. Numerical results demonstrate the superior performance of the algorithm under a dynamic error regime over a fixed one.

The dissertation concludes by providing convergence guarantees for the randomized BCD method applied to a broad class of functions, known as Hölder smooth functions. Convergence rates are derived for non-convex, convex, and strongly convex functions. These convergence rates match those furnished in the existing literature for the Lipschtiz smooth setting.

\pagebreak{}

%
%
%
%
%

\chapter*{ACKNOWLEDGMENTS}
\addcontentsline{toc}{chapter}{ACKNOWLEDGMENTS}  

\indent I would like to express my deepest gratitude to my advisor, Dr. David Huckleberry Gutman, for his invaluable guidance throughout my doctoral studies. This work would not have been possible without his support. Under his mentorship, I not only learned about continuous optimization but also gained insight into the broader research community. Since my first year, Dr. Gutman has been an extraordinary guide, from advising me on foundational research techniques to introducing me to influential conferences and collaborators in the field. I am profoundly grateful for the many ways he has shaped my academic journey.

I am also sincerely grateful to my research committee, Dr. Alfredo Garcia, Dr. Sergiy Butenko, and Dr. Jonathan W. Siegel, for their time and dedication in reviewing this document.

My thanks also extend to Dr. Ismael R. De Farias Junior, whose advice during our collaboration in the summer of 2017 played a key role in my decision to pursue a Ph.D. I am also grateful to Dr. Renato D.C. Monteiro at Georgia Tech, whose expertise over the past two years has been invaluable in advancing my understanding of optimization.

My Ph.D. journey, beginning at Texas Tech University and culminating at Texas A\&M University, has been enriched by the friendships and professional relationships across the Departments of Industrial and Systems Engineering and the Department of Mathematics and Statistics.

Finally, I want to thank my family for their unwavering support. To my loving wife, Thais, who has been by my side throughout these five years, thank you—your support made this accomplishment possible. I am also deeply grateful to my parents, Valmir and Zuleide, who are my main inspiration in life for perseverance, and my brothers, Junior and Luis, who have been a constant source of inspiration.

\pagebreak{}
%
%
%
%
%

\chapter*{CONTRIBUTORS AND FUNDING SOURCES}
\addcontentsline{toc}{chapter}{CONTRIBUTORS AND FUNDING SOURCES}  

\subsection*{Contributors}
This work was supported by a dissertation committee consisting of Professor David Huckleberry Gutman [advisor], Professor Alfredo Garcia, and Professor Sergiy Butenko of the Department of Industrial and Systems Engineering and Professor Jonathan W. Siegel of the Department of Mathematics.

All other work conducted for the dissertation was completed by the student independently.
\subsection*{Funding Sources}
Graduate study was supported by the NSF grant \#241032.
\pagebreak{}

%
%
%
%
%

\phantomsection
\addcontentsline{toc}{chapter}{TABLE OF CONTENTS}  

\begin{singlespace}
\renewcommand\contentsname{\normalfont} {\centerline{TABLE OF CONTENTS}}

\setcounter{tocdepth}{4} 

\setlength{\cftaftertoctitleskip}{1em}
\renewcommand{\cftaftertoctitle}{%
\hfill{\normalfont {Page}\par}}

\tableofcontents

\end{singlespace}

\pagebreak{}


\phantomsection
\addcontentsline{toc}{chapter}{LIST OF FIGURES}  

\renewcommand{\cftloftitlefont}{\center\normalfont\MakeUppercase}

\setlength{\cftbeforeloftitleskip}{-12pt} 
\renewcommand{\cftafterloftitleskip}{12pt}

\renewcommand{\cftafterloftitle}{%
\\[4em]\mbox{}\hspace{2pt}FIGURE\hfill{\normalfont Page}\vskip\baselineskip}

\begingroup

\begin{center}
\begin{singlespace}
\setlength{\cftbeforechapskip}{0.4cm}
\setlength{\cftbeforesecskip}{0.30cm}
\setlength{\cftbeforesubsecskip}{0.30cm}
\setlength{\cftbeforefigskip}{0.4cm}
\setlength{\cftbeforetabskip}{0.4cm}



\listoffigures

\end{singlespace}
\end{center}

\pagebreak{}

%
\phantomsection
\addcontentsline{toc}{chapter}{LIST OF TABLES}  

\renewcommand{\cftlottitlefont}{\center\normalfont\MakeUppercase}

\setlength{\cftbeforelottitleskip}{-12pt} 

\renewcommand{\cftafterlottitleskip}{1pt}

\renewcommand{\cftafterlottitle}{%
\\[4em]\mbox{}\hspace{2pt}TABLE\hfill{\normalfont Page}\vskip\baselineskip}

\begin{center}
\begin{singlespace}

\setlength{\cftbeforechapskip}{0.4cm}
\setlength{\cftbeforesecskip}{0.30cm}
\setlength{\cftbeforesubsecskip}{0.30cm}
\setlength{\cftbeforefigskip}{0.4cm}
\setlength{\cftbeforetabskip}{0.4cm}

\listoftables 

\end{singlespace}
\end{center}
\endgroup
\pagebreak{}  

\pagenumbering{arabic} 
\setcounter{page}{1}

%
%
%
%
%


\pagestyle{plain} 
\pagenumbering{arabic} 
\setcounter{page}{1}

\chapter{\uppercase {Introduction}}

Optimization algorithms play a crucial role in the process of handling enormous datasets in the age of big data. Traditional optimization methods often struggle with modern applications, sometimes taking days or even weeks to find a solution. This research thus consists in developing and providing fast and resource-efficient, large-scale optimization methods to meet the demands of today's challenges.

It is focused on the theme of \underline{block decomposable} methods, which are methods that not only speed up solution time, sometimes reducing it from days to mere minutes or seconds, but also make it feasible to tackle problems that might previously be considered computationally intractable. In particular, block decomposable methods achieve this goal by breaking the decision variable down into smaller, more manageable pieces, or ``blocks''.

The dissertation further describes in greater detail the two classes of algorithms under this thematic area.

\section{Block Decomposable Methods}

Block decomposable methods update only a small block of variables at each iteration, providing iterates that are cheap in terms of memory and computational costs. These methods are particularly well-suited for high-dimensional problems, including target detection for military applications~\cite{Military}, image reconstruction for medical image analysis \cite{ImageAnalysis}, and 
portfolio management for optimizing investment decisions
\cite{PortfolioOptimization}.

Block decomposable methods are one of a few candidate algorithm classes for solving specially structured composite optimization problems. A composite optimization problem is characterized by the structure of its objective function. The objective function splits into the sum of two components, one of which is smooth and the other of which is potentialy non-smooth, but ``simple". This structure offers several advantages, including the development of specialized methods, like the proximal gradient method, that efficiently handle both the smooth and nonsmooth parts of the problem.

This research provides an in-depth exploration of two algorithm classes: the \textbf{Alternating Direction Method of Multipliers} and \textbf{Block Coordinate Descent} Methods. It begins by highlighting some key practical applications of these methods, followed by a detailed outline of this research’s main contributions.


\noindent
\textbf{Alternating Direction Method of Multipliers (ADMM)}, a block decomposable method, is a powerful  algorithm to solve composite optimization problems. It is employed in many applications such as matrix completion for recommendation systems \cite{applicationMC1}, image processing \cite{applicationCS1} and neural networks \cite{applicationCS1}. ADMM offers several advantages, including its ability to handle linearly constrained problems. These methods can also benefit from parallelization to significantly reduce computation time, as the blocks can be solved concurrently.

The main step of ADMM is the solution of its augmented Lagrangian (AL) subproblem(s). An effective and widely used approach within ADMM is the application of the proximal operator to the AL, which gives rise to the term proximal ADMM. However, obtaining an exact solution to the proximal AL can be very challenging, as many practical applications, such as nonlinear mixed-integer programming, optimal control, and stochastic programming problems, do not offer an easily computable solution.  This challenge was tackled in Chapter~\ref{chapter:admm}, where it is proposed an algorithm 
that is adaptive to all problem parameters and that solves the proximal AL subproblem inexactly. This adaptiveness facilitates the highly efficient application of the algorithm to any problem instance. The inexact solution of the proximal AL subproblem overcomes many key challenges in the practical application of ADMM. The resultant algorithm obtains an approximate solution of an optimization problem in a number of iterations that matches the state-of-the-art complexity for the class of proximal ADMM schemes.


\noindent
\textbf{Block Coordinate Descent (BCD) Methods} are blockwise adaptations of gradient descent, and, more generally, proximal gradient methods. At each iteration, a BCD method selects a block cyclically or randomly, then performs a (proximal) gradient update on it. The use of merely partial gradient information, coupled with the ability to leverage parallelization, make BCD methods powerful tools for modern optimization problems.

The Cyclic Block Proximal Gradient (CBPG) method depends upon the efficient, exact solution of each of its defining block proximal subproblems. However, for some real-world applications, the exact solution of said problems might be computationally intractable. In Chapter~\ref{chapter:jota}, it is expanded the CBPG method to allow for only the inexact solution of these subproblems. The resultant algorithm shares the same convergence rate as its exactly computed counterpart, provided the allowable errors decrease sufficiently quickly or are pre-selected to be sufficiently small. In a more theoretical flavor, Chapter~\ref{chapter:holder} provides the first convergence analysis for a block coordinate method applied to a broad class of objective functions, called \textit{H\"older} smooth functions, where at each iteration the selection of the block is random.

%
%
%
%
%


\pagestyle{plain} 


\chapter{\MakeUppercase{An Adaptive Proximal ADMM for Nonconvex Linearly-Constrained Composite Programs$^*$}}\label{chapter:admm}

\renewcommand{\thefootnote}{\fnsymbol{footnote}} 
\footnotetext[1]{This chapter is based on the preprint \cite{MaiaGM24_arXiv}}
\renewcommand{\thefootnote}{\arabic{footnote}} 

\section{Overview}

   This chapter develops an adaptive Proximal Alternating Direction Method of Multipliers (\textproc{P-ADMM}) for solving linearly-constrained, weakly convex, composite optimization problems. 
   This method is adaptive to all problem parameters, including smoothness and weak convexity constants.
   It is assumed that the smooth component of the objective is weakly convex and possibly nonseparable, while the non-smooth component is convex and block-separable. The proposed method is tolerant to the inexact solution of its block proximal subproblem so it does not require that the non-smooth component has easily computable block proximal maps. Each iteration of our adaptive \textproc{P-ADMM} consists of two steps: 
   (1) the sequential solution of each block proximal subproblem, and (2) adaptive tests to decide whether to perform a full Lagrange multiplier and/or penalty parameter update(s). Without any rank assumptions on the constraint matrices, it is shown that  the adaptive \textproc{P-ADMM} obtains an approximate first-order stationary point of the constrained problem in 
   a number of iterations that matches the state-of-the-art complexity for the class of \textproc{P-ADMM}s. The two proof-of-concept numerical experiments that conclude this chapter suggest our adaptive \textproc{P-ADMM} enjoys significant computational benefits.

\section{Introduction}
This chapter develops an adaptive proximal Alternating Direction Method of Multipliers, called {\DA}, for solving the linearly constrained, smooth, weakly convex, composite optimization problem  

\begin{equation}\label{initial.problem}
    F^* = \min_{y\in\R^n} \left\{ F(y) := f(y)+\sum_{t=1}^B\Psi_t(y_t)\ \ : \ \ \sum_{t=1}^B A_ty_t=b \right\},  
\end{equation}
where $n=n_1+\ldots+n_B$, 
$y=(y_{1},\ldots,y_{B})\in\mathbb{R}^{n_{1}}\times\cdots\times\mathbb{R}^{n_{B}}$,
$b\in \mathbb{R}^l$,
$f:\mathbb{R}^n\to \mathbb{R}$ is a real-valued differentiable function which is $m$-weakly convex,
and $\Psi_t:\mathbb{R}^{n_t}\rightarrow (-\infty, \infty]$ is a proper, closed, convex function which is $M_{\Psi}$-Lipschitz continuous on its  compact domain, for every $t\in\{1,\ldots, B\}$.

To ease notation, for a given $x=(x_{1},\ldots,x_{B})\in\mathbb{R}^{n_{1}}\times\cdots\times\mathbb{R}^{n_{B}}$, let $A(x):=\sum_{t=1}^BA_t(x_t)$ and  $\Psi(x):=\sum_{t=1}^B\Psi_t(x_t)$.
The goal in this chapter is to find a $(\rho,\eta)$-stationary solution of \eqref{initial.problem}, i.e., a quadruple  $(\hat x, \hat p, \hat v,\hat\varepsilon) \in (\dom \Psi) \times A(\mathbb{R}^n) \times \mathbb{R}^l
\times \R_{+}$ satisfying 
\begin{equation}\label{eq:stationarysol}
\hat v \in \nabla f(\hat x) + \partial_{\hat\varepsilon} \Psi(\hat x) + A^*\hat p, \quad \sqrt{\|\hat v\|^2 + \hat\varepsilon}  \leq \rho , \quad \|A\hat x-b\|\leq \eta,
\end{equation}
where $(\rho,\eta)\in \mathbb{R}^2_{++}$ is 
a given tolerance pair.

A popular primal-dual algorithmic framework for solving problem \eqref{initial.problem} that takes advantage of its block structure is the Proximal Alternating Direction Method of Multipliers ({\PD}), which is based on the augmented Lagrangian function,
\begin{equation}\label{DP:AL:F}
{\cal L}_{c}(y;p)  :=F(y)+\left\langle p,Ay-b\right\rangle +\frac{c}{2}\left\Vert Ay-b\right\Vert ^{2},
\end{equation}
where $c>0$ is a penalty parameter. 
Given $(x^{k-1}, p^{k-1},c_{k-1})$, {\PD} finds the next triple $(x^k,p^k,c_k)$ as follows. 
Starting from $x^{k-1}$,
it first performs $\ell_k$ iterations
of a block inexact proximal
point (IPP) method  applied to ${\cal L}_{c_{k-1}}(\cdot\,;\,p^{k-1})$ to obtain $x_k$, where $\ell_k$ is a positive integer. 
Next, it performs a Lagrange multiplier update according to 
\begin{equation}
    p^{k} = (1-\theta)\Big[p^{k-1}+\chi c_{k-1}\left(Ax^k-b\right)\Big],\label{eq:p_update}
\end{equation}
where $\theta\in [0,1)$  is a dampening parameter and $\chi$ is a positive relaxation parameter, and chooses a scalar $c_k \ge c_{k-1}$ as the next penalty parameter.

We now formally describe how $x^k$ is obtained from $x^{k-1}$.
We set $z^0 = x^{k-1}$ and, for each $j=1,\ldots,\ell_k$, perform a block IPP iteration from $z^{j-1}$ to obtain $z^j$. After $\ell_k$ iterations, we set $x^k = z^{\ell_k}$.
Each block IPP iteration updates $z^{j-1}$ to $z^j$ by sequentially, for $t=1,\ldots,B$, inexactly solving the $t$-th block proximal augmented Lagrangian subproblem with prox-stepsize $\lambda_t$:
\begin{align}\label{eq:x_update}
z_t^{j} \approx  \operatorname*{argmin}_{u_t\in\R^{n_{t}}}\left\{ \lam_t{\cal L}_{c_{k-1}}(z_1^{j},\ldots,z_{t-1}^j,u_t,z_{t+1}^{j-1},\ldots,z_B^{j-1};p^{k-1})+\frac{1}{2}\|u_t-z^{j-1}_t\|^{2}\right\} ,
\end{align} 
where $\lambda=(\lam_1,\cdots,\lam_B)$ is the block prox stepsizes.

The recent publication  \cite{KongMonteiro2024} 
proposes a version of 
{\PD} for solving \eqref{initial.problem} which assumes that $\ell_k=1$,
$\lam_1=\cdots=\lam_B$, and
$(\chi, \theta) \in (0,1]^2$ satisfies
\begin{equation}\label{eq:assumptio:B}
2\chi B(2-\theta)(1-\theta)\le \theta^2.
\end{equation} 

One of the main contributions of \cite{KongMonteiro2024} is that its convergence guarantees do not require \textit{the last block condition}, 
${\rm{Im}}(A_B)\supseteq \{b\}\cup{\rm{Im}}(A_1)\cup\ldots\cup{\rm{Im}}(A_{B-1})$ and $h_B\equiv 0$, 
that hinders many instances of {\PD}, see \cite{chao2020inertial,goncalves2017convergence,themelis2020douglas,zhang2020proximal}.
However, the main drawbacks of the {\PD} of \cite{KongMonteiro2024} include: 
(i) the strong assumption \eqref{eq:assumptio:B} 
on $(\chi, \theta)$; (ii) subproblem \eqref{eq:x_update}
must be solved exactly;
(iii) the stepsize parameter $\lam$ is conservative and requires the knowledge of $f$'s weak convexity parameter;
(iv) it (conservatively) updates the Lagrange multiplier after each primal update cycle (i.e., $\ell_k=1$);
(v) its iteration-complexity has a high dependence on the number of blocks $B$, namely,
${\cal O}(B^8)$.
Paper \cite{KongMonteiro2024} also presents computational results comparing its {\PD} with a more practical variant where $(\theta,\chi)$, instead of satisfying  \eqref{eq:assumptio:B}, is set to $(0,1)$. Intriguingly, this $(\theta,\chi)=(0,1)$ regime substantially outperforms the theoretical regime of \eqref{eq:assumptio:B} in the provided computational experiments.  
No convergence analysis for the $(\theta,\chi)=(0,1)$ regime is forwarded in \cite{KongMonteiro2024}. 
Thus, \cite{KongMonteiro2024} leaves open the tantalizing question of whether the convergence of {\PD} with $(\theta,\chi)=(0,1)$ can be theoretically secured.\\

\noindent{\bf Contributions:}
This chapter partially addresses the convergence analysis issue raised above by studying
a {\it completely parameter-free} {\PD}, with $(\theta,\chi)=(0,1)$
and $\ell_k$ adaptively chosen, called {\DA}.  
Rather than making the conservative determination that $\ell_k=1$, the studied adaptive method ensures the dual updates are committed as frequently as possible.
It is shown that {\DA} finds a $(\rho,\eta)$-stationary solution
in 
${\cal O}(B\max\{\rho^{-3}, \eta^{-3}\})$ iterations. {\DA} also exhibits the following additional features:

\begin{itemize}
\setlength{\itemsep}{0pt}

    \item 
    Similar to the {\PD} of \cite{KongMonteiro2024}, its complexity is established without assuming that the {\it last block condition} holds;

    \item Compared to the 
   ${\cal O}(B^8\max\{\rho^{-3},\eta^{-3}\})$ iteration-complexity of the {\PD} of \cite{KongMonteiro2024},
    the one for  {\DA} vastly  {\it improves the dependence on $B$};

    \item
    {\DA}  uses an adaptive scheme that aggressively computes {\it variable block prox stepsizes}, instead of fixed ones that require knowledge of the weak convexity parameters $m_1,\ldots,m_B$
    (e.g., the choice $\lam_1=\ldots=\lam_B \in (0, 1/(2m))$ where $m
    := \max\{m_1,\ldots,m_B\}$
    made by the {\PD} of \cite{KongMonteiro2024}). In contrast to the {\PD} of \cite{KongMonteiro2024}, {\DA} may generate various $\lam_t $'s which are larger than $1/m_t$
    (as observed in our computational results), and hence
    which do not guarantee convexity of \eqref{eq:x_update}.
    
    \item
    {\DA} is also adaptive to Lipschitz parameters;

    \item 
    In contrast to the {\PD} in \cite{KongMonteiro2024},
    {\DA} allows the block proximal subproblems \eqref{eq:x_update} to be either exactly or {\it inexactly} solved.
\end{itemize}

\noindent \textbf{Related Works}:  
ADMM methods with $B=1$ are well-known to be equivalent to
augmented Lagrangian methods. Several references have studied augmented Lagrangian and proximal augmented Lagrangian methods in the convex (see e.g., \cite{Aybatpenalty,AybatAugLag,LanRen2013PenMet,LanMonteiroAugLag,ShiqiaMaAugLag16, zhaosongAugLag18,IterComplConicprog,Patrascu2017, YangyangAugLag17}) and nonconvex 
(see e.g. \cite{bertsekas2016nonlinear, birgin2,solodovGlobConv,HongPertAugLag, AIDAL, RJWIPAAL2020, NL-IPAL, RenWilmelo2020iteration, YinMoreau, ErrorBoundJzhang-ZQLuo2020, ADMMJzhang-ZQLuo2020,sun2021dual}) settings.
Moreover, ADMMs and proximal ADMMs in the convex setting have also been broadly studied in the literature (see e.g. \cite{bertsekas2016nonlinear,boyd2011distributed,eckstein1992douglas,eckstein1998operator,eckstein1994some,eckstein2008family,eckstein2009general,gabay1983chapter,gabay1976dual,glowinski1975approximation,monteiro2013iteration,rockafellar1976augmented,ruszczynski1989augmented}).
So from now on, we
just discuss {\PD} variants where $f$ is nonconvex and $B > 1$.

A discussion of the existent literature on nonconvex {\PD} is best framed by dividing it into two different corpora: those papers that assume the last block condition and those that do not. Under the \textit{last block condition}, the iteration-complexity established is ${\cal O}(\varepsilon^{-2})$, where $\varepsilon := \min\{\rho,\eta\}$.  Specifically, \cite{chao2020inertial,goncalves2017convergence,themelis2020douglas, wang2019global} 
introduce {\PD} approaches assuming $B=2$, while \cite{jia2021incremental, jiang2019structured,melo2017iterationJacobi,melo2017iterationGauss} present (possibly linearized) {\PD}s assuming $B\geq 2$. 
A first step towards removing the last block condition was made by \cite{jiang2019structured} which
proposes an ADMM-type method applied to a penalty reformulation of \eqref{initial.problem} that artificially satisfies the last block condition. This method possesses an ${\cal O}(\varepsilon^{-6})$ iteration-complexity bound.

On the other hand, development of ADMM-type methods directly applicable to \eqref{initial.problem} is considerably more challenging and only a few works addressing this topic have surfaced. In addition to \cite{jiang2019structured}, earlier contributions to this topic were obtained in  \cite{hong2016convergence,sun2021dual,zhang2020proximal}. More specifically,
\cite{hong2016convergence,zhang2020proximal} develop a novel small stepsize ADMM-type method without establishing its complexity. Finally, \cite{sun2021dual} considers an interesting but unorthodox negative stepsize for its Lagrange multiplier update, that sets it outside the ADMM paradigm, and thus justifies its qualified moniker, ``scaled dual descent ADMM''.

\subsection{Chapter's Organization}

In this subsection, we outline this article's structure. This section's lone remaining subsection, Subsection~\ref{subsec:notation}, briefly lays out the basic definitions and notation used throughout. Section \ref{subsec:subpro:e} 
decribes assumptions for Problem~\eqref{initial.problem} and introduces a notion of an inexact solution of {\DA}'s foundational block proximal subproblem \eqref{eq:x_update} along with  efficient subroutines designed to find said solutions. Section~\ref{adp:ADMM} presents the static version of the main algorithm of this chapter, {\SA}, and states the theorem governing its iteration complexity (Theorem~\ref{thm:static.complexity}). Section~\ref{subsubsec:subroutine} provides the detailed proof of the iteration-complexity theorem for {\SA} and presents all supporting technical lemmas. Section~\ref{subsection:DA:ADMM} introduces the centerpiece algorithm of this work, a {\PD} method with fixed stepsizes, namely {\DA}. It also states the main theorem of this chapter (Theorem~\ref{the:dynamic}), which establishes the iteration complexity of the method. Section~\ref{section:Adaptive.ADMM} extends {\DA} to an adaptive stepsize version and briefly describes how to obtain a completely adaptive method. Section~\ref{sec:numerical} presents proof-of-concept numerical experiments that display the superb efficiency of {\DA} for two different problem classes. Section \ref{sec:concluding} gives some concluding remarks that suggest further research directions. Finally, Section~\ref{sec:tech.lagrange.multiplier} collects some technical results on convexity and linear algebra.

\subsection{Notation and Basic Definitions}\label{subsec:notation}

This subsection lists the elementary notation deployed throughout the chapter. We let $\R_{+}$  and $\R_{++}$ denote the set of non-negative and positive real numbers, respectively. 
We shall assume that the $n$-dimensional Euclidean space, $\R^n$, is equipped with an inner product, $\inner{\cdot}{\cdot}$. 
We use  $\R^{l\times n}$ to denote the set of all $l\times n$ matrices and  ${\cal S}^n_{++}$ to denote the set of positive definite matrices in $\R^{n\times n}$.
The associated induced norm will be written as $\|\cdot\|$. 

When the Euclidean space of interest has the block structure $\mathbb{R}^{n_{1}}\times\cdots\times\mathbb{R}^{n_{B}}$, we will often consider, for $x=(x_{1},\ldots,x_{ B}) \in\mathbb{R}^{n_{1}}\times\cdots\times\mathbb{R}^{n_{B}}$, the aggregated quantities 
\begin{gather}
\begin{gathered}
x_{<t}:=(x_{1},\ldots,x_{t-1}), \quad  x_{>t}:=(x_{t+1},\ldots,x_{ B}),\quad x_{\leq t}:=(x_{<t},x_{t}),\quad x_{\geq t}:=(x_{t},x_{>t}).
\end{gathered}
\label{eq:intro_agg}
\end{gather}
For a given closed, convex set $Z \subset \R^n$, we let $\partial Z$ designate its boundary. The distance of a point $z \in \R^n$ to $Z$, measured in terms of $\|\cdot\|$, is denoted ${\rm dist}(z,Z)$. 
The indicator function of $Z$, $\delta_Z$, is defined by $\delta_Z(z)=0$ if $z\in Z$, and $\delta_Z(z)=+\infty$ otherwise.

The set of points where a function  $g :\R^n\to (-\infty,\infty]$ is finite-valued, $\dom g := \{x\in \R^n : g(x) < +\infty\}$, is called its domain.
We say that $g$ is proper if $\dom g \ne \emptyset$.
The set of all lower semi-continuous proper convex
functions defined in $\R^n$ is denoted by $\bConv{n}$.
 For $\epsilon\geq 0$, the $\epsilon$-subdifferential of $g\in \bConv{n}$ at $z\in \dom g$ is denoted by 
\[
\partial_\epsilon g(z):=\{w\in \R^n ~:~ g(z')\geq g(z)+\inner{w}{z'-z}-\varepsilon, \forall z'\in \R^n\}.
\]

When $\epsilon=0$, the $\epsilon$-subdifferential recovers the classical subdifferential, $\partial g(\cdot):=\partial_0 g(\cdot)$. If $\psi$ is a real-valued function that is differentiable at $\bar z \in \R^n$, then its affine approximation $\ell_\psi(\cdot,\bar z)$ at $\bar z$ is the function defined, for arbitrary $z\in\R^n$, by the rule $z\mapsto \psi(\bar z) + \inner{\nabla \psi(\bar z)}{z-\bar z}$. The smallest positive singular value of a  nonzero linear operator $Q:\R^n\to \R^l$ is denoted $\nu^+_Q$ and its operator norm is $\|Q\|:=\sup \{\|Q(w)\|:\|w\|=1\}.$  

We conclude this subsection by presenting two well-known properties about the $\epsilon$-subdifferential: for any $\beta>0$, $\varepsilon\ge 0$ and $g\in \bConv{n}$, it holds that
\begin{align}\label{prop1:subd}
\partial_{\varepsilon} \beta g(\cdot ) = \beta \partial_{\varepsilon/\beta} g(\cdot ),
\end{align}
and if $
\mathbb{R}^{n_{1}}\times\cdots\times\mathbb{R}^{n_{B}}\ni (y_1,\cdots, y_B)=y\mapsto g(y)=g_1(y_1)+\cdots+g_B(y_B)
$
and $\varepsilon_t\ge 0$ for any $t \in \{1,\ldots,B\}$, 
then
\begin{align}\label{prop2:subd}
\partial_{\varepsilon}g(y) =\cup\left\{  \partial_{\varepsilon_1}g_1(y_1) \times \ldots \times
\partial_{\varepsilon_B}g_B(y_B) : \varepsilon\ge 0, \;\; \varepsilon_1+\cdots+\varepsilon_B\le \varepsilon\right\}.
\end{align}
For more details, see \cite[Proposition 1.3.1]{lemarechal1993} and \cite[Remark 3.1.5]{lemarechal1993}.

\section{Assumptions for Problem~\eqref{initial.problem} and  an Inexact Solution Concept of \eqref{eq:x_update}}\label{subsec:subpro:e}

This section constains two subsections. Subsection \ref{subsec:assump} details a few mild technical assumptions imposed on \eqref{initial.problem}. Subsection \ref{subsec:inexact-solution} introduces a notion of an inexact stationary point for the block proximal subproblem \eqref{eq:x_update} along with an efficient method for discovering such points. 
This method permits the application of our main algorithm, {\DA}, even when \eqref{eq:x_update} is not exactly solvable.

\subsection{ Assumptions for Problem~\eqref{initial.problem}}\label{subsec:assump}
This subsection describes a series of mild assumptions on this chapter's main problem of interest \eqref{initial.problem}. We assume that vector $b\in \mathbb{R}^l$,
linear operator
$A:\mathbb{R}^n\rightarrow \mathbb{R}^l$, and functions $f:\mathbb{R}^{n}\rightarrow (-\infty, \infty]$ and  $\Psi_t:\mathbb{R}^{n_t}\rightarrow (-\infty, \infty]$ for $t=1,\ldots, B$, satisfy the following conditions:

\begin{enumerate}
\item[(A1)] 
for every $t\in \{1,\ldots,B\}$,
function
$\Psi_t:\R^{n_t}\to(-\infty,\infty]$ is proper closed convex with ${\cal H}_t:=\dom \Psi_t$ compact; also, $\Psi_t(\cdot )$ is prox friendly, i.e., its proximal operator is easily computable;

\item[(A2)] $A$ is nonzero and ${\cal F}:=\{x\in {\cal H}:Ax=b\}\neq \emptyset$, where ${\cal H}:= {\cal H}_1\times\cdots\times {\cal H}_B$;

\item[(A3)] $f$ is block $m$-weakly convex for $m=(m_1,\ldots,m_B)\in\R^B_{++}$, that is, for every $t\in \{1,\ldots,B\}$,
\[
f(x_{<t},\cdot,x_{>t}) + \delta_{{\cal H}_t}(\cdot) + \frac{m_{t}}{2}\|\cdot\|^{2}\text{ is convex for all } x\in{\cal H};
\]
    \item[(A4)] $f$ is differentiable on ${\cal H}$ and, for every $t\in \{1,\ldots,B-1\}$, there exists $L_{t}\ge 0$ such that
\begin{gather}
\begin{aligned}
& \|\nabla_{x_{t}}f(x_{\le t},\tilde{x}_{>t})-\nabla_{x_{t}}f(x_{\le t},{x}_{>t})\|
\leq L_t\|\tilde{x}_{>t}-x_{>t}\| \quad  \forall~ x,\tilde{x}\in{\cal H};
\end{aligned}
\label{eq:lipschitz_x}
\end{gather}

    \item[(A5)] for some $M_{\Psi}\ge 0$, function $\Psi(\cdot)$ is $M_{\Psi}$-Lipschitz continuous on ${\cal H};$
    
    \item[(A6)]  there exists $\bar{\mathrm{x}}\in {\cal F}$ such that $\bar{d}:=\text{dist}(\bar{\mathrm{x}}, \partial {\cal H})>0$.
\end{enumerate}

Note that since ${\cal H}$ is compact, it follows from (A1) and (A2) that the scalars
\begin{equation}\label{def:damH}
       D_{\Psi}:=\sup_{z \in {\cal H}} \|z-\bar{x}\|, \quad \nabla_f:=\sup_{u \in {\cal H}} \| \nabla f(u)\|, \quad F_{\inf} := \inf_{u \in {\cal H}} F(u), \quad F_{\sup} := \sup_{u \in {\cal H}}  F(u)  
\end{equation}
are bounded. Furthermore, throughout this chapter, we let
\begin{equation}\label{eq:block_norm}
\|L\|^2:=\sum_{t=1}^{B-1} \|L_t\|^2, \quad \|A\|_\dagger^2 := \sum_{t=1}^B \|A_t\|^2.
\end{equation}

\subsection{An Inexact Solution Concept for \eqref{eq:x_update}}\label{subsec:inexact-solution}

This subsection introduces our notion (Definition \ref{def:BB}) of an inexact solution of the block proximal AL subproblem \eqref{eq:x_update}. To cleanly frame this solution concept, observe that \eqref{eq:x_update} can be cast in the form 
\begin{equation}\label{ISO:problem}
\psi^*=\min \{\psi(z):=\psi_{\text{s}}(z)+\psi_{\text{n}}(z) : z\in\R^n\},
\end{equation}
where
\begin{equation}\label{def:ps}
\psi_{\text{s}}(\cdot) =\lam \hat {\cal L}_{c}(y_{<t}^{i},\cdot,y_{>t}^{i-1};\tilde q^{k-1})+\frac{1}{2}\|\cdot-y_{t}^{i-1}\|^{2}, \quad \psi_{\text{n}}(\cdot) = \lam \Psi_t(\cdot),
\end{equation}
and  ${\cal \hat L}_{c}(\cdot\,;\tilde q^{k-1})$ is the smooth part of \eqref{DP:AL:F}, defined as
\begin{align}\label{def:smooth:ALM}
{\cal \hat L}_{c}(y;\tilde q^{k-1}):= f(y)+\left\langle \tilde q^{k-1},Ay-b\right\rangle +\frac{c}{2}\left\Vert Ay-b\right\Vert ^{2}.
\end{align}
Hence, to describe a notion of an inexact solution for \eqref{eq:x_update}, it is suffices to describe it in the more general context of \eqref{ISO:problem}
where $\psi_{\text{s}}:\R^n \to \R$ is a differentiable function and
$\psi_{\text{n}}: \R^n\to (-\infty,+\infty]$ is a proper, closed, convex function.

\begin{definition}\label{def:BB}
     For a given $z^0 \in \dom \psi_{\mathrm{n}}$ and parameter
$\tau \in \R_+$,
a triple $(\bar z, \bar r, \bar \epsilon)\in \R^n\times\R^n\times\R_{+}$ satisfying
\begin{align}
\bar r \in \nabla \psi_{\mathrm{s}}(\bar z) + \partial_{\bar\varepsilon} \psi_{\mathrm{n}}(\bar z) \quad \text{and}\quad
\|\bar r\|^2 + 2 \bar\varepsilon \le 
\tau \|z^0-\bar z\|^2 \label{ISO:Cond:1&2}  
\end{align}
is called a
$(\tau;z^0)$-stationary solution of \eqref{ISO:problem} with composite term $\psi_n$.
\end{definition}
\noindent We now make some remarks about  Definition \ref{def:BB}. First, if $\tau=0$, then \eqref{ISO:Cond:1&2} implies that $(\bar r,\bar\varepsilon)=(0,0)$ and \eqref{ISO:Cond:1&2} then implies that $\bar z$ is an exact stationary point of \eqref{ISO:problem}, in the sense that $ 0\in \nabla \psi_{\text{s}}(\bar z) + \partial \psi_{\text{n}}(\bar z)$. Thus, if the triple $(\bar z, \bar r, \bar\varepsilon)$ is a $(\tau; z^0)$-stationary solution of \eqref{ISO:problem}, then $\bar z$ can be viewed as an approximate stationary solution of \eqref{ISO:problem} where the residual pair $(\bar r, \bar\varepsilon)$
is bounded according to \eqref{ISO:Cond:1&2} (instead of being zero as in the exact case). 
Second, if $\bar z$ is an exact stationary point of \eqref{ISO:problem}
(e.g., $\bar z$ is an
exact solution of
\eqref{ISO:problem}), then the triple $(\bar z,0,0)$ is a $(\tau; z^0)$-stationary point of \eqref{ISO:problem} for any $\tau \in \R_+$. 

In general, an exact solution or  stationary point of \eqref{ISO:problem} is not easy to compute. In such a case, 
the following well-known result 
(see, e.g., Proposition~3.5 of \cite{HeMonteiro2015}) states that a variant of an accelerate composite gradient method, due to Nesterov \cite{nesterov2013gradient}, is able to find a $(\tau;z^0)$-stationary solution of \eqref{ISO:problem}.

\begin{prop}\label{pro.inexact.sol}
Assume that for some $(\tilde M,\tilde \mu) \in \R^2_{++}$, the above functional pair $(\psi_{\text{s}},\psi_{\text{n}})$ is such that $\nabla \psi_{\text{s}}(\cdot)$ is $\tilde M$-Lipschitz 
continuous and $(\psi_s+\psi_n)$ is $\tilde \mu$-strongly convex.
Then, Algorithm 1 in \cite{HeMonteiro2015}
finds a  $(\tau;z^0)$-stationary solution of \eqref{ISO:problem} in a number of iterations bounded by
\[
{\cal O} \left( 
\left( 1+\sqrt{ \frac{\tilde M}{\tilde \mu} }\right)
\left\{  1 + \log\left[ 1+ \tilde M( \tilde M+1)\lceil\tau^{-1}\rceil\right] \right\} \right).
\] 
\end{prop}

\section{A Static {\PD}}\label{adp:ADMM}

This section is divided into two subsections. Subsection~\ref{sub:B-IPP} introduces an important component of the ADMM method: a subroutine called {\SBIPP}, which performs a block IPP iteration within ADMM, as discussed in the paragraph containing equations \eqref{DP:AL:F} and \eqref{eq:p_update}. Subsection~\ref{subsection:SA-ADMM} presents an ADMM method with a constant penalty parameter for solving problem~\eqref{initial.problem}, referred to as the static ADMM ({\SA} for short).

\subsection{ A Block IPP Black-Box: {\SBIPP}}\label{sub:B-IPP}


The goal of this subsection is to present the {\SBIPP} subroutine, state its main properties, and provide relevant remarks. We begin by describing {\SBIPP}.

{\floatname{algorithm}{Subroutine}
\renewcommand{\thealgorithm}{}
\begin{algorithm}[H]
\setstretch{1}
\caption{{\SBIPP}}\label{algo:BlockIPP}
\begin{algorithmic}[1]
\Statex \hskip-1.8em
\Require $(z,p,\lam, c)
\in {\cal H} \times A(\mathbb{R}^n) 
\times \R^B_{++}\times \R_{++} $
\Ensure $(z^+,v^+,\delta_+,\lam^+)\in {\cal H} \times \mathbb{R}^l 
\times \R_{++} \times \R^B_{++}$
\vspace{1em}
\\ 
\textcolor{violet}{STEP 1: Block-IPP Iteration}
  \For{$t=1,\ldots, B$}\label{subroutine:loop.inexact.sol}
    
    \State set $\lam_t^+= \lam_t$
    \State compute a $\left(1/8;z_t\right)$-stationary solution $(z_t^+,r_t^+,\varepsilon_t^{+})$
         of \label{algo:BIPP.inexact.subp}
    \begin{equation}\label{algo:BIPP.inexact.subproblem}
    \min_{u \in \R^{n_t}} \left  \{ \lam_t^+ \hat {\cal L}_{c}(z_{<t}^+,u,z_{>t};p)+\frac{1}{2}\|u-z_{t}\|^{2}+\lam_t^+ \Psi_t(u) \right\} 
    \end{equation}
    \hspace{1.3em} with composite term $\lam_t^+ \Psi_t(\cdot)$ (see Definition \ref{def:BB})

  \EndFor
  \State $z^+\gets (z_1^+,\ldots ,z_B^+)$ and
  $\lam^+\gets (\lam_1^+,\ldots ,\lam_B^+)$\label{subroutine.primal}

\vspace{1em}     
    \\
    \textcolor{violet}{STEP 2: Computation of the residual pair $(v^+,\delta_+)$ for $(z^+,p)$}

    \For{$t=1,\ldots ,B$}
    \State $ v_t^+\gets \nabla_t f(z_{< t}^+,z_t^+, z_{>t}^+)- \nabla_t f(z_{< t}^+,z_t^+, z_{>t}) + \dfrac{r_t^+}{\lambda_t^+}+ cA_{t}^{*}\sum_{s=t+1}^{ B}A_{s} (z^+_{s}-z_s) -\dfrac{1}{\lambda_t^+}( z^+_{t}-z_t)$ \label{algo.stat.vit}\;
    \EndFor
    \State $v^+\gets (v_1^+,\ldots,v_B^+)$ and $\delta_+\gets (\varepsilon_1^+/\lam_1^+) + \ldots +  (\varepsilon_B^+/\lam_B^+)$ \label{algo:IPP.v.delta}
    \State \Return $(z^+,v^+,\delta_+, \lam^+)$
   
\end{algorithmic}
\end{algorithm}
\renewcommand{\thealgorithm}{\arabic{algorithm}}
}

We now clarify some aspects of {\SBIPP}.
First, line~\ref{algo:BIPP.inexact.subp} requires a subroutine to find an approximate solution of \eqref{algo:BIPP.inexact.subproblem} as in Definition \ref{def:BB}. A detailed discussion giving examples of such subroutine will be  given  in the second paragraph after Proposition~\ref{main:subroutine:1}. Second, Proposition~\ref{main:subroutine:1} shows that the iterate $z^+$ and the residual pair $(v^+,\delta_+)$ computed in
lines \ref{subroutine.primal} and \ref{algo:IPP.v.delta}, respectively, satisfy the approximate stationary inclusion $v^+ \in \nabla f(z^+) + \partial_{\delta_+} \Psi(z^+) + \textrm{Im}(A^*)$. Hence, upon termination of {\SBIPP}, its output satisfies the first condition in \eqref{eq:stationarysol} (though it may not necessarily fulfill the remaining two) and
establishes an important bound on
the residual pair $(v^+,\delta_+)$ in terms of a Lagrangian function variation
that will be used later to determine a suitable potential function.

We now make some remarks about the prox stepsizes. First, {\SBIPP} does not change the prox stepsize and hence, in principle, both $\lam$ and $\lam^+$ could be removed from its input and output, respectively.
ADMMs based on {\SBIPP} will result in (constant stepsize) ADMM variants that keep their prox stepsize constant throughout.
In Section~\ref{section:Adaptive.ADMM}, we will consider adaptive stepsize ADMM variants based on a adaptive version of {\SBIPP}.
The reason for
including $\lam$ and $\lam^+$ on the input and output of the constant stepsize  {\SBIPP}, and its corresponding ADDMs, at this early stage is
to facilitate
the descriptions of 
their adaptive counterparts, which will essentially involve a minimal but important change to line~\ref{algo:BIPP.inexact.subp} of {\SBIPP}.



\vspace{1em}
Before stating the main result of this subsection, we define some quantities that are used in its statement, namely:
\begin{gather}
\begin{gathered}
\label{eq:sigma1&2}
\chione := 8(25\overline{m} +6\|L\|^2\underline{m}^{-1})+1, \qquad \chitwo:=24B\|A\|^2_{\dagger}
\end{gathered}
\end{gather}
where
\begin{equation}\label{eq:max.m.min.m}
\underline{m}:=\min_{1\leq t\leq B}m_t, \qquad \overline{m}:=\max_{1\leq t\leq B}m_t,
\end{equation}
and $\|L\|$ and $\|A\|_{\dagger}$ are defined in \eqref{eq:block_norm}.

\begin{proposition}\label{main:subroutine:1}
Assume that $(z^+,v^+,\delta_+, \lam^+)={\SBIPP}(z,p,\lam, c)$ for some $(z,p,\lam, c)
\in {\cal H} \times A(\mathbb{R}^n) 
\times \R_{++}^B\times \R_{++}$.
Then,
\begin{gather}\label{prop.B-IPP.inclusion}
\begin{gathered}
v^+  \in \nabla f (z^+)+ \partial_{\delta_+} \Psi(z^+) +
A^*[p+c(Az^+-b)].
\end{gathered}
\end{gather}
Moreover, if the initial prox stepsize is chosen as
\begin{equation}\label{choo:lamt}
\lam_t = \frac{1}{2m_t} \quad \forall t \in \{1,\ldots,B\},
\end{equation}
then:
\begin{itemize}
    \item[(a)] for any $t\in\{1,\ldots,B\}$, the objective function of the $t$-th block subproblem~\eqref{algo:BIPP.inexact.subproblem} is $(1/2)$-strongly convex;
    \item[(b)] it holds that
\begin{gather}\label{lemma:norm:residual:Ine}
\begin{gathered}
\quad \|v^+\|^2 +\delta_+  
\le (\sigma_1 + c \sigma_2)
  \Big[ {\cal L}_c(z;p) - {\cal L}_c(z^+;p)\Big],
\end{gathered}
\end{gather}
where $\chione$ and  $\chitwo$ are as in  \eqref{eq:sigma1&2}.
\end{itemize}
\end{proposition}
\begin{proof}
   At the end of Section~\ref{section:Adaptive.ADMM}.
\end{proof}

\vspace{1em}

We now make some remarks about Proposition \ref{main:subroutine:1}.
First, the inclusion~\eqref{prop.B-IPP.inclusion} shows that $(v^+, \delta_+)$ is a residual pair for the point $z^+$. Second, the inequality in \eqref{lemma:norm:residual:Ine} provides a bound on the magnitude of the residual pair $(v^+, \delta_+)$ in terms of a variation of the Lagrangian function
which, in the analysis of the next section, will play the role of a potential function.
Third, the prox stepsize selection in \eqref{choo:lamt} is so as to guarantee that \eqref{lemma:norm:residual:Ine} holds and will play no further role in
the analyses of ADMM's presented in the subsequent sections. Fourth, adaptive ways of choosing the prox stepsizes will be discussed on Subsection \ref{section:Adaptive.ADMM} which will yield ADMM's that do not require knowledge of the parameters $m_t$'s. They are designed so as  to guarantee that a slightly modified version of the inequality in \eqref{lemma:norm:residual:Ine} holds
(i.e., with a different choice of constant $\sigma_1$) and hence enable the arguments and proofs of the subsequent sections to follow through in a similar fashion for ADMMs based on an adaptive stepsize version of {\SBIPP}.

We now comment on the possible implementations of 
line~\ref{algo:BIPP.inexact.subp} of {\SBIPP}.
As already observed in the paragraph following Definition \ref{def:BB},
if an exact solution $z^+_t$ of \eqref{algo:BIPP.inexact.subproblem}
can be computed in closed form, then
$(z_t^+,v_t^+,\varepsilon_t^+)=(z^+_t,0,0)$ fulfills the
requirements of line~\ref{algo:BIPP.inexact.subp}.
Otherwise, assume that $\nabla_{x_t} f (x_1,\ldots,x_B)$ exists for every $x=(x_1,\ldots,x_B) \in {\cal H}$ and is
$\tilde L_{t}$-Lipschitz continuous with respect to the $t$-th block $x_t$.
Then, using this assumption and  Proposition~\ref{main:subroutine:1}(a),
it follows from Proposition~\ref{pro.inexact.sol} with 
$\tilde \mu=1/2$ and $\tilde M =\lam_t(\tilde L_f + c \|A_t\|^2 )
$ that
the accelerate composite gradient variant of \cite{HeMonteiro2015} (see Algorithm 1 thereof) obtains a triple $(z_t^+,v_t^+,\varepsilon_t^{+})$
fulfilling the requirements of  line~\ref{algo:BIPP.inexact.subp} 
in a number of accelerate composite gradient iterations bounded (up to logarithmic terms) by
${\cal O} ( [\lam_t(\tilde L_f + c \|A_t\|^2 )]^{1/2})$  with $\lam_t=1/(2m_t)$.

Finally, as already observed before, {\SBIPP} is a key component that is invoked once in every iteration of the ADMM's presented in subsequent sections. The complexity bounds 
for these ADMMs will be given in terms of ADMM iterations (and hence {\SBIPP} calls)
and will not take into account the complexities of implementing line~\ref{algo:BIPP.inexact.subp}.
The main reason for doing so is due to the different possible ways of solving the block subproblems (e.g., in closed form, by an accelerate composite gradient variant, or by some convex optimization other solver).
Nevertheless, the discussion in the previous paragraph provides ways of estimating the contribution of each block to the overall algorithmic effort.


\vspace{1em}

\subsection{{\SA}: A Static Version of {\DA}}\label{subsection:SA-ADMM}

This subsection presents {\SA}, a static version of our {\DA}, and its main complexity result (Theorem \ref{thm:static.complexity}). The qualifier ``static" is attached because this variant keeps the penalty parameter constant throughout its course. The proof of Theorem \ref{thm:static.complexity} is the focus of Section \ref{subsubsec:subroutine}. We start by elaborating {\SA}.

\setcounter{algorithm}{0}
\begin{algorithm}[H]
\setstretch{1}
\caption{{\SA}}\label{alg:static}
\begin{algorithmic}[1]
\Statex \hskip-1.8em
\textbf{Universal Input:}  $\rho>0$, $\alpha\in [\rho^2, +\infty)$, $C \in [\rho,+\infty)$

\Require $(y^0,q^0,\lam^0, c)
\in {\cal H} \times A(\mathbb{R}^n) 
\times \R_{++}^B\times \R_{++}$
\Ensure $(\hat y, \hat q, \hat v, \hat \delta , \hat \lam)$
\vspace{1em}

\State $T_0=0$, $k=0$ \label{def:T0} 
\For{$i\gets 1,2,\ldots$} \label{star:cycle}

   \Statex \hspace{1.3em}\textcolor{violet}{STEP 1: Block-IPP call} 
   \State $(y^i,v^i,\delta_i,\lam^i)= \SBIPP(y^{i-1},q^{i-1},\lam^{i-1},c)$\label{call:SBIPP}
   \If{$\|v^i\|^2 +\delta_i  \le \rho^2$}\label{cond:vi:stop} \Comment{termination criteria}
    \State $k\gets k+1$, \ $j_k\gets i$\Comment{end of the last epoch}\label{algo:j_k&k.update:1}
        \State  
        $q^i =  q^{i-1} + c (A y^i - b) $  
    \State \Return 
        $(\hat y, \hat q, \hat v, \hat \delta , \hat \lam)=(y^i, q^i, v^i, \delta_i, \lam^i)$\label{algo.sa.return}

    \EndIf

\vspace{1em}
\Statex \hspace{1.3em}\textcolor{violet}{STEP 2: Test for Multiplier Update}
    \State $T_i= {\cal L}_c( y^{i-1};q^{i-1}) - {\cal L}_c(y^i;q^{i-1}) + T_{i-1}$\label{def:T:i}
    
    \If{$\|v^i\|^2 +\delta_i\le C^2$ and $\dfrac{\rho^2}{\alpha(k+1)} \ge \dfrac{T_i}{i}  $}\label{stat:begin.test.cond}
    \State $k\gets k+1$, \ $j_k\gets i$ \Comment{end of epoch $\Ik$} \label{algo:j_k&k.update}
        \State  
        $q^{i} =  q^{i-1} + c (A y^i - b) $  \label{def:qi:chi}
 \label{stat:test.i.qi.ti.updates} \Comment{Multiplier update}
\label{line:update:i}

    \Else
    \State $q^i= q^{i-1}$ \label{sa.qi.update}
   
\EndIf
\EndFor
\end{algorithmic}
\end{algorithm}

We now make comments about {\SA}. The iteration index $i$ counts the number of iterations
of {\SA}, referred to as {\SA} iterations throughout the chapter.
Index $k$ counts the number of
Lagrange multiplier updates performed by  {\SA}.
The index $j_k$ computed either in lines \ref{algo:j_k&k.update:1} or \ref{algo:j_k&k.update} of {\SA} is the iteration index where the $k$-th Lagrange multiplier occurs. It is shown in Theorem~\ref{thm:static.complexity}(a) that the total number of iterations performed by {\SA} is finite, and hence that the index $j_k$ is well-defined.
If the inequality in line~\ref{cond:vi:stop} is satisfied, {\SA} performs the last Lagrange multiplier update and stops in line~\ref{algo.sa.return}. Otherwise, depending on
the test in line~\ref{stat:begin.test.cond}, {\SA}
performs a Lagrange multiplier in line \ref{def:qi:chi},
or leaves it unchanged in line~\ref{sa.qi.update},
and in both cases moves on to the next iteration.

We next define a few concepts that will be used in the discussion and analysis of {\SA}.
Define the $k$-th epoch $\Ik$
as the index set 
\begin{equation}
\label{eq:cy}
{\cal I}_k := \{j_{k-1}+1, \ldots, j_k\},
\end{equation}
with the convention that $j_0=0$, 
and let
\begin{equation}\label{def.variablesTilde}
    (\ty^k,\tq^k,\tilde\lam^k, \tT_k):=(y^{j_k},q^{j_k},\lam^{j_k}, T_{j_k}) \quad \forall k\geq 0 \quad \text{and}\quad (\tv^k,\tilde \delta_k):=(v^{j_k},\delta_{j_k}) \quad \forall k\ge 1.
\end{equation}

We now make three additional remarks  about the logic of {\SA} regarding the prox stepsize and the Lagrange multiplier. First, since the prox stepsize $\lam^+$ output by {\SBIPP} is equal to the prox stepsize $\lam$ input to it, it follows from line 3 of {\SA} that $\lam^i=\lam^{i-1}$, and hence that $\lam^i=\lam^{0}$ for every $i \ge 1$. Second, due to the definition of $j_k$, it follows that $q^i=q^{i-1}$ for every $i \in \{j_{k-1}+1,\ldots, j_k-1\}=\Ik\setminus \{j_k\}$, which implies that 
\begin{equation}\label{eq:q.constant}
    q^{i-1}=q^{j_{k-1}}=\tq^{k-1} \ \ \forall i\in \Ik.
\end{equation}
Moreover, \eqref{def.variablesTilde} and \eqref{eq:q.constant} with $i=j_k$
imply that 
\begin{equation}\label{eq:q.update}
    \tq^{k} = q^{j_k}= q^{j_k-1} + c(Ay^{j_k}-b)=\tq^{k-1} +  c(A\ty^k-b) \quad \forall k\geq 1.
\end{equation}
Noting that
(A.2) implies that $b \in \text{Im}(A)$,
and using the assumption that $\tq^0=q^0 \in \text{Im}(A)$, identity \eqref{eq:q.update}, and
a simple induction argument,
we conclude that
$\tilde q^k \in A(\R^n)$ for every $k \ge 1$.

\vspace{1em}

Before stating the main result of this subsection,  we define the quantities
\begin{gather}
\begin{gathered}
\label{eq:def.Gamma.Static}
{\kappa(C):=\frac{2D_{\Psi}M_{\Psi}+
(2D_{\Psi}+1)(C + C^2+ \nabla_f)}{\bar d \nu^{+}_A}}, \\ 
\Gamma(y^0,q^0;c):=F_{\sup}-F_{\inf} + c\|Ay^0-b\|^2 + \left[\frac{4 (\chisum)}{\alpha c}+ \frac{1}{c}\right]\left(\|q^0\|^2 + \kappa^2(C)\right),
\end{gathered}
\end{gather}
where $(y^0,q^0,\lam^0, c)$ is the input of {\SA}, $(\chione, \chitwo)$ is as in \eqref{eq:sigma1&2}, $M_{\Psi}$ and $\bar{d}$ are as in (A5) and (A6), respectively, $(D_{\Psi}, \nabla_f)$ is as in \eqref{def:damH},  and  $\nu^{+}_A$ is the smallest positive singular value of the nonzero linear operator $A$.

The main iteration-complexity result for {\SA}, whose
proof is given in Section \ref{subsubsec:subroutine}, is stated next.

\begin{thm}[{\SA} Complexity]\label{thm:static.complexity}
Assume that $(\hat y, \hat q, \hat v, \hat \delta , \hat \lam) = {\SA}(y^0,q^0,\lam^0, c)$ for some
$(y^0,q^0,\lam^0, c) \in {\cal H} \times A(\mathbb{R}^n)\times \R_{++}^B \times \R_{++}$ such that
\begin{equation}\label{def:initial:stepsize}
\lam^0_t = \frac{1}{2m_t} \quad \forall t \in \{1,\ldots,B\}.
\end{equation}
 Then, for any
tolerance pair $(\rho,\eta)\in \R^2_{++}$,
the following statements hold for {\SA}:
\begin{enumerate}
\item[(a)] its total  number of iterations (and hence {\SBIPP} calls) is bounded by  
\begin{align}\label{thm:iter.complexity:SADMM}
 \left(\frac{\chisum}{\rho^2}\right)\Gamma(y^0,q^0; c) + 1,
\end{align}
where  $(\chione, \chitwo)$ and $\Gamma(y^0, q^0;c)$ are as in \eqref{eq:sigma1&2} and \eqref{eq:def.Gamma.Static}, respectively. 

\item[(b)] its output  $(\hat y,\hat q, \hat v, \hat \delta,\hat \lam)$ satisfies
$\hat \lam = \lam^0$,
\begin{equation}\label{residual:theorem:1}
\hat v \in \nabla f(\hat y) + \partial_{\hat \varepsilon} \Psi(\hat y)+A^*\hat q   \quad \text{and}\quad \|\hat v\|^2+\hat \varepsilon \le \rho^2,
\end{equation}
and the following bounds
\begin{equation}\label{feasi:theorem:1}
c\|A\hat y-b\| \le 2\max\{\|q^0\|,\kappa(C)\}  \quad \text{and} \quad
\quad \|\hat q\| \le \max\{\|q^0\|,\kappa(C)\};
\end{equation}
\item[(c)] if $
   c\geq 2 \max\{\|q^0\|,\kappa(C)\}/\eta$,
then the output $(\hat y, \hat q, \hat v, \hat \delta,\hat \lam)$ of {\SA}
  is a $(\rho,\eta)$-stationary solution of problem \eqref{initial.problem} according to \eqref{eq:stationarysol}. 
\end{enumerate}
\end{thm}

We now make some remarks about Theorem \ref{thm:static.complexity}. First, Theorem \ref{thm:static.complexity}(b) implies that {\SA} returns a quintuple $(\hat y, \hat q, \hat v, \hat \delta , \hat \lam)$ satisfying both conditions in \eqref{residual:theorem:1}, but not necessarily the feasibility condition $\|A\hat y-b\|\le \eta$. 
However, Theorem \ref{thm:static.complexity}(c) guarantees that, if $c$ is chosen large enough, i.e., $c = \Omega(\eta^{-1})$,
then the feasibility also holds,
and hence that $(\hat y,\hat q, \hat v,\hat \delta)$ is a $(\rho,\eta)$-stationary solution of \eqref{initial.problem}. 
Second, assuming for simplicity that $q^0=0$, it follows from \eqref{eq:def.Gamma.Static} and \eqref{thm:iter.complexity:SADMM} that the overall complexity of {\SA} is 
\[
{\cal O}\left(\frac{1+c}{\rho^2} \left( 1 + c f_0^2 \right) \right),
\]
where $f_0:=\|Ay^0-b\|$. Moreover,
 if the initial point $y^0$ satisfies $cf_0^2={\cal{O}}(1)$, then the bound further reduces to  ${\cal{O}}((1+c) \rho^{-2})$. 
Third, under the assumption made in  Theorem \ref{thm:static.complexity}(c), i.e., that
$c  = {\Theta}(\eta^{-1})$, then
the above two complexity estimates reduces to 
$ {\cal O}(\eta^{-2}\rho^{-2})$ if $y^0$ is arbitrary and to
${\cal O}(\eta^{-1}\rho^{-2})$ if
$y^0$ satisfies $cf_0^2 = {\cal O}(1)$.

Finally, it is worth discussing the dependence of the complexity bound~\eqref{thm:iter.complexity:SADMM} in terms of number of blocks $B$ only.
It follows from the definition of $\sigma_2$ 
in \eqref{eq:sigma1&2} that $\sigma_2={\Theta}(B)$. This implies that  $\Gamma(y^0,q^0;
c) = {\cal O}
(1+ B/\alpha)$ due to \eqref{eq:def.Gamma.Static}, and hence that
the complexity bound~\eqref{thm:iter.complexity:SADMM} is
${\cal O}(B(1+B\alpha^{-1})$.
Thus, if $\alpha$ is chosen to be
$\alpha=\Omega(B)$ then the dependence of
\eqref{thm:iter.complexity:SADMM} in terms of $B$ only is ${\cal O}(B)$.

In Section \ref{subsection:DA:ADMM}, we present an ADMM variant, namely {\DA}, which gradually increases the penalty parameter and achieves the complexity
bound $ {\cal O}(\eta^{-1}\rho^{-2})$ of the previous paragraph regardless of the choice of the initial point $y^0$.
Specifically,
{\DA}
 repeatedly invokes
{\SA}  using a warm-start strategy, i.e.,
if $c$ is the  penalty parameter used in the previous {\SA} call and $(\hat y, \hat q, \hat v, \hat \delta , \hat \lam)$ denotes its output, then the current {\SA} call uses 
$(\hat y,\hat q,\hat \lam, 2c)$
as input, and hence with penalty parameter
multiplied by two.

\vspace{1em}

\section{The Proof of {\SA}'s Complexity Theorem (Theorem \ref{thm:static.complexity})}
\label{subsubsec:subroutine}

This section gives the proof 
of Theorem~\ref{thm:static.complexity}. Its first result shows that every iterate $(y^i,v^i,\delta_i,\lam^i)$ of {\SA} satisfies the stationary inclusion $v^i  \in \nabla f (y^i)+ \partial_{\delta_i} \Psi(y^i) +
\text{Im}(A^*)$
and derives a bound on the residual error
$(v^i,\delta_i)$.

\begin{lemma}\label{lemma:SADMM.BIPP}
The following statements about {\SA} hold:
\begin{itemize}
    \item[(a)]
    for every iteration index $i \ge 1$,
\begin{equation}\label{eq.SA.stat.inclusion}
    v^i  \in \nabla f (y^i)+ \partial_{\delta_i} \Psi(y^i) +
A^*[q^{i-1} + c(Ay^i-b)];
\end{equation}
\item[(b)] if the initial prox stepsize $\lam^0$ is chosen as in \eqref{def:initial:stepsize}, then
for every iteration index $i \ge 1$,
\begin{align}\label{claim:bounded:iterations}
    \frac{1}{\chisum}(\|v^i\|^2+\delta_i) \le {\cal L}_c( y^{i-1};q^{i-1}) - {\cal L}_c(y^i;q^{i-1}) =  T_i-T_{i-1}.
\end{align}
\end{itemize}
\end{lemma}

\begin{proof}
It follows from line \ref{call:SBIPP} of {\SA}, Proposition~\ref{main:subroutine:1} with $(z,p,\lam,c)=(y^{i-1},q^{i-1},\lam^{i-1},c)$ and $(z^+,v^+,\delta_+,\lam^+,)=(y^{i},v^{i},\delta_i,\lam^{i})$, and inclusion \eqref{prop.B-IPP.inclusion} that \eqref{eq.SA.stat.inclusion} holds. It also follows that $\lam^i=\lam^{i-1}$, which easily implies that $\lam^j=\lam^0$ for every iteration index $j$ generated {\SA}. Thus, the previous fact and inequality \eqref{lemma:norm:residual:Ine} with $(z^+,p,v^+,\delta_+)=(y^{i},q^{i-1},v^{i},\delta_i)$ imply that \eqref{claim:bounded:iterations} holds. 
\end{proof}

\vspace{1em}

We make make some remarks about Lemma \ref{lemma:SADMM.BIPP}.
First,  \eqref{claim:bounded:iterations} implies that:
$\{T_i\}$ is nondecreasing;
and,
if $T_i=T_{i-1}$, then $(v^i,\delta_i)=(\mathbf{0},0)$,
which together with ~\eqref{eq.SA.stat.inclusion}, implies that the algorithm stops in line~\ref{algo.sa.return} with an exact stationary point for problem~\eqref{initial.problem}. 
In view of this remark, it is natural to view 
$\{T_i\}$ as a potential sequence.
Second,
if  $\{T_i\}$ is bounded, \eqref{claim:bounded:iterations} immediately implies that 
the quantity $\|v^i\|^2+\delta_i$ converges to zero, and hence that $y^i$ eventually becomes a near stationary point for problem~\eqref{initial.problem}, again
in view of ~\eqref{eq.SA.stat.inclusion}.
A major effort of our analysis will be to show that $\{T_i\}$ is  bounded.

With the above goal in mind, the following result gives an  expression for $T_i$ that plays an important role in our analysis.

\begin{lemma}\label{lemma:Ti.expression}
If $i$ is an iteration index generated by {\SA} such that $i\in \Ik$, then
\begin{equation}\label{eq:Ti.value}
T_i = \left [ {\cal L}_c(\ty^{0};\tq^{0}) - {\cal L}_c(y^i;\tq^{k-1}) \right ]+\frac{1}{c}\sum_{\ell=1}^{k-1}\|\tq^\ell-\tq^{\ell-1}\|^2.
\end{equation}   
\end{lemma}

\begin{proof}
We first note that
\begin{align}\label{eq.Ti.sum.1}
T_i -T_1 &= \sum_{j=2}^i (T_j-T_{j-1}) =
\sum_{j=1}^{i-1} (T_{j+1}-T_{j})=
\sum_{j=1}^{i-1}
 \left[ \Lc (y^{j};q^{j}) - \Lc(y^{j+1};q^{j})\right] 
 \nonumber \\
 &= \sum_{j=1}^{i-1} \left[ \Lc (y^{j};q^{j}) - \Lc (y^{j};q^{j-1})\right] +  \sum_{j=1}^{i-1}\left[\Lc(y^{j};q^{j-1}) - \Lc(y^{j+1};q^{j})\right]. 
\end{align}
Moreover, using the definition of $T_i$ with $i=1$ (see line \ref{def:T:i} of {\SA}), the fact that $q^{i-1}=\tq^{k-1}$ due to \eqref{def.variablesTilde} and simple algebra, we have
\begin{align}\label{eq.Ti.sum.2}
T_1 &+ \sum_{j=1}^{i-1} \left[\Lc(y^{j};q^{j-1}) - \Lc(y^{j+1};q^{j})\right] = T_1+ \Lc (y^1;q^0) - \Lc(y^i;q^{i-1})\nonumber \\
&= [\Lc(y^0;q^0) - \Lc(y^1;q^0)] + 
[\Lc (y^1;q^0) - \Lc(y^i;\tq^{k-1})] = \Lc(y^0;q^0) - \Lc(y^i;\tq^{k-1}).
\end{align}
Using the definition of the Lagrangian function (see definition in~\eqref{DP:AL:F}), relations \eqref{eq:q.update} and \eqref{eq:q.constant}, we conclude that for any $\ell \le k$,
\[
\Lc (y^{j};q^{j}) - \Lc (y^{j};q^{j-1}) \overset{\eqref{DP:AL:F}}=
\inner{Ay^j-b}{q^j-q^{j-1}} \overset{\eqref{eq:q.update}, \eqref{eq:q.constant}}= \left\{
\begin{array}{ll}
0 & \mbox{, if $j\in \Il\setminus \{j_\ell\}$;} \\
\dfrac{\|\tq^{\ell}-\tq^{\ell-1}\|^2}{c} & \mbox{, if $j=j_{\ell}$,}
\end{array}
\right.
\]
and hence that
\begin{equation*}
\sum_{j=1}^{i-1} \left[ \Lc (y^{j};q^{j}) - \Lc (y^{j};q^{j-1})\right] = \frac{1}{ c}\sum_{\ell=1}^{k-1}\|\tq^{\ell}-\tq^{\ell-1}\|^2.
\end{equation*}
Identity~\eqref{eq:Ti.value} now follows by combining the above identity with the ones in \eqref{eq.Ti.sum.1} and \eqref{eq.Ti.sum.2}.
\end{proof}

\vspace{1em}
The next technical  result will be used to establish an upper bound on the first term of the right hand side of \eqref{eq:Ti.value}.

\begin{lemma}\label{lem:lowerbound.L}
For any given 
$c >0$ and pairs
$(u,p) \in  {\cal H} \times \mathbb{R}^l$ and
$(\tilde u,\tilde p) \in  {\cal H} \times \mathbb{R}^l$, we have
\begin{equation}\label{def:diff:lagrangian}
{\cal L}_{c}(u; p)
- {\cal L}_{c}(\tu ; \tp )
\le F_{\sup}-F_{\inf}
+ c\|Au-b\|^2 +\frac{1}{c}\max\{\|p\|,\|\tilde p\|\}^2
\end{equation}
where  $(F_{\sup},F_{\inf} )$ is as in \eqref{def:damH}.
\end{lemma}

\begin{proof}
Using the definitions of ${\cal L}_c(\cdot\,;\,\cdot)$ and $F_{\inf}$ as in \eqref{DP:AL:F} and \eqref{def:damH}, respectively, we have
\begin{align*}
{\cal L}_{c}(\tu ; \tp)-F_{\inf} &\overset{\eqref{def:damH}}\geq  {\cal L}_{c}(\tu ; \tp)-(f+h)(\tu) \\
&\overset{\eqref{DP:AL:F}}=\langle \tp, A \tu-b\rangle +\frac{c}{2}\|A \tu-b\|^2=\frac{1}{2}\left\|\frac{\tp}{\sqrt{c}}+\sqrt{c}(A \tu-b)\right\|^2-\frac{\|\tp\|^2}{2c} \ge -\frac{\|\tp \|^2}{2c}.
\end{align*}
On the other hand, using the definitions of ${\cal L}_{c}(\, \cdot \,;\, \cdot \, )$  and $F_{\sup}$ as in \eqref{DP:AL:F} and \eqref{def:damH},
respectively, and the Cauchy-Schwarz inequality, we have
\begin{align*}
{\cal L}_{c}(u; p) - F_{\sup} 
&\overset{\eqref{def:damH}}\le {\cal L}_{c}(u; p)-(f+h)(u)  \overset{\eqref{DP:AL:F}}= \langle p, A u -b\rangle +\frac{c}{2}\|A u-b\|^2\\
&\le \|p\|\|A u-b\|+\frac{c}{2}\|A u-b\|^2 .
\end{align*}
 Combining the above two relations we then conclude that \eqref{def:diff:lagrangian} holds.
\end{proof}

\vspace{1em}

The following result shows that the sequence $\{T_i\}$ generated by {\SA} is bounded.

\begin{prop}\label{lem:number.of.cycles}
The following statements about {\SA} hold:
\begin{itemize}
    \item[(a)] 
    its total
    number ${\mK}$ of epochs is bounded by $\lceil  (\chione+c\chitwo)/\alpha\rceil$
where $\chione$ and $\chitwo$ are as in \eqref{eq:sigma1&2};
\item[(b)]
for every iteration index $i$, 
we have
$T_i \le \Gamma_{\mK}(y^0;c)$, where
\begin{equation}\label{pre:Ti}
 \Gamma_{\mK}(y^0;c) := F_{\sup}-F_{\inf}
+ c\|A\ty^0-b\|^2 +\frac{Q_{\mK}^2}{c}+\frac{(\chisum)F_{\mK}^2}{c\alpha},
\end{equation}
and
\begin{equation}\label{eq:def.Qk.Fk}
Q_{\mK} := \max \left\{ \|\tilde q^k\| : k \in \{0,\ldots,{\mK}-1\} \right\},  \quad F_{\mK} := \max \left\{ \| \tq^k-\tq^{k-1} \| : k \in \{1,\ldots,{\mK}-1\} \right\};
\end{equation}
\item[(c)]
 the number of iterations  performed by {\SA} is bounded by
\begin{equation}\label{total:complexity:E}
1+ \left( \frac{\chisum}{\rho^2} \right) \Gamma_{\mK}(y^0;c).
\end{equation}

\end{itemize} 
\end{prop}

\begin{proof}
(a) Assume for the sake of contradiction that {\SA} generates an epoch ${\mathcal{I}}_{K}$ such that $K > \lceil  (\chione+c\chitwo)/\alpha\rceil$, and hence $K \ge 2$. Using the
fact that $j_{K-1}$ is the last index of ${\cal I}_{K-1}$ and
noting the
epoch termination criteria
in line \ref{stat:begin.test.cond} of {\SA}, we then conclude that
$\tilde T_{K-1}/j_{K-1}\leq \rho^2/K$.
Also, since {\SA} did not terminate during epochs ${\cal I}_1,\ldots,{\cal I}_{K-1}$, it follows from its termination criterion in line~\ref{cond:vi:stop} that 
$\|v^i\|^2+\delta_i > \rho^2$ for every
iteration $i \le j_{K-1}$.
These two previous observations, \eqref{claim:bounded:iterations} with $i\in \{1,\ldots ,j_{K-1}\}$, the facts that $T_0=0$ by definition and
$T_{j_{K-1}}=\tilde T_{K-1}$ due to \eqref{def.variablesTilde}, imply that
\begin{align*}
\rho^2 < \frac{1}{j_{K-1}}\sum_{i=1}^{j_{K-1}}(\|v^i\|^2+\delta_i) \overset{\eqref{claim:bounded:iterations}}\leq \frac{\chisum}{j_{K-1}}\sum_{i=1}^{j_{K-1}}(T_{i}-T_{i-1}) =\frac{(\sigma_1+ c \sigma_2)\tilde T_{K-1}}{j_{K-1}}\leq \frac{(\sigma_1+ c \sigma_2)}{\alpha K} \rho^2.
\end{align*}
Since this inequality and the assumption (for the contradiction)  that  $K >\lceil (\sigma_1+c\sigma_2)/\alpha\rceil$
yield an immediate contradiction, the conclusion of the statement follows.

(b)  Since $\{T_i\} $ is nondecreasing, it suffices to show that $T_i \le \Gamma_{\mK}(y^0;c)$ holds for any $i\in {\cal I}_{\mK}$, where $\mK$ is the total number of epochs of {\SA} (see statement (a)). It follows from the definition of $Q_{\mK}$, and Lemma~\ref{lem:lowerbound.L} with $(u,p)=(\ty^{0};\tq^{0})$ and $(\tilde u,\tilde p)=(y^i;\tq^{\mK-1})$, that
\[
{\cal L}_c(\ty^{0};\tq^{0}) - {\cal L}_c(y^i;\tq^{\mK-1})\le F_{\sup}-F_{\inf}
+ c\|A\ty^0-b\|^2 +\frac{1}{c}Q_{\mK}^2.
\]
Now, using the definition of $F_{\mK}$ as in~\eqref{eq:def.Qk.Fk}, we have that
\[
 \frac{1}{ c}\sum_{j=1}^{\mK-1}\|\tq^j-\tq^{j-1}\|^2\leq  \frac{(\mK-1)}{c}F_{\mK}^2\leq       \frac{ (\chisum)}{c\alpha}F_{\mK}^2,
\]
where the last inequality follows from the fact that $\mK -1 \le (\chisum)/\alpha $  due to statement (a). 
The inequality $T_i \le \Gamma_{\mK}(y^0;c)$ now follows from the two inequalities above and identity \eqref{eq:Ti.value} with $k=\mK$.

(c) Assume by contradiction that there exists an iteration index $i$ generated by {\SA} such that
\begin{equation}\label{prel.contr:assum:i}
i > \left(\frac{\chisum}{\rho^2}\right)\Gamma_{\mK}(y^0;c) + 1.
\end{equation}
Since {\SA} does not stop at any iteration index smaller than $i$,  the stopping criterion in line \ref{cond:vi:stop} is violated at these iterations, i.e.,
$
\|v^j\|^2 +\delta_j > \rho^2
$ for every $j \le i-1$.
Hence, it follows from \eqref{claim:bounded:iterations}, the previous inequality, the fact that $T_0=0$ due to line \ref{def:T0} of {\SA}, and statement (b) that
\[
\frac{(i-1) \rho^2}{\chisum} < \frac{1}{\chisum}\sum_{j=1}^{i-1}(\|v^j\|^2+\delta_j) \le \sum_{j=1}^{i-1}(T_j-T_{j-1}) = T_{i-1}-T_0 \le T_i \leq  \Gamma_{\mK}(y^0,c),
\]
which contradicts \eqref{prel.contr:assum:i}.
Thus, statement (c)  holds.
\end{proof}




\vspace{1em}

We now make some remarks about Lemma~\ref{lem:number.of.cycles}.
First, Lemma~\ref{lem:number.of.cycles}(a) shows that the number of epochs depends linearly on $c$.
Second, Lemma~\ref{lem:number.of.cycles}(c) shows that the total number of iterations of
{\SA} is bounded
but
the derived bound is given in terms of the quantities $Q_{\mK}$ and $F_{\mK}$ in
\eqref{eq:def.Qk.Fk}, both of which depend on the magnitude of the sequence of generated Lagrange multipliers $\{\tilde q_k : k=1,\ldots,E\}$.
Hence, the bound in~\eqref{total:complexity:E} is algorithmic dependent in that it depends on the sequence $\{\tilde q^k\}$ generated by {\SA}.

In what follows, we derive a bound on the total number of iterations performed by {\SA} that depends only on the instance of \eqref{initial.problem} under consideration.
With this goal in mind,
the following result provides a uniform bound on the sequence of  Lagrange multipliers generated by {\SA} that depends only on
the instance of \eqref{initial.problem}.

\begin{lemma}\label{bounds:q:f}
    The following statements about {\SA} hold:
    \begin{itemize}
        \item[(a)] it holds that
        \begin{equation}\label{eq:Lag.Feas.Bound.2}
        \|\tq^k\| \leq \max\{\|q^0\|,\kappa (C)\}
        , \ \ \ \text{$\forall$ $k \in \{1,\ldots, {\mK}\}$} ;
         \end{equation}
        \item[(b)] if $i$ is an iteration index such that $\|v^i\|^2+\delta_i\leq C^2$, then
        \[
        c\|Ay^i-b\| \le 2 \max\{\|q^0\|,\kappa (C)\} ;
        \]

     \item[(c)] it  holds that
        \begin{equation}\label{bound:f:final}
        \|\tq^k - \tq^{k-1}\| \le 2 \max\{\|q^0\|,\kappa (C)\} 
        \ \ \ \text{$\forall$ $k \in \{1,\ldots, {\mK}\}$}.
        \end{equation}
    \end{itemize}
\end{lemma}

\begin{proof} 
(a) We first define the index set
\begin{equation}\label{def:I:k:C1}
    {\cal I}_k(C) := \{i\in {\cal I}_k:\|v^{i}\|^2+\delta_i\leq C^2\},
\end{equation}
where $C>0$ is part of the input for {\SA} and $(v^i, \delta_i)$ is as in line \ref{algo:IPP.v.delta} of {\SBIPP}.
We now claim that the vector pair $(\tq^{k-1}, y^i)$ satisfies
\begin{equation}\label{eq:FeasBound.proof.1}
\|\tilde q^{k-1} + c A(y^i-b)\| \le 
\max \{ \|\tilde q^{k-1}\|,\kappa(C)\}, \quad \forall i \in \Ik(C).
\end{equation}
To show the claim, 
let $i \in \Ik(C)$ be given. To simplify notation, define
\[
p^i:=\tilde q^{k-1}+c(Ay^i-b) \quad \text{and} \quad r^i := v^i-\nabla f(y^i)
\]
and note that the triangle inequality for norms, and the definitions of $\nabla_f$ in \eqref{def:damH} and $\Ik(C)$ in \eqref{def:I:k:C1}, imply that
\begin{equation}\label{eq:ri.deltai.C}
  \delta_i + \|r^i\|  \le C^2+\|v^i||+\|\nabla f(y^i)\| \le C^2 + (C + \nabla_f).  
\end{equation}

Using the fact that 
$(y^{i},v^{i},\delta_{i}, \lam^i)={\SBIPP}(y^{i-1},q^{i-1},\lam^{i-1}, c)$
in view of line \ref{call:SBIPP} of {\SA},
 the definitions of $p^i$ and $r^i$, Lemma \ref{lemma:SADMM.BIPP}(a)
 and \eqref{eq:q.constant},   
we have that
$
r^i \in  \partial_{\delta_i} \Psi(y^i) + A^*p^i
$,
and hence that the pair $(q^-,\varrho) = (q^{i-1},c)$ and the quadruple $(z,q,r,\delta)=(y^i, p^i, r^i, \delta_i)$ satisfy the 
conditions in \eqref{eq:cond:Lem:A1} of Lemma~\ref{lem:qbounds-2}. 
Hence, the conclusion of the same lemma, inequality~\eqref{eq:ri.deltai.C}, and the fact that $\varphi$ is non-decreasing, imply that
\begin{align*}
 \|p^i\| &\leq \max\left\{\| q^{i-1}\|,\varphi(\|r^i\|+\delta_i) \right\}\leq \max\left\{\| q^{i-1}\|,\varphi(C+C^2 + \nabla_f) \right\}
\overset{\eqref{eq:sigma1&2}}= \max\{\|\tq^{k-1}\|,\kappa(C)\},
\end{align*}
where the equality follows from 
\eqref{eq:q.constant} and the definitions  of $\kappa(\cdot)$ and $\varphi(\cdot )$ in \eqref{eq:def.Gamma.Static} and \eqref{eq:Technical.varphi.def}, respectively.
We have thus proved that the claim holds.

We now show that \eqref{eq:Lag.Feas.Bound.2} holds. 
Using \eqref{eq:FeasBound.proof.1} with $i=j_k$, the facts that $\tq^k=\tq^{k-1}+ c(A\ty^k-b)$ due to \eqref{eq:q.update},
and $\tilde y^{k}=y^{j_k}$ due to \eqref{def.variablesTilde}, and the triangle inequality for norms, we have that
\begin{align*}
\|\tq^k \|= \|\tilde q^{k-1} + c A(\ty^k-b)\| \le \max \{ \|\tilde q^{k-1}\|,\kappa(C)\}.
\end{align*}
Inequality \eqref{eq:Lag.Feas.Bound.2} now follows by recursively using the last inequality and the fact that $\tq^0=q^0$ due to  \eqref{def.variablesTilde}.  

 (b)
Using that $c(Ay^i-b) = p^i-\tq^{k-1}$ and \eqref{eq:FeasBound.proof.1} we have 
\begin{align*}
c\|A y^i-b\| \leq \| p^i\|+\|\tq^{k-1}\|
\overset{\eqref{eq:FeasBound.proof.1}}\le \max\{\|\tilde q^{k-1}\|,\kappa(C)\}+\|\tq^{k-1}\|
\overset{\eqref{eq:Lag.Feas.Bound.2}}\leq 2\max\{\|\tilde q^{0}\|,\kappa(C)\}
\end{align*}
where the last inequality above follows from \eqref{eq:Lag.Feas.Bound.2} with $k=k-1$ and the fact that $\tq^0=q^0$ due to  \eqref{def.variablesTilde}.

 (c) Statement (c) follows from \eqref{eq:q.update}, the triangle inequality for norms, statement (b) with $i=j_{k}$ and the fact that $\tilde y^k = y^{j_{k}}$ due to \eqref{def.variablesTilde}.
\end{proof}




\vspace{1em}

\noindent
{\bf Proof of Theorem~\ref{thm:static.complexity}:} (a) It follows from Proposition \ref{lem:number.of.cycles}(c) that the total number of iterations generated by {\SA} is bounded by the expression in \eqref{total:complexity:E}.
Now, recalling that $\mK$ is the last epoch generated by {\SA}, using \eqref{eq:def.Qk.Fk}, \eqref{eq:Lag.Feas.Bound.2} and \eqref{bound:f:final} we see that $Q_{\mK} \leq \max\{\|q^0\|,\kappa (C)\}$ and $ F_{\mK} \leq 2\max\{\|q^0\|,\kappa (C)\}$, which implies that $\Gamma_{\mK}(y^0;c)\leq \Gamma(y^0, q^0;c)$, 
where $\Gamma_{\mK}(y^0;c)$ and   $\Gamma(y^0, q^0;c)$ are as in \eqref{pre:Ti} and \eqref{eq:def.Gamma.Static}, respectively.
The conclusion now follows from the two previous observations.

(b) We first prove that 
the inclusion in \eqref{residual:theorem:1} holds. It follows from Lemma~\ref{lemma:SADMM.BIPP}(a) with $i=j_{\mK}$ and~\eqref{def.variablesTilde} with $k=\mK$ that
\[
    \tv^{\mK}  \in \nabla f (\ty^{\mK})+ \partial_{\tilde \delta_{\mK}} \Psi(\ty^{\mK}) +
A^*[q^{j_{\mK}-1} + c(A\ty^{\mK}-b)].
\]
We conclude that \eqref{residual:theorem:1} holds by using~\eqref{eq:q.constant} with $i=j_{\mK}$, \eqref{eq:q.update} with $k=\mK$, and the fact that $(\hat y,\hat q, \hat v, \hat \delta)=(\ty^{\mK},\tq^{\mK},\tv^{\mK},\tilde \delta_{\mK})$.

The inequality in \eqref{residual:theorem:1} follows from the fact that {\SA} terminates in  line~\ref{cond:vi:stop} with the condition $\|\hat v\|^2+\hat \delta = \|v^{j_{\mK}}\|^2+\delta_{j_{\mK}}\le \rho^2$ satisfied.

The first inequality in~\eqref{feasi:theorem:1} follows from Lemma \ref{bounds:q:f}(b) with $i=j_{{\mK}}$ and the fact that $\tilde y^{\mK}= y^{j_{\mK}}$ due to \eqref{def.variablesTilde}.
Finally, the second inequality in \eqref{feasi:theorem:1} follows from  Lemma~\ref{bounds:q:f}(a) and the fact that $(y^{j_{\mK}}, q^{j_{\mK}})=(\ty^{\mK}, \tq^{\mK})$ due to \eqref{def.variablesTilde}.

(c) Using the assumption that $
   c\geq 2 \max\{\|q^0\|,\kappa(C)\}/\eta$,  statement (b)  guarantees that {\SA} outputs $\hat y= y^{j_k}$ satisfying $\|A\hat y-b\|\leq [2\max\{\|q^0\|,\kappa(C)\}]/c\le \eta.$ Hence, the conclusion that $(\hat y,\hat q,\hat v, \hat \delta)=(\tilde y^k, \tilde q^k,\tilde v^k,\tilde\delta_k)$ satisfies \eqref{eq:stationarysol} follows from the previous inequality, the inclusion in \eqref{residual:theorem:1}, and the last inequality in \eqref{feasi:theorem:1}.

\section{An {\DA} With Fixed Prox Stepsizes}
\label{subsection:DA:ADMM}

This section describes  {\DA}, the main algorithm of this chapter, and its iteration complexity.  
The version of {\DA} presented in this section keeps the prox stepsize constant throughout since it performs multiple calls to the {\SA} 
which, as already observed, also has this same attribute.
An adaptive variant of {\DA}
with variable prox stepsizes is presented in Section~\ref{section:Adaptive.ADMM}.

{\DA} is formally stated next.

\begin{algorithm}[H]
\setstretch{1}
\caption{{\DA}}\label{alg:dynamic}
\begin{algorithmic}[1]

\Statex \hskip-1.8em
\textbf{Universal Input:} tolerance pair $(\rho,\eta) \in \R^2_{++}$, $\alpha\in [\rho^2, +\infty)$, and $C \in [\rho,+\infty)$ 
\Require $x^0\in {\cal H}$ and $\gamma^0=(\gamma_1^0,\ldots,\gamma_B^0)\in \R^B_{++}$
\Ensure $(\hat x,\hat p, \hat v, \hat{\varepsilon},\hat c)$
\vspace{1em}
\State $p^0=(p^0_1,\ldots,p^0_B) \gets (0,\ldots,0)$ and $c_0\gets 1/[1+\|Ax^0-b\|]$ \label{def:c0:p0} 

\For{$\ell \gets 1,2,\ldots$}
    \State $(x^\ell ,p^\ell,v^\ell,\varepsilon_\ell, \gamma^\ell )=\SA(x^{\ell-1},p^{\ell-1},\gamma^{\ell-1}, c_{\ell-1})$ \label{dyn:stat.call}
    \State $c_{\ell}=2c_{\ell-1}$ \label{dyn:penalty.update}
    \If{$\|Ax^\ell-b\|\leq \eta$}\label{test:DA:ADMM}

        \State $(\hat x,\hat p, \hat v , \hat\varepsilon,\hat c)=(x^\ell,p^\ell, v^\ell, \varepsilon_\ell, c_\ell)$
        \State \Return $(\hat x,\hat p, \hat v , \hat\varepsilon,\hat c)$   
    \EndIf
\EndFor
\end{algorithmic}
\end{algorithm}

We now make some remarks about {\DA}. First,  even though an
initial penalty parameter $c_0$
is prescribed in line~\ref{def:c0:p0} for the sake of analysis simplification, {\DA} can be equally shown to converge for other choices of $c_0$. Second, it uses a ``warm-start" strategy for calling {\SA}, i.e.,  after the first call to {\SA},  the input of  any 
 {\SA} call is the output 
of the previous {\SA} call. 
Third, Lemma~\ref{lemma:translate} below and Theorem~\ref{thm:static.complexity}(b) imply that each  {\SA} call in line \ref{dyn:stat.call} of {\DA} generates a 
quintuple $(x^\ell ,p^\ell,v^\ell,\varepsilon_\ell, \gamma^\ell )$ satisfying the first two conditions in \eqref{eq:stationarysol}, but not necessarily the feasibility bound. To achieve the feasibility bound in \eqref{eq:stationarysol}, {\DA} increases the penalty parameter according to the rule $c_{\ell} = 2c_{\ell-1}$, as specified in its line~\ref{dyn:penalty.update}. Finally, {\DA} stops
if the test in line~\ref{test:DA:ADMM} is satisfied; if the test fails, then {\SA} is called again using the ``warm-start" strategy.

\vspace{1em}

Before describing the main result, we define the following constant, and which appear in the total iteration complexity,
\begin{equation}\label{eq:def.Gamma.dynamic}
\bar \Gamma(x^0;C):=F_{\sup}-F_{\inf} +  \frac{8\sigma_2\kappa^2(C)}{\alpha} +2\kappa^2(C)\left(\frac{4\sigma_1}{\alpha}+9\right)(1+\|Ax^0-b\|),
\end{equation}
where $(\sigma_1,\sigma_2)$ is as in~\eqref{eq:sigma1&2} and $\kappa(C)$ is as in \eqref{eq:def.Gamma.Static}.

Recalling that every {\DA} iteration makes a {\SA} call, the following result 
translates the properties of {\SA} established  in Theorem~\ref{thm:static.complexity} to the context of {\DA}.

\begin{lemma}\label{lemma:translate}
Let $\ell$ be an iteration index of ${\DA}$. Then,  the following statements hold:
\begin{itemize}
    \item[(a)] the sequences $\{(x^{k},p^{k},v^k,\varepsilon_k,\gamma^k)\}_{k=1}^\ell$ 
    and $\{c_k\}_{k=1}^{\ell}$ satisfy
\begin{equation}\label{eq:prop:input.station}
v^k \in \nabla f(x^k) + \partial_{\varepsilon_k} \Psi(x^k)+A^*p^k   \quad \text{and}\quad \max_{1 \le k \le \ell} \|v^k\|^2+\varepsilon_k \le \rho^2,
\end{equation}
the identity $\gamma^{k}=\gamma^0$,
and the following bounds
\begin{equation}
\label{eq:prop:input.bounds}
\max_{1 \le k \le \ell} \|p^{k}\| \le \kappa(C) \quad \text{and} \quad \max_{1 \le k \le \ell}c_{k} \|Ax^{k} -b \|
\le 4\kappa(C);
\end{equation}

\item[(b)] the number of iterations performed
by the  {\SA} call  within the $\ell$-th iteration of {\DA} (see line~\ref{dyn:stat.call} of {\DA}) is bounded by 
\begin{equation}\label{eq:DA.ell-th.SA.call}
\left(\frac{\sigma_1+c_{\ell-1}\sigma_2}{\rho^2}\right)\bar \Gamma(x^0;C) + 1,
\end{equation}
where $\bar \Gamma(x^0;C)$ is as in \eqref{eq:def.Gamma.dynamic};

\item[(c)] if $c_{\ell}\geq 4\kappa(C)/\eta$
then $(x^\ell, p^\ell, v^\ell, \varepsilon_\ell, \gamma^\ell)$ is a $(\rho,\eta)$-stationary solution of problem \eqref{initial.problem}. 
\end{itemize}
\end{lemma}

\begin{proof}
(a)  Using Theorem \ref{thm:static.complexity}(b) with $(y^0,q^0, \lam^0, c)= (x^{k-1},p^{k-1}, \gamma^{k-1}, c_{k-1})$ and noting line~\ref{dyn:stat.call} of {\DA}, we conclude that for any $k\in\{1,\ldots,\ell\}$, the quintuple $(x^{k},p^{k},v^{k},\varepsilon_{k},\gamma^k)$ satisfies \eqref{eq:prop:input.station} and the conditions
\begin{equation}\label{step:induction}
\gamma^{k}=\gamma^{k-1}, \quad \|p^{k}\| \le \max\{\|p^{k-1}\|, \kappa(C)\}, \quad  c_{k-1} \|Ax^{k} -b \|
\le 2\max\{\|p^{k-1}\|, \kappa(C)\}.
\end{equation}
A simple induction argument applied to both the identity and the first inequality in \eqref{step:induction}, with the fact that $p^0 = 0$, show that $\gamma^k=\gamma^0$ and that the first inequality in \eqref{eq:prop:input.bounds}  holds.
The second inequality in \eqref{step:induction}, the assumption that $p^0=0$, the fact that $c_k=2c_{k-1}$ for every $k \in \{1,\ldots, \ell\}$,
and the first inequality in~\eqref{eq:prop:input.bounds}, imply that the second inequality in \eqref{eq:prop:input.bounds} also holds. 

(b) Theorem \ref{thm:static.complexity}(a) with $(y^0,q^0,\lam^0, c) = (x^{\ell-1},p^{\ell-1},\gamma^{k-1},c_{\ell-1})$ imply that the total number of iterations performed
by the  {\SA} call  within the $\ell$-th iteration of {\DA} is bounded by
\[
\left(\frac{\sigma_1+c_{\ell-1}\sigma_2}{\rho^2}\right)\Gamma(x^{\ell-1},p^{\ell-1};c_{\ell-1})+1,
\]
where $\Gamma(\cdot,\cdot;\cdot)$ is as in~\eqref{eq:def.Gamma.Static}. 
Thus, to show 
\eqref{eq:DA.ell-th.SA.call}, it suffices to show that $\Gamma(x^{\ell-1},p^{\ell-1};c_{\ell-1})\leq \bar \Gamma(x_0;C)$.

Before showing the above inequality, we first show that
$c_{\ell-1}\|Ax^{\ell-1}-b\|^2\leq 16\kappa^2(C)/c_0$ for every index $\ell$. 
Indeed, this observation trivially holds for $\ell=1$ due to the fact that $c_0=1/(1+\|Ax^0-b\|) \le 1$ (see line~\ref{def:c0:p0} of {\DA}) and the assumption that  $\kappa(C)\geq 1$. Moreover, the second inequality in~\eqref{eq:prop:input.bounds} and the fact that $c_{\ell-1}\geq c_0$ show that the inequality also holds for
$\ell>1$, and thus it holds for any $\ell \ge 1$.

Using the last conclusion, the definition of $\Gamma(x^{\ell-1},p^{\ell-1};c_{\ell-1})$, the fact that $c_{\ell-1}\geq c_0$, and the first inequality in \eqref{eq:prop:input.bounds}, we have

\begin{align*}
\Gamma(x^{\ell-1},p^{\ell-1};c_{\ell-1})&\leq F_{\sup}-F_{\inf} + \frac{16\kappa^2(C)}{c_0} + \left[\frac{4\sigma_1}{\alpha c_{0}}+\frac{1}{c_{0}}+\frac{4\sigma_2}{\alpha}\right]\left(\|p^{\ell-1}\|^2 + \kappa^2(C)\right)\\
&\overset{\eqref{eq:prop:input.bounds}}\leq F_{\sup}-F_{\inf} + \frac{16\kappa^2(C)}{c_0} + \left[\frac{4\sigma_1}{\alpha c_{0}}+\frac{1}{c_{0}}+\frac{4\sigma_2}{\alpha}\right]\left(2\kappa^2(C)\right)\\
&=F_{\sup}-F_{\inf} +  \frac{8\sigma_2\kappa^2(C)}{\alpha} +\frac{2\kappa^2(C)}{c_0}\left(\frac{4\sigma_1}{\alpha}+9\right)=\bar \Gamma(x_0;C),
\end{align*}
where the last identity follows from $c_0=1/(1+\|Ax^0-b\|)$ and the definition of $\bar \Gamma(x_0;C)$ in \eqref{eq:def.Gamma.dynamic}.

(c) Assume that  $c_{\ell}\geq 4\kappa(C)/\eta$.
This assumption, the first inequality in \eqref{eq:prop:input.bounds}, and the fact that $c_{\ell}=2c_{\ell-1}$, immediately imply that $c_{\ell-1}\geq 2\max\{\|p^{\ell-1}\|,\kappa(C)\}/\eta$. The statement now follows from the previous observation, line~\ref{dyn:stat.call} of {\DA}, and Theorem~\ref{thm:static.complexity}(c) with $(y^0,q^0,\lam^0,c) = (x^{\ell-1},p^{\ell-1},\gamma^{\ell-1},c_{\ell-1})$.

\end{proof}

\vspace{1em}
The next result describes the iteration-complexity of {\DA} in terms of total ADMM iterations (and hence {\SBIPP} calls within {\SA}). 

\begin{theorem}[{\DA} Complexity]\label{the:dynamic}

The following statements about {\DA} hold:
   \begin{itemize}
       \item[(a)]
       it obtains a $(\rho,\eta)$-stationary solution of \eqref{initial.problem} in no more than $
   \log_2 \left[1+4\kappa(C)/(c_0\eta) \right]+1$ calls to {\SA};
   \item[(b)]
   its total number of {\SA} iterations (and hence {\SBIPP} calls within {\SA}) is bounded by
\[
\frac{8\sigma_2\bar \Gamma(x^0;C) 
\kappa(C)}{\rho^2\eta} + \frac{\sigma_2\bar \Gamma(x^0;C)}{c_0\rho^2}+
\left[ 1+\frac{\sigma_1\bar \Gamma(x^0;C)}{\rho^2}\right]\log_2\left( 2+ \frac{8\kappa(C)}{c_0\eta}\right) 
\]
\end{itemize} 
where $(\sigma_1,\sigma_2)$, $\kappa(C)$   and $\bar \Gamma(x^0;C)$ are as in~\eqref{eq:sigma1&2}, \eqref{eq:def.Gamma.Static} and \eqref{eq:def.Gamma.dynamic}, respectively, and $c_0$ is defined in line \ref{def:c0:p0} of {\DA}.
\end{theorem}
\begin{proof}
(a)
Assume for the sake of contradiction that {\DA} generates an iteration index $\hat \ell$ such that $\hat \ell  > 1 + \log_2 \left[1+4\kappa(C)/(c_0\eta) \right] > 1$, and hence 
\[
    c_{\hat \ell-1} = c_0 2^{\hat \ell-1} > c_0\left(1+\frac{4 \kappa(C)}{c_0\eta} \right) > \frac{4 \kappa(C)}{\eta}.
\]
Using Lemma~\ref{lemma:translate}(c) with $\ell=\hat\ell-1 \ge 1$, we conclude that the quintuple $(x^{\hat\ell-1},p^{\hat\ell-1},v^{\hat\ell-1},\varepsilon_{\hat\ell-1}, \gamma^{\hat\ell-1})$ is a $(\rho,\eta)$ stationary solution of problem \eqref{initial.problem}, and hence satisfies $\|Ax^{\hat\ell-1}-b\| \le \eta$. In view of line~\ref{test:DA:ADMM} of {\DA}, this implies that {\DA} stops at the $(\hat \ell-1)$-th iteration, which hence contradicts the fact  that $\hat \ell$ is an iteration index. We have thus proved that (a) holds.

(b) Let $\tilde \ell$ denote the total number of {\SA} calls and observe that $\tilde{\ell}\le1+\log_2[1+4\kappa(C)/(c_0\eta)]$ due to (a). Now, using Lemma~\ref{lemma:translate}(b) and the previous observation, we have that the overall number of iterations performed by {\SA} is bounded by
\begin{align*}
\sum_{\ell=1}^{\tilde \ell}\left[ \left(\frac{\sigma_1+c_{\ell-1}\sigma_2}{\rho^2}\right)\bar \Gamma(x^0;C) + 1 \right] &=\left[ 1+\frac{\sigma_1\bar \Gamma(x^0;C)}{\rho^2}\right] \tilde \ell + \frac{\sigma_2 \bar \Gamma(x^0;C)}{\rho^2}\sum_{\ell=1}^{\tilde \ell}c_{\ell-1} \\
&\leq \left[ 1+\frac{\sigma_1\bar \Gamma(x^0;C)}{\rho^2}\right]\tilde \ell + \frac{c_0\sigma_2\bar \Gamma(x^0;C)}{\rho^2}\left( 2^{\tilde \ell} -1\right)\\
&\leq \left[ 1+\frac{\sigma_1\bar \Gamma(x^0;C)}{\rho^2}\right]\tilde \ell + \frac{c_0\sigma_2\bar \Gamma(x^0;C)}{\rho^2}\left( 1+\frac{8\kappa(C)}{c_0\eta} \right).
\end{align*}
The result now follows by using that $\tilde{\ell}\le1+\log_2[1+4\kappa(C)/(c_0\eta)]$.
\end{proof}

We now make some comments about Theorem~\ref{the:dynamic}.
First, it follows from Theorem~\ref{the:dynamic}(a)  that {\DA} ends with a $(\rho,\eta)$-stationary solution of \eqref{initial.problem} by calling
{\SA} no more than ${\cal O} (\log_2(\eta^{-1}))$ times.
Second, under the mild assumption that $\|Ax^0-b\| = {\cal O}(1)$, Theorem~\ref{the:dynamic}(b) implies that the complexity of {\DA}, in terms of the tolerances only, is ${\cal O} (\rho^{-2}\eta^{-1})$, and thus
${\cal O}(\epsilon^{-3})$
where $\epsilon:=\min\{\rho, \eta\}$.
On the other hand,
{\SA} only achieves this complexity 
with (a generally non-computable)
$c = \Theta(\eta^{-1})$ and with the condition that
$c\|Ax^0-b\|^2 =
{\cal O}(1)$ (see the first paragraph following Theorem~\ref{thm:static.complexity}), or equivalently,
$\|Ax^0-b\|= {\cal O}(\eta^{1/2})$, and hence the initial point $x^0$ being nearly feasible. 


\section{An {\DA} With Adaptive Prox Stepsizes}\label{section:Adaptive.ADMM}

This section presents an
{\DA}, with adaptive proximal stepsizes, which does not require knowledge of the weakly convexity parameters $m_t$'s. 

In view of our description of {\SA} and {\DA}, which redundantly included prox stepsizes in their input and output, the description of their adaptive prox stepsize versions now requires minimal changes to the presentation of the previous subsections.
Specifically,
instead of calling {\SBIPP},
the new prox stepsizes of {\SA} now calls an adaptive version of {\SBIPP}, referred to as {\ABIPP}, which we now describe.

Recall that {\SBIPP}, presented in Subsection \ref{sub:B-IPP}, assumes that the $m_t$'s are  available and keeps every block proximal stepsize constant,
i.e., the $t$-block prox stepsize $\lam_t$ is set to $1/(2m_t)$ at every iteration within {\SA}.
Instead,
{\ABIPP}
requires as input an arbitrary initial prox stepsize $\lam_t^0$ for each block $t \in \{1,\ldots,B\}$, which are then
adaptively
changed during its course.

{\ABIPP} is formally stated below.




{\floatname{algorithm}{Subroutine}
\renewcommand{\thealgorithm}{}
\begin{algorithm}[H]
\setstretch{1}
\caption{{\ABIPP}}\label{algo:BlockIPP}
\begin{algorithmic}[1]
\Statex \hskip-1.8em
\Require $(z,p,\lam, c)
\in {\cal H} \times A(\mathbb{R}^n) 
\times \R^B_{++}\times \R_{++} $
\Ensure $(z^+,v^+,\delta_+,\lam^+)\in {\cal H} \times \mathbb{R}^l 
\times \R_{++} \times \R^B_{++}$
\vspace{1em}
\\ 
\textcolor{violet}{STEP 1: Block-IPP Iteration}
  \For{$t=1,\ldots, B$}\label{algo:ABIPP.loop.inexact.sol}
    
    \State set $\lam_t^+= \lam_t$\label{line:ABIPP:lam}
    \State compute an $\left(1/8;z_t\right)$-stationary solution $(z_t^+,r_t^+,\varepsilon_t^{+})$ of \eqref{algo:BIPP.inexact.subproblem} with composite
     term $\lam_t^+\Psi_t(\cdot)$  \label{output:step:2:ABIPP}
    \vspace{1em}
\If{$z_t^+$ does \textbf{NOT} satisfy 
\begin{equation}\label{eq:test:adap:new}
{\cal L}_c(z_{<t}^+,z_{t}, z_{>t}; p)-{\cal L}_c(z_{<t}^+,z_t^+ ,z_{>t}; p)  \ge 
 \frac{1}{8\lam_t^+} \|z_t^+-z_t\|^2 + \frac{c}{4} \|A_t (z_t^+-z_t)\|^2
 \end{equation}\hskip1.3em}\label{algo:ABIPP.if.loop}
 \State $\lambda_t^+ \gets \lambda_t^+/2$  and go to line \ref{output:step:2:ABIPP}.\label{eq:test:ABIPP}
\EndIf
\EndFor
\vspace{1em}
\State $z^+\gets (z_1^+,\ldots ,z_B^+)$ and
  $\lam^+\gets (\lam_1^+,\ldots ,\lam_B^+)$\label{algo:ABIPP.primal}

\vspace{1em}     
    \\
    \textcolor{violet}{STEP 2: Proceed exactly as in STEP 2 of {\SBIPP} to obtain $(v^+,\delta_+)$} \label{algo:ABIPP.dual}
   
\end{algorithmic}
\end{algorithm}
\renewcommand{\thealgorithm}{\arabic{algorithm}}
}

We now clarify some aspects of {\ABIPP}. 
First, in contrast to {\SBIPP} which sets $\lam_t=1/(2m_t)$ based on the assumption that the $m_t$'s are available,
{\ABIPP} has to perform a prox stepsize search as $m_t$'s are assumed not be available.
Specifically, starting with $\lam_t^+$ set to $\lam_t$, it searches for an appropriate $\lam_t^+$ in the loop consisting of lines \ref{output:step:2:ABIPP} to \ref{eq:test:ABIPP}
that will eventually satisfies~\eqref{eq:test:adap:new}.
Second, the main motivation to enforce condition~\eqref{eq:test:adap:new}
is that it allows us to show
(see Proposition \ref{prop:ABIPP.output}(b) below) an inequality similar to
the one in~\eqref{claim:bounded:iterations} which, as already observed in the
third and fourth remarks of the
paragraph
immediately following Proposition~\ref{main:subroutine:1}, plays a fundamental role in the analysis of {\SA}
given in Sections \ref{adp:ADMM} and \ref{subsubsec:subroutine}.

\vspace{1em}
The following result shows that any of the $B$ loops (lines \ref{output:step:2:ABIPP}-\ref{eq:test:ABIPP}) of {\ABIPP} always terminate.

\begin{lemma}\label{Lemma:ABIPP}    
The following statements about {\ABIPP} hold for every $t \in \{1,\ldots,B\}$:
\begin{itemize}
    
\item[(a)] The quadruple $(z_t^+,r_t^+,\varepsilon_t^{+}, \lam_t^+)$ computed in line~\ref{output:step:2:ABIPP} of {\ABIPP} satisfies the relations
\begin{gather}\label{eq:ref:lemma:A1}
\begin{gathered}
r_t^+\in \nabla\left[\lam_t^+ {\cal \hat L}_{c}(z_{<t}^+,\cdot ,z_{>t}; p)+\frac{1}{2}\|\cdot -z_{t}\|^2\right](z_t^+)+\partial_{\varepsilon_t^+}(\lam_t^+ \Psi_t)(z_t^+), \\
\|r_t^+\|^2 + 2 \varepsilon_t^+ \leq \frac{1}{8}\| z_t^+-z_t\|^2,
\end{gathered}
\end{gather}
where ${\cal \hat L}_{c}(z_{<t}^+,\cdot ,z_{>t}; p)$ is defined in \eqref{def:smooth:ALM};

    \item
    [(b)] if $\lam_t \in (0,1/(2m_t)]$
    then 
    the loop consisting of lines \ref{output:step:2:ABIPP} to \ref{eq:test:ABIPP} terminates  in one 
    iteration with a pair $(z_t^+,\lam_t^+)$
    satisfying the identity $\lam_t^+=\lam_t$ and the inequality \eqref{eq:test:adap:new},
    and the condition that
    the objective function of the $t$-th block subproblem \eqref{algo:BIPP.inexact.subproblem} 
 is $(1/2)$-strongly convex;
 \item[(c)]
 for any $\lam_t>0$,  
  the loop consisting of lines \ref{output:step:2:ABIPP} to \ref{eq:test:ABIPP} terminates  in at most $1+\lceil \log_2(1+4m_t\lam_t) \rceil$ iterations with a pair $(z_t^+,\lam_t^+)$ satisfying \eqref{eq:test:adap:new} and the inequality $\lam_t^+ \ge \min\{ \lam_t, 1/(4m_t)\}$.

\end{itemize}
\end{lemma}

\begin{proof}
(a) Both relations in~\eqref{eq:ref:lemma:A1} immediately follows from line~\ref{output:step:2:ABIPP} of {\ABIPP} and Definition~\ref{def:BB}.

(b) This statement follows immediately from the following claim, namely,
if $\lam_t^+\leq 1/(2m_t)$ at the beginning of a loop iteration then the objective function of the $t$-th block subproblem \eqref{algo:BIPP.inexact.subproblem} 
 is $(1/2)$-strongly convex, the pair
$(z_t^+, \lam_t^+)$ satisfies \eqref{eq:test:adap:new}, and hence that the loop ends at this iteration.

To prove the claim, let $h_t(\cdot)$ denote the objective function of the $t$-th block subproblem \eqref{algo:BIPP.inexact.subproblem} and assume that $\lam_t^+\leq 1/(2m_t)$ at the beginning of some loop iteration. We first prove that the function in  $h_t(\cdot)$  is $(1/2)$-strongly convex. The assumption that  $\lam_t^+ \in (0, 1/(2m_t)]$ implies that the matrix $B_t : = (1-\lam_t^+ m_t)I+\lam_t^+ cA_t^*A_t$ is clearly positive definite, and hence defines the norm whose square is
\begin{equation}\label{def:norm:Bt:Ine:AB}
\|\cdot\|_{B_t}^2:=\langle\;\cdot\;, \, B_t(\cdot) \, \rangle 
\ge
\lam_t^+ c \|A_t (\cdot)\|^2 + \frac12 \|\cdot\|^2.
\end{equation}
Moreover, using assumption (A3), the definition of the Lagrangian function in \eqref{DP:AL:F}, and the definition of $h_t(\cdot)$, we can easily see that the function 
$h_t(\cdot) - \frac12 \|\cdot\|^2_{B_t}$ is  convex, and hence $h_t(\cdot)$ is $(1/2)$-strongly convex due to the inequality in \eqref{def:norm:Bt:Ine:AB}.
We now prove that inequality \eqref{eq:test:adap:new} holds. 
Due to \eqref{eq:ref:lemma:A1}, the quadruple $(z_t^+,r_t^+, \varepsilon_t^+,\lam_t^+)$ satisfies
\begin{align*}
r_t^+\overset{\eqref{eq:ref:lemma:A1}}\in  \nabla \Big[\lam_t \hat {\cal L}_{c}(z_{<t}^+,\cdot,z_{>t}; p)+\frac{1}{2}\|\cdot-z_{t}\|^{2}\Big](z_t^+) 
         + \partial_{\varepsilon_t^+}(\lam_t \Psi_t) (z_t^+)
         \subset \partial_{\varepsilon_t^+}
         h_t(z_t^+) 
        \end{align*}
where the set inclusion is due to
\cite[Theorem 3.1.1]{lemarechal1993} and the fact that $\lam_t^+ \hat {\cal L}_{c}(z_{<t}^+,\cdot,z_{>t}; p)+\frac{1}{2}\|\cdot-z_{t}\|^{2}$, $\lam_t^+ \Psi_t(\cdot )$ and $h_t(\cdot )$ are convex functions.
By applying Lemma \ref{conv:result}  with $\psi=h_t$
   , $(\xi,\tau, Q)=(1, 1, B_t)$, $(u,y,v) = (z_{t},z_t^+,r_t^+)$, and $\eta=\varepsilon_t^+$,
   we see that
\begin{align*}
    \lam_t^+ {\cal L}_c(z_{<t}^+,z_{t}, z_{>t}; p)&-\lam_t^+ {\cal L}_c(z_{<t}^+,z_t^+, z_{>t}; p) \ge  \frac{1}{2}\|z_t^+-z_t\|^2 + \frac{1}{4}\|z_t^+-z_t\|_{B_t}^2 -2\varepsilon_t^++\langle r_t^+, z_t-z_t^+\rangle\\
    &\overset{\eqref{def:norm:Bt:Ine:AB}}\ge \frac{1}{2}\|z_t^+-z_t\|^2 + \frac{\lam_t c}{4}\|A_t(z_t^+-z_t)\|^2 -2\varepsilon_t^+ +\langle r_t^+, z_t-z_t^+\rangle,
\end{align*}
where the last inequality is due to \eqref{def:norm:Bt:Ine:AB}.
Combining the previous inequality, the inequality $a b \le (a^2 + b^2)/2$ with $(a,b)=(\sqrt{2}\|r_t^+\|, (1/\sqrt{2})\|z_t^+-z_t\|)$, and the condition on the error $(r_t^+, \varepsilon_t^+)$ as in \eqref{eq:ref:lemma:A1}, we have
\begin{align}
\label{tes:SA:line:search:Ine}
 {\cal L}_c(z_{<t}^+,z_{t}, z_{>t}; p)&-{\cal L}_c(z_{<t}^+,z_t^+, ,z_{>t}; p) \nonumber\\ 
&\ge  \frac{1}{2\lam_t^+}\|z_t^+-z_t\|^2 + \frac{c}{4} \|A_t (z_t^+-z_t)\|^2  -\frac{1}{\lam_t^+}\left(\| \sqrt{2} r_t^+\| \left\|\frac1{\sqrt{2}} (z_t^+-z_t)\right\|+ 2\varepsilon_t^+\right) \nonumber\\
&\ge \frac{1}{2\lam_t^+}\|z_t^+-z_t\|^2 + \frac{c}{4} \|A_t (z_t^+-z_t)\|^2  -\frac{1}{\lam_t^+} \left(  \|r_t^+\|^2 + \frac14 \|z_t^+-z_t\|^2+ 2\varepsilon_t^+ \right) \nonumber\\
&\overset{\eqref{eq:ref:lemma:A1}}\ge \frac{1}{2\lam_t^+}\|z_t^+-z_t\|^2 + \frac{c}{4} \|A_t (z_t^+-z_t)\|^2  -\frac{1}{\lam_t^+} \left(  \frac14 +\frac18  \right) \|z_t^+-z_t\|^2 \nonumber\\
&= \frac{1}{8\lam_t^+} \|z_t^+-z_t\|^2 + \frac{c}{4} \|A_t(z_t^+-z_t)\|^2, \nonumber
\end{align}
which concludes that \eqref{eq:test:adap:new} holds.
We have thus proved that the claim holds, and hence that b) holds.

(c) Assume for the sake of contradiction that
there exists a loop iteration $j$ such that 
$j-1 \ge \lceil \log_2(1+4m_t\lam_t) \rceil$,
and hence $j \ge 2$.
In view of line \ref{eq:test:ABIPP} of {\ABIPP}, the stepsize $\lam_t^+$ at the beginning of the $(j-1)$-th loop iteration satisfies
\[
\lam_t^+ = \frac{\lam_t}{2^{j-2}} =
\frac{2\lam_t}{2^{j-1}} \le \frac{1}{2m_t},
\]
where the inequality follows from the fact that $j-1\geq \log_2(4m_t\lam_t)$. In view of the claim made at the beginning of b), it then follows that the loop ends at the $(j-1)$-iteration, a conclusion that contradicts the  assumption that  $j$ is a loop iteration.
The conclusion that $\lam_t^+ \ge \min\{ \lam_t, 1/(4m_t)\}$
follows from the claim at the beginning of b) and the fact that 
$\lam_t^+$ is halved at the loop iterations for which 
\eqref{eq:test:adap:new} does not hold.
\end{proof}

\vspace{1em}

The following result,
which is a more general version of Proposition \ref{main:subroutine:1},
describes the main properties of the quadruple
$(z^+,v^+,\delta_+,\lam^+)$ output 
by {\ABIPP}.

\begin{proposition}\label{prop:ABIPP.output}
Assume that {\ABIPP} with input $(z,p,\lam,c)
\in {\cal H} \times A(\mathbb{R}^n) 
\times \R^B_{++}\times \R_{++}$ is well-defined and let
$(z^+,v^+,\delta_+, \lam^+)={\ABIPP}(z,p,\lam, c)$. 
Then the following statements hold: 
\begin{itemize}
\item[(a)] there holds
\begin{equation}\label{eq:adapt.Delta.Lc.ineq}
\Delta {\cal L}_c := {\cal L}_c(z;p) - {\cal L}_c(z^+;p) \ge \frac{1}{8}\sum_{t =1}^B \frac{\|z_t^+-z_t\|^2}{\lam_t^+} + \frac{c}{4}\sum_{t =1}^B \|A_t (z_t^+-z_t)\|^2;
\end{equation}

\item[(b)] the quadruple $(z^+,v^+,\delta_+, \lam^+)$ satisfies
\begin{gather}
\begin{gathered}\label{lemma:norm:residual:Ine.2}
v^+  \in \nabla f (z^+)+ \partial_{\delta_+} \Psi(z^+) +
A^*[p+c(Az^+-b)] ,\\
\|v^+\|^2 +\delta_+  
\le \left[\tilde\sigma_1(\lam_{\min}^+,\lam_{\max}^+) + c \sigma_2\right]
  \Big[ {\cal L}_c(z;p) - {\cal L}_c(z^+;p)\Big],
\end{gathered}
\end{gather}
where $\chitwo$ is as in  \eqref{eq:sigma1&2},
\begin{align}\label{def.lambda.min.max}
\lam_{\min}^+ = \min_{1\leq t\leq B}\{ \lam_t^+\}, \quad\lam_{\max}^+ = \max_{1\leq t\leq B}\{ \lam_t^+\},
\end{align}
and
\begin{align}\label{def.sigm.lam}
    \tilde \sigma_1(\lam_{\min}^+,\lam_{\max}^+) = 48\lam_{\max}^+\|L\|^2+50(\lam_{\min}^+)^{-1}+1.
\end{align}
\end{itemize}
\end{proposition}

\begin{proof}
(a) 
We first observe that Lemma \ref{Lemma:ABIPP}(c) implies that the loop consisting of lines \ref{output:step:2:ABIPP} to \ref{eq:test:ABIPP} terminates with a pair $(z_t^+,\lam_t^+)$ satisfying  \eqref{eq:test:adap:new}.
Hence, by summing \eqref{eq:test:adap:new} from $t=1$ to $t=B$ we conclude that  \eqref{eq:adapt.Delta.Lc.ineq} holds.

\noindent
(b) We first prove the inclusion in \eqref{lemma:norm:residual:Ine.2}, and to ease notation we let $p^+=p+c(Az^+-b)$. Using \eqref{prop1:subd} with $(\varepsilon, \beta)=(\varepsilon_t^+, \lambda_t^+)$,
we easily see that 
\eqref{eq:ref:lemma:A1} implies that
\begin{align}\label{eq:COA1}
\frac{r_t^+}{\lambda_t^+}
&\overset{\eqref{eq:ref:lemma:A1}}\in \nabla_{z_{t}^+}f(z_{< t}^+,z_t^+, z_{>t})+A_{t}^{*}\left[ p+c[A(z_{< t}^+,z_t^+,z_{>t})-b]\right]+\frac{1}{\lambda_t^+}(z_t-z_t^+)+\partial_{(\varepsilon_t^+/\lambda_t^+)}\Psi_t(z_t^+)\nonumber \\
& =\nabla_{z_{t}^+}f(z_{< t}^+,z_t^+, z_{>t})+A_{t}^{*}\left(p^+ - c\sum_{s=t+1}^{ B}A_s (z_s^+-z_s)\right)+\frac{1}{\lambda_t^+}(z_t-z_t^+)+\partial_{(\varepsilon_t^+/\lambda_t^+)}\Psi_t(z_t^+), \nonumber
\end{align}
for every $t\in \{1,\ldots,B\}$. Rearranging the above inclusion and using the definition of $v_t^+$ (see STEP 2 of {\ABIPP}), we see that for every $t\in \{1,\ldots,B\}$,
\[
v_t^+ \in \nabla_{z_{t}^+}f(z^+)+\partial_{(\varepsilon_t^+/\lambda_t^+)}\Psi_t(z_t^+)+A_t^*p^+.
\]
Now using \eqref{prop2:subd} with $(\varepsilon,\varepsilon_t)=(\delta_+, \varepsilon_t^+/\lam_t^+)$ for every $t\in \{1,\ldots,B\}$, and $\delta_+=(\varepsilon_1^+/\lam_1^+) + \ldots +  (\varepsilon_B^+/\lam_B^+)$ (see STEP 2 of {\ABIPP}), we have that
\[
\partial_{\delta_+}\Psi(z^+) \supset  \partial_{(\varepsilon_1^+/\lambda_1^+)}h_1(z_1^+) \times \ldots \times
\partial_{(\varepsilon_B^+/\lambda_B^+)}h_B(z_B^+),
\]
and we conclude that
the inclusion in \eqref{lemma:norm:residual:Ine.2} holds.

We now prove the inequality in \eqref{lemma:norm:residual:Ine.2}. Using \eqref{eq:ref:lemma:A1}, \eqref{eq:adapt.Delta.Lc.ineq},  and that $1/\lam_t^+\leq (\lam_{\min}^+)^{-1}$ due to \eqref{def.lambda.min.max}, we have
\begin{align*}
\sum_{t=1}^B \left( 2\frac{\|r_t^+\|^2}{(\lam_t^+)^2}  + \frac{\varepsilon_t^+}{\lam_t^+} \right)
&\overset{\eqref{def.lambda.min.max}}\le (2(\lam_{\min}^+)^{-1}+1) \sum_{t=1}^B \left( \frac{\|r_t^+\|^2 + \varepsilon_t^+}{\lam_t^+}\right)\\
&\overset{\eqref{eq:ref:lemma:A1}}\le (2(\lam_{\min}^+)^{-1}+1)  \sum_{t=1}^B \left( \frac{\|z_t^+-z_t\|^2}{8\lam_t^+}  \right) \overset{\eqref{eq:adapt.Delta.Lc.ineq}}\le  (2(\lam_{\min}^+)^{-1}+1) \Delta {\cal L}_c.
\end{align*}
Defining $D_{t}:=\|v_t^+-r_t^+/\lam_t^+\|^2$, using the previous inequality, the definition of $\delta_+$ (see STEP 2 of {\ABIPP}), and that $\|a+b\|^2\leq 2\|a\|^2+2\|b\|^2$, for any $a,b\in \R^n$, we have
\begin{align}\label{eq:suff.decr.inter}
\|v^+\|^2 &+\delta_+  =
\sum_{t=1}^B \left( \|v_t^+\|^2 + \frac{\varepsilon_t^+}{\lam_t^+}  \right) 
\le 
\sum_{t=1}^B \left( 2 D_{t} + 2\frac{\|r_t^+\|^2}{(\lam_t^+)^2}  + \frac{\varepsilon_t^+}{\lam_t^+}
\right) \leq 2 \sum_{t=1}^B  D_{t} +  (2(\lam_{\min}^+)^{-1}+1) \Delta {\cal L}_c.
\end{align}
We will now bound $\sum_{t=1}^BD_{t}$. Using \eqref{eq:adapt.Delta.Lc.ineq} and that $\lam_s^+ \leq \lam_{\max}^+$ due to \eqref{def.lambda.min.max}, we have
\begin{align}\label{eq:bound:vi:interm1}
L_t^2 \|z_{>t}^+-z_{>t}\|^2 
=
L_t^2 \sum_{s =t+1}^B \|z_s^+-z_s\|^2
\overset{\eqref{def.lambda.min.max}}\le
L_t^2 \left( 
\lam_{\max}^+ \sum_{s =t+1}^B \frac{\|z_s^+-z_s\|^2}{\lam_s^+}
\right)   
\overset{\eqref{eq:adapt.Delta.Lc.ineq}}\le 8 \lam_{\max}^+ L_t^2 \Delta {\cal L}_c.
\end{align}
Moreover, it follows from the definitions of $D_{t}$ given above and $v_t$ (see STEP 2 of {\ABIPP}), the Cauchy-Schwarz inequality, and assumption (A4), that
\begin{align}\label{eq:bound:vi:first:Ine}
D_{t} &=\left\|v_{t}^+-\frac{r_t^+}{\lam_t^+}\right\|^2  =\left\|\nabla_{z_{t}^+}f(z_{\leq t}^+, z_{>t}^+)-\nabla_{z_{t}^+}f(z_{\leq t}^+, z_{>t})+A_{t}\left(c\sum_{s=t+1}^{ B}A_{s}(z_s^+-z_s)\right)-\frac{(z_s^+-z_s)}{\lambda_t^+}\right\|^2\nonumber\\
&\leq 3\left( \|\nabla_{z_{t}^+}f(z_{\leq t}^+, z_{>t}^+)-\nabla_{z_{t}^+}f(z_{\leq t}^+, z_{>t})\|^2+ \left( c\|A_t\| \sum_{s=t+1}^{B} \|A_{s} (z_s^+-z_s) \| \right)^2 +\frac{\|z_s^+-z_s\|^2}{(\lambda_t^+)^2}\right)\nonumber\\
&\overset{\eqref{eq:lipschitz_x}}\leq 3\left( L_t^2 \|z_{>t}^+-z_{>t}\|^2+c^2\|A_t\|^2 (B-t)\sum_{s=t+1}^{B}\|A_{s}(z_s^+-z_s)\|^2+(\lam_{\min}^+)^{-1} \frac{\|z_s^+-z_s\|^2}{\lambda_t^+} \right) \nonumber\\
& \overset{\eqref{eq:adapt.Delta.Lc.ineq},\eqref{eq:bound:vi:interm1}}\le
3  \left[  \left( 8 \lam_{\max}^+ L_t^2 
+  4 c \|A_t\|^2 B \right)  \Delta {\cal L}_c + (\lam_{\min}^+)^{-1} \frac{\|z_s^+-z_s\|^2}{\lambda_t^+} \right] . \nonumber
\end{align}
Summing up the previous inequality from $t=1$ to $t=B$, using the definitions of $\|L\|$ and $\|A\|_\dagger$  as in \eqref{eq:block_norm}, and using inequality \eqref{eq:adapt.Delta.Lc.ineq}, we have 
\begin{align*}
    \sum_{t=1}^B D_{t} & \overset{\eqref{eq:block_norm}}
\le  
\left[ 12 c \|A\|^2_\dagger B + 24 \lam_{\max}^+\|L\|^2   \right] \Delta {\cal L}_c + 3 (\lam_{\min}^+)^{-1} \sum_{t=1}^B \frac{\|z_t^+-z_t\|^2}{\lambda_t^+} \\
&
\overset{\eqref{eq:adapt.Delta.Lc.ineq}}\leq 12\left[  c \|A\|^2_\dagger B+ 2\lam_{\max}^+\|L\|^2+2(\lam_{\min}^+)^{-1} \right]\Delta {\cal L}_c. 
\end{align*}
Inequality \eqref{lemma:norm:residual:Ine.2} now follows by combining \eqref{eq:suff.decr.inter} with the
 previous inequality, and using the definitions of $\tilde \chione(\lam_{\min}^+,\lam_{\max}^+)$ and $\chitwo$ as in \eqref{def.sigm.lam} and \eqref{eq:sigma1&2}, respectively.
\end{proof}

\vspace{1em}
We now use Lemma \ref{Lemma:ABIPP} and Proposition \ref{prop:ABIPP.output} to show that Proposition \ref{main:subroutine:1} holds.

\vspace{1em}
\noindent
{\bf Proof of Proposition \ref{main:subroutine:1}:}
The assumption that
$\lam_t=1/(2m_t)$ for every $t\in\{1,
\ldots,B\}$,  the definitions in \eqref{eq:max.m.min.m}, \eqref{def.lambda.min.max},
and Lemma~\ref{Lemma:ABIPP}(b), imply that
\[
\lam_{\min}^+ =  \min_{1\leq t\leq B}\{ \lam_t\} = \min_{1\leq t\leq B} \left\{ \frac{1}{2m_t} \right \} = \frac{1}{2\overline{m}}, \quad \lam_{\max}^+ =\max_{1\leq t\leq B}\{ \lam_t\}=\lam_{\max}=\max_{1\leq t\leq B} \left\{ \frac{1}{2m_t} \right \} = \frac{1}{2\underline m},
\]
and hence that the conclusion of 
holds in view of the definitions of
$\sigma_1$ in \eqref{eq:sigma1&2} and $\tilde\sigma_1$ in \eqref{def.lambda.min.max}. \qedflush

\vspace{1em}

In the remaining of this subsection, we argue that similar complexity results obtained for {\SA} can also be derived for its adaptive stepsize analog that, instead of
invoking {\SBIPP} in its line \ref{call:SBIPP}, calls its adaptive counterpart {\ABIPP}. We refer to this modified {\SA} as the adaptive {\SA}.

We start with some remarks about the stepsizes $\{\lam_t^k\}$ generated by the adaptive {\SA}. First, a very simple induction applied to Lemma~\ref{Lemma:ABIPP}(b) shows that the stepsize vector $\lam^k=(\lam_1^k,\ldots,\lam_B^k)$ output by the $k$-th {\ABIPP} call within the adaptive {\SA}
satisfies 
\[
\min_{1\leq t\leq B}\{ \lam_t^k\} \ge  \min_{1\leq t\leq B} \left\{ \lam_t^0 , \frac{1}{4m_t} \right \} =: \underline{\lam},  \qquad
\max_{1\leq t\leq B}\{ \lam_t^k\}\le \max_{1\leq t\leq B} \left\{\lam_t^0, \frac{1}{2m_t} \right \} =\overline \lam.
\]
Moreover, it follows from Proposition~\ref{prop:ABIPP.output} with $(z,p,\lam,c)=(y^{i-1},q^{i-1},\lam^{i-1},c)$, the fact that 
$(y^i,v^i,\delta_i,\lam^i)={\ABIPP}(y^{i-1},q^{i-1},\lam^{i-1},c)$, and the above two bounds, that
the inclusion \eqref{eq.SA.stat.inclusion} holds and
\begin{gather}\label{eq.def.adaptive.stepsize}
\begin{gathered}
\|v^i\|^2 +\delta_i  
\le \left[\tilde \sigma_1 + c \sigma_2\right]
  \Big[ {\cal L}_c(y^{i-1};q^{i-1}) - {\cal L}_c(y^i;q^{i-1})\Big], 
\end{gathered}
\end{gather}
where $\tilde \sigma_1 := \tilde 
\sigma_1(\underline{\lam}, \bar \lam)$ and
$\tilde \sigma_1(\cdot,\cdot)$ is as in \eqref{def.sigm.lam}. Using the observation that all the complexity results for {\SA} were derived using~\eqref{claim:bounded:iterations}, one can similarly  obtain complexity results for  the adaptive {\SA} (and hence adaptive {\DA})   using \eqref{eq.def.adaptive.stepsize} and a similar reasoning.
For example, the complexity of the adaptive {\SA} is
\begin{align}\label{thm:iter.complexity:SADMM.adaptive}
 \left(\frac{\tilde \sigma_1+c\sigma_2}{\rho^2}\right)\Gamma(y^0,q^0; c) + 1.
\end{align}
Before ending this section, we consider the case where the function $f$ in
\eqref{initial.problem} is separable. In this case, the constant $L$ defined in~\eqref{eq:block_norm} is zero and hence $\tilde \sigma_1$ reduces to
\begin{equation}\label{eq:def.sigma1.adaptive}
\tilde \sigma_1= 50\underline{\lam}^{-1}+1.
\end{equation}
By examining \eqref{thm:iter.complexity:SADMM.adaptive} and \eqref{eq:def.sigma1.adaptive}, we see that the larger $\underline \lam$ is, the smaller the complexity bound 
\eqref{thm:iter.complexity:SADMM.adaptive} becomes.
It can be easily  seen that 
an alternative bound on the number of iterations is \eqref{thm:iter.complexity:SADMM.adaptive}
with $\tilde \sigma_1$ replaced by
$1 + 50 (\theta_k)^{-1}$ where
$\theta_k = \min \{ \lam_t^k : k=1,\ldots,K, \ t=1,\ldots,B\}$ and
$K$ is the last iteration of the adaptive {\SA}.
This observation thus provides a reasonable explanation for why  the 
 policy of choosing all the initial stepsizes $\{\lam^0_t\}_{t=1}^B$ large and gradually decreasing them as needed is superior from a practical point of view than choosing all the stepsizes constant, and hence for why the adaptive {\SA} performs better than its
the constant stepsize version.



     

\section{Numerical Experiments}\label{sec:numerical}

This section showcases the numerical performance of {\DA} on two linearly and box constrained, non-convex, quadratic programming problems. Subsection \ref{sub:numerics.Problem1} summarizes the performance of {\DA} on a distributed variant of our experimental problem, while Subsection \ref{sub:numerics.Problem2} focuses on a non-distributable variant. The distributed variant employs a small number of high-dimensional blocks while the non-distributable variant conversely has a large number of uni-dimensional blocks. These two proof-of-concept experiments indicate that {\DA} may not only substantially outperform the relevant benchmarking methods in practice, but also be relatively robust to the relationship between block counts and sizes.

All experiments were implemented and executed in  MATLAB 2021b and run on a macOS machine with a 1.7 GHz Quad-Core Intel processor, and 8 GB of memory.

\subsection{Distributed Quadratic Programming Problem}\label{sub:numerics.Problem1}

This subsection studies the performance of {\DA} for finding stationary points of a  box-constrained, nonconvex block distributed quadratic programming problem with $B$ blocks (DQP).  

The $B$-block DQP is formulated as
\begin{align}\label{eq:num.problem1}
\min_{(x_{1},\ldots,x_{B})\in\mathbb{R}^{Bn}}\  & -\sum_{i=1}^{B-1}\left[\frac{\alpha_{i}}{2}\|x_{i}\|^{2}+\left\langle x_{i},\beta_{i}\right\rangle \right]\nonumber\\
\text{s.t.}\  & \|x\|_{\infty}\leq \omega\\
& x_i-x_B=0\quad \text{ for }i=1,\ldots,B-1
\end{align}
where $\omega > 0$, $n\in\mathbb{N}$, $\{\alpha_i\}_{i=1}^{B-1}\subseteq [0,1]$, and $\{\beta_i\}_{i=1}^{B-1}\subseteq [0,1]^n$. It is not difficult to see that DQP fits the template of \eqref{initial.problem}. The smooth component is taken to be
    \[
f(x)= -\sum_{i=1}^{B-1}\left[\frac{\alpha_{i}}{2}\|x_{i}\|^{2}+\left\langle x_{i},\beta_{i}\right\rangle \right].
\]
The non-smooth function $h_i$ is set to the indicator of the set $\{x\in \R^n~:~ \|x_i\|_\infty\leq \omega\}$ for $1\leq i\leq B$. For $1\leq i\leq B-1$, we take $A_i\in\mathbb{R}^{n\times Bn}$ to be the operator which includes $A_i\in\mathbb{R}^n$ into the $i$-th block of  $\mathbb{R}^{Bn}$, i.e. 
    \[
    A_i=\begin{bmatrix} 0_{(i-1)n\times n}\\ I_{n\times n}\\ 0_{(B-i)n\times n}\end{bmatrix}
    \]
where $0_{j\times k}$ denotes a $j\times k$ zero matrix. The matrix $A_B\in\mathbb{R}^{n\times Bn}$ is defined by the action $A_Bx=(-x,\ldots,-x)^\top$.

We shall now outline how we conducted our DQP experiments. The number of blocks, $B$, for each experiment was set to $3$, while for the block-size, $n$, the dimensions $n=10,20,100,5000$ were considered. For each setting of $n$, we ran experiments where $\omega=10^1,10^3,10^5,10^7,10^9$. The values of $\{\alpha_i\}_{i=1}^{B-1}\subseteq [0,1]$, and $\{\beta_i\}_{i=1}^{B-1}\subseteq [0,1]^n$ were sampled uniformly at random. To generate $b$, we sampled $x_b\in [-\omega,\omega]^{Bn}$ uniformly at random, then set $b=Ax_b$. The initial iterate $x_0$ was also selected uniformly at random from $[-\omega,\omega]^{Bn}$. 

For this problem,  {\DA} was applied with $c_0=1$, $C=1$, $\alpha=10^{-2}$, $p^0=\mathbf{0}$, and each block's initial stepsize set to $10$, i.e., $\gamma^0_i=10$ for $1\leq i\leq B$. To provide an adequate benchmark for {\DA}, we compared its performance against two instances of the method from \cite{KongMonteiro2024} and three instances of the method from \cite{sun2021dual}. The method of \cite{KongMonteiro2024} was deployed with two different choices of $(\theta,\chi)$: $(0,1)$ and $(1/2,1/18)$. We call these two instances DP1 and DP2, respectively. Both DP1 and DP2 set $(\lam, c_1)=(1/2, 1)$. The method of \cite{sun2021dual} was deployed with three different settings of the penalty parameter $\rho$: $0.1$, $1.0$, and $10.0$. We call the resultant instances SD1, SD2, and SD3, respectively. Moreover, all three instances make the parameter selections $(\omega,\theta,\tau)=(4,2,1)$ and $(M_{\Psi},K_k,J_h,L_h)=(4\gamma,1,1,0)$ in accordance with \cite[Section 5.1]{sun2021dual}. We reiterate that \cite{KongMonteiro2024} provides no convergence guarantees for the pragmatic choice of $(\theta,\chi)=(0,1)$. All executed algorithms were run for a maximum of $500,000$ iterations. Any algorithm that met this limit took at least $10$ milliseconds to complete.



\begin{table}[htb!]
\caption{Performance for all algorithms applied to the DQP problem \eqref{eq:num.problem1}, with $B=3$, $C=1$ and $\alpha=10^{-2}$ for different pair of values $(n,\omega)$. The iteration and time columns record the number of iterations and time in seconds to find a $(10^{-5},10^{-5})$-stationary point.}
\centering
\resizebox{\textwidth}{!}{\begin{tabular}{cccccccccccccc}
\toprule
& & \multicolumn{6}{c}{Iteration} & \multicolumn{6}{c}{Time (ms)}  \\
\cmidrule(lr){3-8} \cmidrule(lr){9-14}
$n$ & $\omega$ & AD & DP1 & DP2 & SD1 & SD2 & SD3 & AD & DP1 & DP2 & SD1 & SD2 & SD3 \\
\midrule
10 & $10^{1}$ & \textbf{18} & 76 & 83 & 427 & 223 & 976 & \textbf{1.592} & 5.402 & 4.291 & 28.192 & 13.881 & 60.184  \\
10 & $10^{3}$ & \textbf{34} & 228 & 232 & 569 & 399 & 1855 & \textbf{2.259} & 11.752 & 11.858 & 65.049 & 30.209 & 119.417  \\
10 & $10^{5}$ & \textbf{50} & 385 & 385 & * & 581 & 2778 & \textbf{3.228} & 13.374 & 13.004 & * & 35.368 & 168.964  \\
10 & $10^{7}$ & \textbf{66} & 541 & 537 & * & * & 3701 & \textbf{4.419} & 18.706 & 18.363 & * & * & 239.235  \\
10 & $10^{9}$ & \textbf{81} & 697 & 689 & * & * & 4625 & \textbf{5.866} & 24.540 & 24.237 & * & * & 323.790  \\
20 & $10^{1}$ & \textbf{22} & 62 & 68 & 433 & 298 & 1261 & \textbf{1.538} & 2.520 & 2.484 & 27.928 & 19.164 & 80.375  \\
20 & $10^{3}$ & \textbf{44} & 166 & 171 & * & 498 & 2304 & \textbf{2.560} & 6.484 & 6.153 & * & 31.821 & 152.866  \\
20 & $10^{5}$ & \textbf{65} & 273 & 275 & * & 700 & 3347 & \textbf{4.213} & 10.548 & 9.879 & * & 45.812 & 223.235  \\
20 & $10^{7}$ & \textbf{84} & 379 & 379 & * & * & 4383 & \textbf{4.684} & 13.961 & 14.009 & * & * & 290.884  \\
20 & $10^{9}$ & \textbf{103} & 485 & 483 & * & * & 5418 & \textbf{5.629} & 17.635 & 17.393 & * & * & 365.368  \\
100 & $10^{1}$ & \textbf{20} & 40 & 46 & * & 433 & 6231 & 2.072 & \textbf{1.820} & 1.831 & * & 28.705 & 420.276  \\
100 & $10^{3}$ & \textbf{33} & 78 & 77 & * & 695 & 9444 & \textbf{2.132} & 3.013 & 2.871 & * & 44.640 & 617.636  \\
100 & $10^{5}$ & \textbf{45} & 116 & 107 & * & * & 12664 & \textbf{3.148} & 4.545 & 4.041 & * & * & 898.522  \\
100 & $10^{7}$ & \textbf{57} & 155 & 137 & * & * & 15876 & \textbf{3.736} & 5.844 & 5.465 & * & * & 1051.830  \\
100 & $10^{9}$ & \textbf{68} & 193 & 167 & * & * & 19087 & \textbf{4.841} & 7.629 & 6.436 & * & * & 1267.924  \\
5000 & $10^{1}$ & \textbf{25} & 121 & 125 & * & 646 & 2257 & \textbf{13.733} & 26.455 & 27.511 & * & 206.646 & 861.279  \\
5000 & $10^{3}$ & \textbf{37} & 221 & 223 & * & 851 & 3324 & \textbf{20.084} & 52.456 & 50.828 & * & 282.999 & 1264.225  \\
5000 & $10^{5}$ & \textbf{49} & 321 & 321 & * & * & 4390 & \textbf{27.591} & 72.829 & 72.810 & * & * & 1692.375  \\
5000 & $10^{7}$ & \textbf{61} & 422 & 419 & * & * & 5449 & \textbf{32.080} & 96.010 & 97.450 & * & * & 1968.163  \\
5000 & $10^{9}$ & \textbf{72} & 522 & 517 & * & * & 6507 & \textbf{41.682} & 118.872 & 118.632 & * & * & 2440.377\\
\hline\multicolumn{14}{l}{\footnotesize\textit{Bolded values equal to the best algorithm according to iteration count or time.}}\\ 
\multicolumn{14}{l}{\footnotesize\textit{* indicates the algorithm failed to find a  stationary point meeting the tolerances by the 500,000th iteration.}}\\
\bottomrule
\end{tabular}}
\label{tab:numeric.table1}
\end{table}

Table \ref{tab:numeric.table1}, the record of the performance of all algorithms on this experimental problem, lays bare the superior performance of {\DA}. In this table, we label {\DA} as AD for the sake of concision. In terms of iterations, {\DA} outperforms all other algorithms for all settings of $B$ and $m$. Along the dimension of time, {\DA} is faster than all algorithms, for all settings of $n$ and $\omega$, except DP1 when $n=100$ and $\omega=10$.

\newpage
\subsection{Nonconvex QP with Box Constraints}\label{sub:numerics.Problem2}


In this subsection, we evaluate the performance of {\DA} for solving a general nonconvex quadratic problem with box constraints (QP-BC). The QP-BC problem is formulated as
\begin{equation}\label{eq:num.problem2}
\min_{\|x\|_\infty\leq \omega } \left\{f(x) := \frac{1}{2}\inner{x}{Px} + \inner{r}{x}: Ax=b \right\}.    
\end{equation}
where $P\in \R^{B\times B}$ is negative definite, $A\in \R^{m\times B}$, $r,b\in \R^m \times\R^m$ and $\omega\in \R_{++}$. As for the previous problem, it is not difficult to check that QP-BC fits within the template of \eqref{initial.problem}. For this problem, we take our blocks to be single coordinates. Consequently, each column of $A$ gives rise to a $A_i$ matrix. The non-smooth components of the objective are again picked to be the indicator functions of the sets $\{x_i\in \R: |x_i|\le \omega\}$ for $i\in \{1,\ldots, B\}$.

We now describe how we orchestrated our QP-BC experiments. In all instances, $\omega=1$. To generate $\tilde r\in \R^m$, $\tilde P\in\R^{B\times B}$ and $\tilde A\in\R^{m\times B}$, we started by generating a diagonal matrix $D\in\R^{B\times B}$ whose entries are selected uniformly at random in $[1,1000]$. Next, we generated $\tilde r\in [-1,1]^m$, $\tilde P\in[-1,1]^{B\times B}$ negative definite, and $\tilde A\in[-1,1]^{m\times B}$ uniformly at random. Finally, we set $P=D\tilde P D$, $A=\tilde A D$, and $r=D\tilde r$. The vector $b\in\R^m$ was set as $b=Ax_b,$ where $x_b$ is a uniformly at random selected vector in $[-1,1]^B$. The initial starting point $x^0$ was chosen in this same fashion.

For this problem,  three instances of {\DA}, referred to as AD1, AD2, and AD3, were applied with $C=1$, $ =10^{-2}$, $p^0=\mathbf{0}$, and each block's initial stepsize set to $10$, i.e. $\gamma^0_i=1000$ for $1\leq i\leq B$. The three methods differ only in their choice of initial penalty parameter $c_0$: $c_0=10$ for AD1, $c_0=1$ for AD2, and $c_0=.1$ for AD3.
The benchmarking algorithms for this experiment were three instances of the method from \cite{KongMonteiro2024}, which we refer to as DP1, DP2 and DP3. Like the three instances of {\DA}, these instances differ only in their choice of $c_0$: $c_0=10$ in DP1, $c_0=1$ in DP2 and $c_0=0.1$ in DP3. Each of these three methods were applied with $(\theta,\chi)=(0,1)$. Yet again, we remind the reader that \cite{KongMonteiro2024} provides no convergence guarantees for this choice of $(\theta,\chi)$.  To ensure timely execution of all algorithms, each algorithm terminated upon meeting a 500,000 iteration limit or the discovery of an approximate stationary triple $(x^+, p^+, v^+)$ satisfying the relative error criterion
\[
v^+ \in \nabla f(x^+) + \partial \Psi(x^+) + A^*p^+, \quad \frac{\|v^+\|}{1+\|\nabla f(x^0)\|}\leq \rho , \quad \frac{\|Ax^+-b\|}{1+\|Ax^0-b\|}\leq \eta.
\]
for $\rho=\eta=10^{-5}$.


\begin{table}[htb!]
\caption{Performance for all algorithms applied to the QP-BC Problem \eqref{eq:num.problem2}, with $C=1$ and $\alpha=10^{-2}$, for different pair of values $(B,\omega)$. The iteration and time columns record the number of iterations and time in seconds to find a stationary point satisfying the relative error condition with $(\rho,\eta)=(10^{-5},10^{-5})$.}
\centering
\resizebox{\textwidth}{!}{\begin{tabular}{cccccccccccccc}
\toprule
& & \multicolumn{6}{c}{Iteration} & \multicolumn{6}{c}{Time (sec)}  \\
\cmidrule(lr){3-8} \cmidrule(lr){9-14}
$B$ & $m$ & AD1 & AD2 & AD3 & DP1 & DP2 & DP3 & AD1 & AD2 & AD3 & DP1 & DP2 & DP3  \\
\midrule
10 & 1 & 44 & \textbf{23} & 33 & 3554 & 3560 & 3532 & 0.067 & 0.010 & \textbf{0.008} & 0.296 & 0.286 & 0.275  \\
10 & 2 & 23 & \textbf{19} & 37 & 1355 & 1282 & 1395 & 0.025 & \textbf{0.005} & 0.007 & 0.128 & 0.111 & 0.116  \\
10 & 5 & \textbf{1280} & 2421 & 1469 & * & * & * & \textbf{0.162} & 0.287 & 0.171 & * & * & *  \\
20 & 1 & \textbf{23} & 30 & 30 & 803 & 296 & 417 & 0.021 & 0.009 & \textbf{0.008} & 0.148 & 0.052 & 0.075  \\
20 & 2 & 87 & \textbf{44} & 89 & 297 & 2233 & * & 0.032 & \textbf{0.014} & 0.026 & 0.063 & 0.441 & *  \\
20 & 5 & 147 & 144 & \textbf{114} & 1862 & 6210 & * & 0.036 & 0.034 & \textbf{0.027} & 0.333 & 1.088 & *  \\
20 & 10 & 682 & 1105 & \textbf{550} & 847 & * & * & 0.168 & 0.267 & \textbf{0.132} & 0.152 & * & *  \\
20 & 15 & \textbf{1286} & 2753 & 3691 & 1808 & * & * & \textbf{0.308} & 0.656 & 0.879 & 0.329 & * & *  \\
50 & 1 & 21 & \textbf{17} & 66 & 327 & 1616 & 1385 & 0.022 & \textbf{0.014} & 0.049 & 0.176 & 0.850 & 0.746  \\
50 & 2 & 219 & \textbf{22} & 66 & 377 & 1180 & 2772 & 0.178 & \textbf{0.019} & 0.055 & 0.243 & 0.760 & 1.704  \\
50 & 5 & 188 & \textbf{123} & 226 & 1880 & 2296 & * & 0.147 & \textbf{0.098} & 0.175 & 1.208 & 1.467 & *  \\
50 & 10 & 462 & \textbf{377} & 1647 & 1456 & 699 & * & 0.352 & \textbf{0.285} & 1.229 & 0.972 & 0.466 & *  \\
50 & 20 & \textbf{1082} & 56530 & 9363 & 2058 & * & * & \textbf{0.842} & 42.686 & 7.063 & 2.173 & * & *  \\
50 & 25 & 1326 & 2361 & 2307 & \textbf{1157} & * & * & \textbf{1.230} & 1.913 & 1.835 & 1.243 & * & *  \\
50 & 30 & 3430 & \textbf{1262} & 2045 & 3981 & * & * & 2.989 & \textbf{1.044} & 1.654 & 4.412 & * & *  \\
100 & 1 & 95 & \textbf{22} & 182 & 2792 & 1476 & 9554 & 0.446 & \textbf{0.084} & 0.572 & 9.545 & 4.887 & 33.713  \\
100 & 2 & 104 & \textbf{32} & 102 & 802 & 1531 & * & 0.429 & \textbf{0.120} & 0.361 & 3.154 & 5.996 & *  \\
100 & 5 & 449 & 256 & \textbf{83} & 4603 & * & * & 1.570 & 0.902 & \textbf{0.295} & 20.776 & * & *  \\
100 & 10 & 1675 & 37263 & \textbf{427} & 3050 & 3281 & * & 5.724 & 124.269 & \textbf{1.429} & 15.528 & 16.771 & *  \\
100 & 25 & 2388 & 12916 & \textbf{2346} & 2687 & * & * & 8.041 & 43.605 & \textbf{7.885} & 15.529 & * & *  \\
100 & 50 & 4596 & 3526 & * & \textbf{3395} & * & * & 16.336 & \textbf{12.488} & * & 23.325 & * & *  \\
100 & 75 & 7070 & 27964 & 123020 & \textbf{4816} & * & * & \textbf{26.387} & 104.134 & 459.301 & 38.100 & * & *  \\
\hline\multicolumn{14}{l}{\footnotesize\textit{Bolded values equal to the best algorithm according to iteration count or time.}}\\ 
\multicolumn{14}{l}{\footnotesize\textit{* indicates the algorithm failed to find a  stationary point meeting the tolerances by the 500,000th iteration.}}\\
\bottomrule
\end{tabular}}
\label{tab:numeric.results.table2}
\end{table}

The results for this experiment, shown in Table \ref{tab:numeric.results.table2}, echo those for its predecessor by again displaying the computational superiority of {\DA}. Measured by iteration count, {\DA} performed better in $87\%$ of the problem instances. In terms of time, {\DA} performed better in $100\%$ of the instances. It is worth mentioning that {\DA} converged for all instances, while the same cannot be said for the DP1, DP2, and DP3 benchmark methods. DP1 converged for $96\%$ instances, DP2 converged for $54\%$ instances, and DP3 converged only for $27\%$ instances. For $m=1,2,5$, our method was at least $10$ times faster in terms of iteration count and time than any DP variant. Notably, we attempted to apply multiple versions of the method from \cite{KongMonteiro2024} with choices of $(\theta,\chi)$ that theoretically should ensure convergence. None of the methods managed to find the desired point within the iteration limit, so we omitted their results from the table.

\section{Concluding Remarks}\label{sec:concluding}


We now discuss some further related research directions. First, the ability of {\DA} to allow for the inexact solution of its block subproblem opens up many possible avenues for application. 
A systematic numerical study of its performance when applied to problems requiring inexact computation would be intriguing.
Second, it would be interesting to develop a {\PD} that performs Lagrange multiplier updates
with 
$(\theta,\chi)=(0,1)$ at every iteration, rather than
just at the last iteration of each epoch. This {\PD} would then be an instance of the class of ADMMs outlined in the Introduction with
$\ell_k=1$. 
The {\PD} of \cite{KongMonteiro2024} satisfies this last property but chooses $(\theta,\chi)$ in a very conservative way, namely, satisfying
\eqref{eq:assumptio:B}.

Finally, we have assumed in this chapter that $\dom \Psi$ is bounded (see assumption (A1)). It would be interesting to extend its analysis  to  the case where ${\cal H}$
is unbounded.


\section{Technical Results for Proof of Lagrange Multipliers}\label{sec:tech.lagrange.multiplier}

This section provides some technical results about convexity and shows that the sequence of Lagrange multipliers generated by {\SA} is bounded.

The first two results,  used to prove Lemma~\ref{lem:qbounds-2},  can be found in  \cite[Lemma B.3]{goncalves2017convergence} and 
\cite[Lemma 3.10]{kong2023iteration}, respectively. 

\begin{lemma}\label{lem:linalg} 
Let $A:\R^n \to \R^l$ be a  nonzero linear operator. Then,
\[
\nu^+_A\|u\|\leq \|A^*u\|,   \quad \forall u \in A(\R^n).
\]
\end{lemma}

\begin{lemma}\label{lem:bound_xiN}
Let $h$ be a function as in (A5).
Then, for every $\delta \ge 0$, $z\in {\mathcal H}$,   and $\xi \in \partial_{\delta} \Psi(z)$, we have
\begin{equation*}\label{bound xi}
\|\xi\|{\rm dist}(u,\partial {\mathcal H}) \le \left[{\rm dist}(u,\partial {\mathcal H})+\|z-u\|\right]M_{\Psi} + \inner{\xi}{z-u}+\delta \quad \forall u \in {\mathcal H}
\end{equation*}
$\partial {\cal H}$ denotes the boundary of ${\cal H}$.
\end{lemma}

The following result, whose statement is in terms of
the $\delta$-subdifferential instead of the classical subdifferential,
is a slight generalization of \cite[Lemma B.3]{sujanani2023adaptive}.
For the sake of completeness, we also include its proof.  

\begin{lemma}\label{lem:qbounds-2}
Assume that
$b \in \R^{l}$,
linear operator $A:\mathbb{R}^n \to \mathbb{R}^l$,
and function $\Psi(\cdot)$,
satisfy assumptions (A2),
(A5) and (A6).  
If $(q^-,\varrho) \in A(\R^n) \times (0,\infty)$ and $(z,q,r,\delta) \in  \dom \Psi \times A(\mathbb{R}^n) \times \mathbb{R}^n\times \R_+$ satisfy
\begin{equation}\label{eq:cond:Lem:A1}
r \in \partial_\delta \Psi(z)+A^{*}q\quad\text{and}\quad 
q=q^-+\varrho(Az-b), 
\end{equation}
then we have 
\begin{equation}\label{q bound-2}
 \|q\|\leq \max\left\{\|q^-\|, \varphi\left(\|r\|+\delta \right) \right\}
\end{equation}
where $M_{\Psi},$ $\bar d>0$, and $D_{\Psi}$, are as in (A5), (A6), and \eqref{def:damH}, respectively, $\nu^+_A$ denotes  the smallest positive singular value of $A$, and 
\begin{equation}\label{eq:Technical.varphi.def}
\varphi(t) := \frac{2D_{\Psi}M_{\Psi}+(2D_{\Psi}+1)t}{\bar d \nu^{+}_A} \quad \forall t
\in \R_+.
\end{equation}
\end{lemma}	
\begin{proof}
 We first claim that
\begin{equation}\label{ineq:aux9001-2}
 \bar{d}\nu_A^{+}\|q\|
\leq  2 D_{\Psi}\left(M_{\Psi} + \|r\| \right)  - \inner{q}{Az-b}+\delta
\end{equation}
holds.
The assumption on $(z,q,r,\delta)$ implies that $r-A^{*}q \in \partial_{\delta} \Psi(z)$. Hence, using the Cauchy-Schwarz inequality, the definitions of $\bar d$ and $\bar x$ in (A6), and
Lemma~\ref{lem:bound_xiN} with $\xi=r-A^{*}q$, and  $u=\bar x$, we have:
 \begin{align}\label{first inequality p-2}
   \bar d\|r-A^{*}q\|-\left[\bar d+\|z-\bar x\|\right]M_{\Psi} &\overset{\eqref{bound xi}}{\leq}  \inner{r-A^{*}q}{z-\bar x}+\delta\leq \|r\| \|z-\bar x\| - \inner{q}{Az-b}+\delta.
 \end{align}
 Now, using the above inequality,
 the triangle inequality, the definition of $D_{\Psi}$ in \eqref{def:damH}, and the facts that $\bar d \leq D_{\Psi}$ and $\|z-\bar x\|\leq D_{\Psi}$, we conclude that:
 \begin{align}\label{second inequality p-2}
 \bar d \|A^*q\| + \inner{q}{Az-b}
 &\overset{\eqref{first inequality p-2}}{\leq} \left[\bar d+\|z-\bar x\|\right]M_{\Psi} + \|r\| \left(D_{\Psi} + \bar d\right)+\delta \leq 2 D_{\Psi}\left(M_{\Psi} + \|r\| \right)+\delta.
 \end{align}
 Noting the assumption that
$q \in A(\R^n)$,
inequality
 \eqref{ineq:aux9001-2} now follows
from the above inequality
and
Lemma~\ref{lem:linalg}.

We now prove \eqref{q bound-2}. 
Relation \eqref{eq:cond:Lem:A1}  implies that $\inner{q}{Az-b}=\|q\|^2/\varrho-\inner{q^-}{q}/\varrho$, and hence that
\begin{equation}\label{q relation-2}
    \bar d \nu^{+}_A\|q\|+\frac{\|q\|^2}{\varrho}\leq 2D_{\Psi}(M_{\Psi}+\|r\|)+\frac{\inner{q^-}{q}}{\varrho}+\delta\leq 2D_{\Psi}(M_{\Psi}+\|r\|)+\frac{\|q\| }{\varrho}\|q^-\|+\delta,
\end{equation}
where the last inequality is due to the Cauchy-Schwarz inequality.
Now, letting
$W$  denote the right hand side of \eqref{q bound-2} and using \eqref{q relation-2},
we conclude that
\begin{equation}\label{prelim q bound-2}
\left(\bar d \nu^+_A+\frac{\|q\|}{\varrho}    \right)\|q\|\overset{\eqref{q relation-2}}{\leq} \left(\frac{2D_{\Psi}(M_{\Psi}+\|r\|)+\delta}{W}+\frac{\|q\| }{\varrho}\right)W\leq \left(\bar d \nu^+_A+\frac{\|q\|}{\varrho}    \right) W,
\end{equation}
and hence that \eqref{q bound-2} holds.
\end{proof}

\vspace{1em}

We conclude this section with a technical result of convexity which is used in the proof of Lemma \ref{Lemma:ABIPP}.
Its proof can be found in \cite[Lemma A1]{melo2023proximal}.

\begin{lemma}\label{conv:result}
		Assume that $\xi>0$, $\psi \in \bConv{n}$ and $Q \in {\cal S}^n_{++}$ are such that
		$\psi - (\xi/2) \|\cdot\|^2_Q$ is convex and let
		$(y,v,\eta) \in \R^n \times \R^n \times  \R_+$ be such that $v\in \partial_\eta \psi(y)$.
		Then, 
		for any $\tau>0$,
			\begin{equation}\label{eq:auxlemA1}
		\psi(u) \ge \psi(y) + \inner{v}{u-y} - (1+\tau^{-1})\eta+ \frac{(1+\tau)^{-1}\xi}{2} \|u-y\|_Q^2  \quad \forall u \in \R^n.
		\end{equation}
	\end{lemma}

\begin{proof}
		Let $\psi_v := \psi-\inner{v}{\cdot}$.
		The assumptions imply that $\psi_v$ has a unique global minimum $\bar y $
		and that
		\begin{equation}\label{ineq:u}
		\psi_v(u)\ge \psi_v(\bar y)+\frac{\xi}{2}\|u-\bar y\|_Q^2  \ge \psi_v(y)-\eta  +\frac{\xi}{2}\|u-\bar y\|_Q^2
		\end{equation}
		for every $u \in \R^n$. 
The above inequalities with $u=y$
 imply that $(\xi/2) \|\bar y-y\|_Q^2 \le \eta$. 
On the other hand, for any $\tilde u,u'\in \R^n$ and $\tau>0$, it holds
\begin{align*}
\|\tilde u+u'\|^2 
&= \|\tilde u\|^2 + \|u'\|^2 + 2 \inner{\frac{1}{\sqrt{\tau}} \tilde u}{\sqrt{\tau} u'} \\
&\le \|\tilde u\|^2 + \|u'\|^2+\frac{1}{\tau}\|\tilde u\|^2+\tau \|u'\|^2\\
&=(1+\tau) \|u'\|^2 + (1+\tau^{-1}) \|\tilde u\|^2    
\end{align*}
which implies in
\begin{align*}
(1+\tau)^{-1} \|\tilde u+ u'\|^2&\leq \|u'\|^2+(1+\tau)^{-1}(1+\tau^{-1}) \|\tilde u\|^2\\
&=\|u'\|^2+\tau^{-1} \|\tilde u\|^2.
\end{align*}
 
Hence, adding and subtracting the term $(\tau^{-1}\xi/2)\|\bar y-y\|_Q^2$ in the right hand side of \eqref{ineq:u}
and using the previous inequality with $\tilde u=u-\bar{y}$ and $u'=\bar{y}-y$, we obtain that
		\begin{align*}
		\psi_v(u)&
		\ge  \psi_v(y)-\eta-\frac{\tau^{-1}\xi}{2}\|\bar y-y\|_Q^2 +\frac{\xi}{2}  \left(  \tau^{-1} \|y-\bar y\|_Q^2+\|u-\bar y\|_Q^2 \right) \nonumber \\
		&\ge \psi_v(y)-(1+\tau^{-1})\eta+\frac{(1+\tau)^{-1}\xi}{2}\|u-y\|_Q^2 \end{align*}
 		for every $u \in \R^n$.		
Hence, \eqref{eq:auxlemA1}  follows  from the above conclusion and the  definition of $\psi_v$.	
\end{proof}
    \noindent

%
%
%
%
%


\pagestyle{plain} 


\chapter{\MakeUppercase{The Inexact Cyclic Block Proximal Gradient Method and Properties of Inexact Proximal Maps}$^{*}$}\label{chapter:jota}

\renewcommand{\thefootnote}{\fnsymbol{footnote}} 
\footnotetext[1]{Reprinted with permission from \cite{MaiaGH24}, Copyright 2024 by Springer Nature.
}
\renewcommand{\thefootnote}{\arabic{footnote}} 


\section{Overview}

    This chapter expands the Cyclic Block Proximal Gradient method for block separable composite minimization by allowing for inexactly computed gradients and pre-conditioned proximal maps. The resultant algorithm, the Inexact Cyclic Block Proximal Gradient (I-CBPG) method, shares the same convergence rate as its exactly computed analogue provided the allowable errors decrease sufficiently quickly or are pre-selected to be sufficiently small. It is provided numerical experiments that showcase the practical computational advantage of I-CBPG for certain fixed tolerances of approximation error and for a dynamically decreasing error tolerance regime in particular. The experimental results indicate that cyclic methods with dynamically decreasing error tolerance regimes can actually outpace their randomized siblings with fixed error tolerance regimes. It is established a tight relationship between inexact pre-conditioned proximal map evaluations and $\delta$-subgradients in our $(\delta,B)$-Second Prox Theorem. This theorem forms the foundation of our convergence analysis and enables us to show that inexact gradient computations can be subsumed within a single unifying framework. 

\section{Introduction}

We propose an Inexact Cyclic Block Proximal Gradient method (I-CBPG) for the block separable composite optimization problem

\begin{equation}\label{eq.problem}
F^*:=\min\left\{F(x):=f(x)+\sum_{i=1}^p \Psi_i\left(  U_i^T  x\right):x\in\mathbb{R}^n\right\}.
\end{equation}
We assume that $f:\mathbb{R}^n\to\mathbb{R}\cup\{\infty\}$ is smooth and convex, the matrices $U_i\in\mathbb{R}^{n\times n_i}$ are chosen such that $(U_1,\ldots,U_p)$ is an $n\times n$ permutation matrix, and each $\Psi_i: \mathbb{R}^{n_i} \to \mathbb{R}\cup\{\infty\}$ is proper, closed, and convex. Problem \eqref{eq.problem} naturally arises in data science whenever regularization is present. Matrix factorization \cite{Schmidt11}, LASSO \cite{Richtarik14}, group LASSO \cite{Qin13, Simon12}, matrix completion \cite{Wright09}, and compressive sensing \cite{Donoho06, Wright09} are but a few such problems.

The class of Block Proximal Gradient (BPG) methods readily exploits block separability to provide iterates that are cheap in terms of memory and computational costs, so they are popular for large-scale versions of problem \eqref{eq.problem} \cite{Beck13, FrongilloReid15, Leventhal10, Nesterov12,Richtarik16}. Often, BPG methods make considerable progress before a single full proximal gradient step can even completely execute. BPG methods principally differ in how they select the block $i$: greedily \cite{Richtarik12}, randomly \cite{Nesterov12}, or cyclically \cite{Beck13}. Cyclic BPG methods, the focus of our work, received their first convergence analysis in \cite{Beck13} which established the benchmark $\mathcal{O}(p/k)$ convergence rate when each $\Psi_i$ is the indicator of a closed and convex set. Later \cite{Shefi16} extended the analysis to account for $\Psi_i$ functions that are more generally proper, closed, and convex. Both \cite{Beck13} and \cite{Shefi16} assume exact computation of gradients and proximal maps and therefore avoid considering the effect of inexactness on their  analyses.

Since gradients and proximal maps are the main ingredients for a broad swath of first-order algorithms, the push to achieve lower iterate costs in large-scale settings has fueled research interest around their inexact computation (``inexactness''). Such inexactness provides a variety of benefits, but from a practical standpoint the most important is the ability to compute approximate updates quickly when a closed-form solution does not exist or would be prohibitively expensive from a computational perspective. The main focus of \cite{Schmidt11} is the convergence of the unaccelerated and accelerated proximal gradient schemes equipped with inexactly computed gradients and proximal maps. The ``inexact oracle" framework of \cite{Devolder14}, which is extended in \cite{Devolder13} and \cite{Dvurechensky16}, analyzes the convergence of common gradient based-methods when gradient or gradient-type mappings are computed inexactly. 

While some prior work has explored the effect of inexact computation on BPG methods, to our knowledge these studies have either restricted their attention to randomized schemes, or required the assumption of strong convexity for cyclic schemes. We summarize these contributions below in Table \ref{tab:comparison}. A central inspiration for this work, \cite{Richtarik14}, considers how inexactly computed proximal maps and gradients affect the randomized BPG method. It further explores the benefits of incorporating pre-conditioning into proximal map evaluations. Specifically, while pre-conditioning provides the benefit of making the problem of step-size selection trivial, this advantage comes at the cost of making closed-form evaluation of the pre-conditioned proximal map no longer possible in general. This lack of a closed-form solution for the pre-conditioned proximal map drives the need for inexact proximal map evaluation. The paper \cite{Hua16} treats the linear convergence of cyclic BPG methods and associated restrictions on the degree of proximal map inexactness under the assumption that the smooth component $f$ is strongly convex. To ensure the aforementioned linear convergence, though, \cite{Hua16} requires that the errors satisfy a restrictive decrease condition. To date, we are unaware of any works considering when $f$ is merely convex. Our primary aim is to fill this apparent deficiency in the literature. To this end, our 
algorithm guarantees sublinear convergence provided that maximum allowable error sizes decrease sublinearly, and our computational approach eliminates the need for often-tedious (if not impossible) explicit checks of subdifferential set membership.

\begin{table}[htb!]
\caption{Comparison of inexact coordinate descent methods by convexity assumption, block selection, and error tolerance type. The convergence rate characterizes the number of block iterations required to find $x$ with $F(x)-F^*<\epsilon$. For the randomized scheme of \cite{Richtarik14}, this convergence occurs in expectation and probability.$^{1}$}
\centering
\begin{tabular}{lc|cc|cc|}
\cline{3-6}
                                           &                                                                   & \multicolumn{2}{c|}{\textbf{\begin{tabular}[c]{@{}c@{}}Convergence Rate \\ by Convexity Type\end{tabular}}} & \multicolumn{2}{c|}{\textbf{Error Type}}               \\ \hline
\multicolumn{1}{|c|}{\textbf{\begin{tabular}[c]{@{}c@{}}Algorithm \\ by Paper\end{tabular}}} & \textbf{\begin{tabular}[c]{@{}c@{}}Cyclic/\\ Random\end{tabular}} & \multicolumn{1}{l|}{\textit{Strongly Convex}}            & \multicolumn{1}{l|}{\textit{Convex}}            & \multicolumn{1}{c|}{\textit{Dynamic}} & \textit{Fixed} \\ \hline
\multicolumn{1}{|c|}{This Paper}        & Cyclic                                                            & \multicolumn{1}{c|}{}                                    & $\mathcal{O}(1/\epsilon)$                                            & \multicolumn{1}{c|}{X}                & X              \\ \hline
\multicolumn{1}{|c|}{\cite{Hua16}}       & Cyclic                                                            & \multicolumn{1}{c|}{$\mathcal{O}\left(\log\left(1/\epsilon\right)\right)$                                           }                                &                                                 & \multicolumn{1}{c|}{X}                &                \\ \hline
\multicolumn{1}{|c|}{\cite{Richtarik14}} & Random                                                            & \multicolumn{1}{c|}{$\mathcal{O}\left(\log\left(1/\epsilon\right)\right)$                                          }                                & $\mathcal{O}(1/\epsilon)$                                                                                       & \multicolumn{1}{c|}{}                 & X              \\ \hline
\end{tabular}
    
    \label{tab:comparison}
\end{table}


\subsection{Chapter's Organization}

We describe
 the key contributions along with the chapter's layout below.
\begin{itemize}
	\item In Section \ref{section:inexact_proximal}, we analyze inexactly computed proximal maps that incorporate pre-conditioning in the sense of \cite{Richtarik14}. Our main theorem, the $(\delta,B)$-Second Prox Theorem (Theorem \ref{thm:second-prox}), generalizes what \cite{Beck17} calls the Second Prox Theorem \cite[Theorem 6.39]{Beck17} that supports the convergence proofs of a broad swath of proximal map-based algorithms. This Theorem's main feature is the tight relationship it expresses between inexact pre-conditioned proximal map evaluations and $\delta$-subgradients of the underlying function. It also generalizes similar relationships discovered in \cite[Lemma 1]{Salzo12} for inexact proximal maps without pre-conditioning. This equivalence facilitates a simple proof of an important observation as a corollary. Namely, instead of treating errors in proximal map and gradient computations separately, it is feasible to regard them both more generally as inexactly computed proximal map evaluations (Corollary \ref{cor:gradient-inexact}).
	
	\item In Section \ref{section:inexact_block_method_composite}, we define and analyze our Inexact Cyclic Block Proximal Gradient (I-CBPG) method. To the best of our knowledge, this is the first coordinate descent-type scheme with deterministic guarantees that incorporates inexactly computed proximal maps and gradients for both smooth and non-smooth convex minimization without requiring strong convexity. The paper \cite{Hua16} exhibits a linearly convergent method for smooth and non-smooth minimization, but requires strong convexity. Our analysis provides two flavors of convergence results. First, we are able to show that, for a fixed tolerance of approximation error, the standard $\mathcal{O}(p/k)$ convergence rate for cyclic BPG methods is preserved provided the aforementioned error is pre-selected to be sufficiently small. Analogous results for randomized BPG methods with fixed errors are found in \cite{Richtarik14}. Second, we are able to show that the aforementioned rate is preserved under the relatively permissive condition that the error tolerance decreases at a $\mathcal{O}(1/k^2)$ rate. The decreasing error tolerance regime, in contrast to the fixed error tolerance regime, does not require any error tuning based on properties of the objective function, such as smoothness parameters, the initial optimality gap, or the initial iterate's distance from the set of optima. More importantly, as we see in our numerical experiments in Section \ref{sec:numerical}, the latitude that comes with looser approximations may yield significant speed advantages for early iterations in terms of CPU time. 
	\item In Section \ref{sec:numerical}, we provide numerical experiments on the well-known ordinary least squares (OLS) and LASSO problems. These experiments demonstrate the power of our method and the particular benefits of a dynamically decreasing error tolerance. For each experimental setup, we witness the computational superiority of the I-CBPG method with a dynamically decreasing error tolerance regime over its randomized analogue equipped with a fixed error tolerance as in \cite{Richtarik14}.
\end{itemize}

%
%
\section{The Inexact Proximal Map and the {$(\delta,B)$}-Second Prox Theorem}
\label{section:inexact_proximal}

In this section, we introduce a framework for analyzing the effect of inexact computation on the \emph{pre-conditioned proximal map}
\begin{equation}\label{eq.prox:precond}
\Prox_{\Psi}^B(x,g):=\argmin_{y\in\mathbb{R}^n}\left\{\inner{g}{y}+\frac{1}{2}\|y-x\|_B^2+\Psi(y)\right\},
\end{equation}
where $x,g\in\mathbb{R}^n$, $\inner{\cdot}{\cdot}$ is an inner product on $\mathbb{R}^n$, $\|\cdot\|_B$ is the norm induced by the inner product $(x,y)\mapsto \inner{Bx}{y}$ with $B\in\mathbb{R}^{n\times n}$ positive definite, and $\Psi:\mathbb{R}^n\to\mathbb{R}\cup\{\infty\}$ is proper, closed, and convex. The dual norm of $\|\cdot\|_B$, which we denote $\|\cdot\|_B^*$, is easily shown to be $\|\cdot\|_{B^{-1}}$.

We must emphasize two crucial facts about the function \eqref{eq.prox:precond}. First, it is a generalization of the standard proximal map. Indeed, by setting $B=I_n$ and $g=0$ we recover 
\[
\Prox_{\Psi}^{I_n}(x,0)=\argmin_{y\in\mathbb{R}^n}\left\{\frac{1}{2}\|y-x\|^2+\Psi(y)\right\}=:\Prox_{\Psi}(x).
\]
Second, for common choices of $\Psi$ such as the $1$-norm, $\|\cdot\|_1$, the pre-conditioned proximal map does not have a closed-form expression unless $B$ is very simple, e.g. when $B=c\cdot I_n$ for some $c\in\mathbb{R}$. Outside of these special cases, one must usually recover $\Prox_{\Psi}^B(x,g)$ via numerical approximation. 

Instead of finding the unique, exact minimizer used to define $\Prox_{\Psi}^B$, though, our goal will be to find some $y\in\mathbb{R}^n$ that solves the problem up to a small predetermined error $\delta \in \mathbb{R}_+$. We are now prepared to formally define this section's centerpiece, the set-valued \emph{inexact pre-conditioned proximal map}, as the collection of all such approximate minima at $x$ with respect to $g,\delta,B,$ and $\Psi$:
\[
\Prox^B_\Psi(x,g,\delta):=\left\{y: \inner{g}{y}+\frac{1}{2}\|y-x\|_B^2+\Psi(y)\leq\min_{z\in\mathbb{R}^n}\left\{\inner{g}{z}+\frac{1}{2}\|z-x\|_B^2+\Psi(z)\right\}+\delta\right\}.
\]
One may also regard this as a generalized pre-conditioned proximal map, since we recover the exact pre-conditioned proximal map by setting $\delta=0$.

Our central result, the $(\delta,B)$-Second Prox Theorem (Theorem \ref{thm:second-prox}) is an inexact analogue of what \cite{Beck17} calls the  ``Second Prox Theorem", a key component in the bulk of convergence proofs for proximally-driven algorithms. A special case of this theorem without pre-conditioning is contained in \cite[Lemma 1]{Salzo12}. Bregman-type generalizations of this theorem also support convergence proofs for Bregman proximal methods \cite{Lu17}. The $(\delta,B)$-Second Prox Theorem codifies the tight relationship between elements of $\Prox^B_\Psi$ and the $\delta$-subdifferential of $\Psi$. We say that $s\in\mathbb{R}^n$ is a $\delta$\emph{-subgradient of }$\Psi$\emph{ at }$x\in\dom(\Psi)$, with $\delta\geq 0$, if
\[
\Psi(y)\geq\Psi(x)+\inner{s}{y-x}-\delta\text{ for all }y\in\mathbb{R}^n.
\]
The $\delta$\emph{-subdifferential of }$\Psi$\emph{ at }$x$, $\partial \Psi_\delta(x)$, denotes the set of all $\delta$-subgradients of $\Psi$ at $x$. Rudiments of the relationship between $\Prox^B_\Psi$ and $\partial_\delta\Psi(x)$ appear in \cite{Schmidt11}, where it is shown that each $u\in\Prox^B_\Psi(x,g,\delta)$ associates to certain norm-bounded $s\in\partial_{\hat{\delta}}\Psi(x)$ for judicious selections of $\hat{\delta}$. Our $(\delta,B)$-Second Prox Theorem secures this result and its converse: each $\delta$-subgradient of $\Psi$ corresponds to a specially chosen inexact proximal map element. As the reader will see in the last half of this section, this equivalent characterization is particularly useful for theoretically determining whether $u\in\Prox^{B}_\Psi(x,g,\delta)$. Indeed, we show by verification of this condition that one may regard approximation errors in gradient and proximal map computations through the unifying framework of inexact proximal map evaluations (Corollary \ref{cor:gradient-inexact}).


Before stating and proving the $(\delta,B)$-Second Prox Theorem, we recall two rules of the $\delta$-subdifferential calculus that are crucial to its proof: a sum rule and an optimality condition.

\begin{theorem}[$\delta$-Subdifferential Calculus]\label{thm:optimality:approx}
If $\Psi,\tilde{\Psi}:\mathbb{R}^n\to\mathbb{R}\cup\{\infty\}$ are proper, closed, and convex, and $\delta\geq 0$ then
\begin{enumerate}[(i)]
	\item (Optimality Condition)  It holds that $\Psi(x)-min_{y\in\mathbb{R}^n}\Psi(y)\leq\delta$ if and only if $0\in\partial_\delta\Psi(x)$ \cite[Theorem XI 1.1.5]{Hiriart13}.
	\item (Sum Rule) If $\text{ri}[\dom(\Psi)]\cap\text{ri}\left[\dom\left(\tilde{\Psi}\right)\right] \neq \emptyset$, where $\text{ri}(\cdot)$ denotes the relative interior of a convex set, then 
	\[
	\partial_\delta \left[\Psi+\tilde{\Psi}\right](x)=\bigcup_{\hat{\delta}\in[0,\delta]}\left[\partial_{\hat{\delta}}\Psi(x)+\partial_{\left(\delta-\hat{\delta}\right)}\tilde{\Psi}(x)\right]
	\]
	for all $x\in\dom(\Psi)\cap\dom\left(\tilde{\Psi}\right)$, where the sum on the right is taken in the Minkowski sense \cite[Theorem XI 3.1.1]{Hiriart13}.
\end{enumerate}
\end{theorem}

We now state and prove the primary result of this section, the $(\delta,B)$-Second Prox Theorem. All other results in this section, along with the majority of those in its sequel, hang on this theorem.

\begin{theorem}[$(\delta,B)$-Second Prox]\label{thm:second-prox}
Let $\Psi:\mathbb{R}^n\to\mathbb{R}\cup\{\infty\}$ be proper, closed, and convex, $\delta\geq 0$, and $B\succ 0$. The following are equivalent:
\begin{enumerate}[(i)]
	\item $u\in \Prox^B_\Psi(x,g,\delta)$.
	\item There exists $\hat{\delta}\in[0,\delta]$ and $v\in\mathbb{R}^n$ such that $\|v\|_B^*\leq\sqrt{2\left(\delta-\hat{\delta}\right)}$ and $v-g-B(u-x)\in\partial_{\hat{\delta}}\Psi (u)$.
	\item There exists $\hat{\delta}\in[0,\delta]$ and $v\in \mathbb{R}^n$ such that $\|v\|_B^*\leq\sqrt{2\left(\delta-\hat{\delta}\right)}$ and 
\begin{equation}\label{eqn:approxsecond}
\inner{v-g-B(u-x)}{y-u}\leq \Psi(y)-\Psi(u)+\hat{\delta}
\end{equation}
for all $y\in\mathbb{R}^n$.
\end{enumerate}
\end{theorem}

\begin{proof}
The equivalence of (ii) and (iii) is immediate from the definition of $\partial_{\hat{\delta}}\Psi(u)$ so it suffices to show the equivalence of (i) and (ii). The crux of this equivalence's proof is the expression of the $\delta$-subgradient for the function $z\mapsto \inner{g}{z}+\frac{1}{2}\|z-x\|_B^2+\Psi(z)$ for $x$ fixed,
\begin{equation}\label{eqn:approx:subgrad:sum}
\partial_{\delta}\left[\inner{g}{\cdot}+\frac{1}{2}\|\cdot-x\|_B^2+\Psi(\cdot)\right](z)=g+\bigcup_{\hat{\delta}\in[0,\delta]}\left[\partial_{\left(\delta-\hat{\delta}\right)} \frac{1}{2} \|\cdot-x\|_B^2(z)+\partial_{\hat{\delta}}\Psi(z)\right],
\end{equation}
which is a consequence of Theorem \ref{thm:optimality:approx}(ii). A straightforward computation produces
\[
\partial_{\left(\delta-\hat{\delta}\right)} \frac{1}{2} \|\cdot-x\|_B^2(z)=B(z-x)+\left\{v:\|v\|_B^*\leq\sqrt{2\left(\delta-\hat{\delta}\right)}\right\}{,}
\]
where the sum is taken in the Minkowski sense. Thus, we write \eqref{eqn:approx:subgrad:sum} more explicitly as
\begin{multline*}
\partial_{\delta}\left[\inner{g}{\cdot}+\frac{1}{2}\|\cdot-x\|_B^2+\Psi(z)\right](z)=\\
g+B(z-x)+\bigcup_{\hat{\delta}\in[0,\delta]}\left[\partial_{\hat{\delta}}\Psi(z)+\left\{v:\|v\|_B^*\leq\sqrt{2\left(\delta-\hat{\delta}\right)}\right\}\right].
\end{multline*}
In light of the $\delta$-subdifferential optimality condition (Theorem \ref{thm:optimality:approx}(i)), $u\in\Prox^B_\Psi(x,g,\delta)$ if and only if
\[
0\in g+B(u-x)+\bigcup_{\hat{\delta}\in[0,\delta]}\left[\partial_{\hat{\delta}}\Psi(u)+\left\{v:\|v\|_B^*\leq\sqrt{2\left(\delta-\hat{\delta}\right)}\right\}\right].
\]
Noting that $\|v\|_B^* = \|-v\|_B^*$ and rearranging, this inclusion is clearly equivalent to 2).
\end{proof}
\noindent The $\delta$-subdifferential characterization of the inexact pre-conditioned proximal map will facilitate easy proofs of this map's key properties and algorithmic convergence. Practically speaking, it is not necessary to directly check any $\delta$-subdifferential condition when computing these maps. In Section \ref{sec:numerical}, the conjugate gradient method and box constrained gradient projection method of \cite{Broughton11} provide easily and efficiently implementable procedures for computing inexact pre-conditioned proximal maps.

A notable product of the $(\delta,B)$-Second Prox Theorem is that inexact proximal maps exhibit Lipschitz continuity in $x${,} $g$ up to the level of inexactness, as the theorem below formalizes. It will play the same role in the convergence proof for I-CBPG as its exact analogue does in the convergence proof of the Cyclic Block Proximal Gradient method (compare Lemmata \ref{lemma:prox:block-decrease} and \ref{lemma:prox:convex:sufficient-decrease} with \cite[Lemma 11.11]{Beck17} and \cite[Lemma 11.16]{Beck17}).

\begin{theorem}[Error-Dependent Lipschitz Continuity of the $(\delta,B)$-Proximal Map]\label{thm:non-expansive}
Let $\Psi:\mathbb{R}^n\to\mathbb{R}\cup\{\infty\}$ be proper, closed, and convex, $x,y,g,h\in\mathbb{R}^n${,}  $\delta,\epsilon\geq 0$, and $B\succ 0$. Then
\begin{equation}\label{eq:non-expansive}
\|u-w\|_B\leq\|g-h\|_B^*+\|y-x\|_B+\left(1+\frac{\sqrt{2}}{2}\right)\cdot\left(\sqrt{\delta}+\sqrt{\epsilon}\right)
\end{equation}
for all $u\in \Prox^B_\Psi(x,g,\delta)$ and $w\in \Prox^B_\Psi(y,h,\epsilon)$.
\end{theorem}

\begin{proof}
Invoking the $(\delta,B)$-Second Prox Theorem, we may choose $v_u,v_w\in \mathbb{R}^n$, $\hat{\delta}\in[0,\delta]$, and $\hat{\epsilon}\in[0,\epsilon]$, such that $\|v_u\|_B^*\leq\sqrt{2\left(\delta-\hat{\delta}\right)}$, $\|v_w\|_B^*\leq\sqrt{2\left(\epsilon-\hat{\epsilon}\right)}$, and 
\begin{align*}
\inner{v_u-g-B(u-x)}{z-u}&\leq \Psi(z)-\Psi(u)+\hat{\delta}\\
\inner{v_w-h-B(w-y)}{z'-w}&\leq \Psi(z')-\Psi(w)+\hat{\epsilon}
\end{align*}
for all $z,z'\in\mathbb{R}^n$. If we add these two inequalities with $z=w$ and $z'=u$ then
\[
\inner{v_u-g-B(u-x)}{w-u}+\inner{v_w-h-B(w-y)}{u-w}\leq\hat{\delta}+\hat{\epsilon},
\]
which simplifies to the more informative
\[
\|u-w\|_B^2+\inner{(g-h)+(v_w-v_u)+B(y-x)}{u-w}-(\hat{\delta}+\hat{\epsilon})\leq 0.
\]
The standard Cauchy-Schwarz inequality, in conjunction with its more general form for arbitrary norms and their dual norms along with the triangle inequality, implies
\[
\|u-w\|_B^2-\left(\|(g-h)+(v_w-v_u)\|_B^*+\|y-x\|_B\right)\cdot\|u-w\|_B-(\hat{\delta}+\hat{\epsilon})\leq 0.
\]
The left-hand side of this inequality is quadratic in $\|u-w\|_B$. Thus we have the generic bound
\begin{multline}\label{eq:non-expansiveness-quadratic}
\|u-w\|_B\leq\frac{\|(g-h)+(v_w-v_u)\|_B^*+\|y-x\|_B}{2}\\+\frac{\sqrt{(\|(g-h)+(v_w-v_u)\|_B^*+\|y-x\|_B)^2+4(\hat{\delta}+\hat{\epsilon})}}{2}\;.
\end{multline}

We now reason by cases to derive four specialized versions of \eqref{eq:non-expansiveness-quadratic} (equations \eqref{eq:non-expansive-1}, \eqref{eq:non-expansive-2},  \eqref{eq:non-expansive-3}, and \eqref{eq:non-expansive-4})  that we chain together via the triangle inequality in \eqref{eq:non-expansive-triangle} to secure our result \eqref{eq:non-expansive}. First, with $x$, $g$, $\delta$, and $u$  fixed as above, suppose that $y = x$, $h = g$, and $\epsilon = 0$. Then $w=\Prox_\Psi^B(x,g,0)$, $v_w=0$, and the generic bound  \eqref{eq:non-expansiveness-quadratic} reduces to  
\begin{align}
\left\|u-\Prox_\Psi^B(x,g,0)\right\|_B\leq\frac{\|v_u\|_B^*+\sqrt{\left(\|v_u\|_B^*\right)^2+4\hat{\delta}}}{2}&\leq\frac{\sqrt{2\delta}+\sqrt{2(\delta+\hat{\delta})}}{2}\notag\\
&\leq\left(1+\frac{\sqrt{2}}{2}\right)\cdot\sqrt{\delta}{,}\label{eq:non-expansive-1}
\end{align}
with the latter two inequalities respectively resulting from $\|v_u\|_B^*<\sqrt{2\left(\delta-\hat{\delta}\right)}$ and $0\leq\hat{\delta}\leq\delta$. Second, with $x,g,$ and $h$ similarly fixed, let $y = x$ and $\delta=\epsilon=0$. Obviously,  $v_u=v_w=0$, $u=\Prox_\Psi^B(x,g,0)$, and $w=\Prox_\Psi^B(x,h,0)$ so now \eqref{eq:non-expansiveness-quadratic} reduces to
\begin{equation}\label{eq:non-expansive-2}
\|\Prox_\Psi^B(x,g,0)-\Prox_\Psi^B(x,h,0)\|_B\leq\frac{\|g-h\|_B^*}{2}+\frac{\sqrt{(\|g-h\|_B^*)^2}}{2}=\|g-h\|_B^*.
\end{equation}
Third, with $x$, $y$, and $h$ fixed as above, suppose that $\delta=\epsilon=0$ and $g=h$. Then we obtain
\begin{equation}\label{eq:non-expansive-3}
\|\Prox_\Psi^B(x,h,0)-\Prox_\Psi^B(y,h,0)\|_B\leq\frac{\|y-x\|_B}{2}+\frac{\sqrt{(\|y-x\|_B)^2}}{2}=\|y-x\|_B{,}
\end{equation}
since here $v_u=v_w=0$, $u=\Prox_\Psi^B(x,h,0)$, and $w=\Prox_\Psi^B(y,h,0)$. Finally, with $y$, $h$, and $\epsilon$ fixed at their original values, we may let $x = y$, $g = h$, and $\delta = 0$ to see that, by the same logic as  \eqref{eq:non-expansive-1},
\begin{equation}\label{eq:non-expansive-4}
\|\Prox_\Psi^B(y,h,0)-w\|_B \leq 
\frac{\|v_w\|_B^*+\sqrt{\left(\|v_w\|_B^*\right)^2+4\hat{\epsilon}}}{2}\leq \left(1+\frac{\sqrt{2}}{2}\right)\cdot\sqrt{\epsilon}.
\end{equation}

Now, we may complete our proof. Furnished with \eqref{eq:non-expansive-1}, \eqref{eq:non-expansive-2}, \eqref{eq:non-expansive-3}, and  \eqref{eq:non-expansive-4}, we compute
\begin{multline}\label{eq:non-expansive-triangle}
\|u-w\|_B\leq\|u-\Prox_\Psi^B(x,g,0)\|_B\\+
\|\Prox_\Psi^B(x,g,0)-\Prox_\Psi^B(x,h,0)\|_B+\|\Prox_\Psi^B(x,h,0)-\Prox_\Psi^B(y,h,0)\|_B\\
+\|\Prox_\Psi^B(y,h,0)-w\|_B\\
\leq \left(1+\frac{\sqrt{2}}{2}\right)\cdot\sqrt{\delta}+\|g-h\|_B^*+\|y-x\|_B+\left(1+\frac{\sqrt{2}}{2}\right)\cdot\sqrt{\epsilon}
\end{multline}
for all $u\in \Prox^B_\Psi(x,g,\delta)$ and $w\in \Prox^B_\Psi(y,h,\epsilon)$, which is precisely \eqref{eq:non-expansive}.
\end{proof}

Let us now discuss how the inexact pre-conditioned proximal map treats approximation of gradients and proximal maps in a unified manner. In the context of first-order, proximally-based algorithms, a number of researchers study approximation error in just one of either the proximal map or gradient \cite{Alexandre_2008,Villa13,Devolder14}, or, in cases where both types of errors are considered, their treatment is often handled separately \cite{Schmidt11}. For example, \cite{Schmidt11}, one of the authoritative works on the proximal gradient scheme with errors for the composite minimization problem $\min_{x\in\mathbb{R}^n} f(x)+\Psi(x)$ tenders the error dependent scheme
\begin{align*}
	x_k=&\Prox_\Psi^{I_n}\left(y_{k-1},t\left(\nabla f(y_{k-1})+e_k\right),\delta_k\right)\\
	y_k&=x_k+\beta_k(x_k-x_{k-1})
\end{align*}
where $\{e_k\}_{k\geq 1}$ records the gradient approximation error, $\{\beta_k\}\subseteq[0,\infty)$ dictates the momentum used to accelerate the algorithm, and $t>0$ is a stepsize parameter. 

Theorem \ref{thm:non-expansive} unveils an intriguing property of our $(\delta,B)$-Second Prox Theorem framework for analyzing the inexact proximal map: it is possible to unify the treatment of inexactly computed proximal maps and gradients by simply considering inexact proximal map computations. This observation is formalized via the next corollary's set inclusion.

\begin{corollary}\label{cor:gradient-inexact}
Let $\Psi:\mathbb{R}^n\to\mathbb{R}\cup\{\infty\}$ be proper, closed, and convex, and $B\succ 0$. For all $x,g,e\in\mathbb{R}^n$ and $\delta\geq 0$, we have the inclusion
\[
\Prox_\Psi^B(x,g+e,\delta)\subseteq\Prox_\Psi^B\left(x,g,\delta+\sqrt{2\delta}\|e\|_B^*+\frac{1}{2}\left(\|e\|_B^*\right)^2\right).
\]
\end{corollary}

\begin{proof}
This follows from Theorem \ref{thm:second-prox}(ii). Fix $u\in\Prox_\Psi^B(x,g+e,\delta)$. There exist $\hat{\delta}\in[0,\delta]$ and $v\in\mathbb{R}^n$ such that $\|v\|_B^*\leq\sqrt{2\left(\delta-\hat{\delta}\right)}$ and $v-(g+e)-B(u-x)\in\partial_{\hat{\delta}}\Psi(u)$. Equivalently, $(v-e)-g-B(u-x)\in\partial_{\hat{\delta}}\Psi(u)$. Thus, $u\in\Prox_\Psi^B\left(x,g,\delta+\sqrt{2\delta}\|e\|_B^*+\frac{1}{2}\left(\|e\|_B^*\right)^2\right)$ since
\[
\|v-e\|_B^*\leq\sqrt{2\left[\left(\frac{1}{2}\left(\|v-e\|_B^*\right)^2+\hat{\delta}\right)-\hat{\delta}\right]}\leq\sqrt{2\left[\left(\delta+\sqrt{2\delta}\|e\|_B^*+\frac{1}{2}\left(\|e\|_B^*\right)^2\right)-\tilde{\delta}\right]}\;.
\]
\end{proof}

%
%
\section{The Inexact Cyclic Block Proximal Gradient Method}
\label{section:inexact_block_method_composite}

In this section we introduce and analyze the Inexact Cyclic Block Proximal Gradient method (I-CBPG), a variant of the cyclic block proximal gradient method for \eqref{eq.problem} that allows for approximate evaluations of (pre-conditioned) proximal maps and gradients. Throughout, we will assume that the smooth and convex component $f$ of \eqref{eq.problem} satisfies the following \emph{block smoothness} condition: for each $i=1,\ldots,p${,} there exist a positive definite matrix $B_i\in\mathbb{R}^{n_i\times n_i}$ and $L_i>0$ such that
\begin{equation}\label{eq:smooth:block}
f(x+U_i t)\leq f(x)+\inner{\nabla_i f(x)}{t}+\frac{L_i}{2}\|t\|_{(i)}^2\text{ for all }x\in\mathbb{R}^n\text{ and }t\in\mathbb{R}^{n_i},
\end{equation}
where $\|\cdot\|_{(i)}$ denotes the norm on $\mathbb{R}^{n_i}$ induced by the inner product $(x,y)\mapsto\inner{B_i x}{y}$ and $\nabla_i f(x) = U_i^T \nabla f(x)$. We will also assume that $f$ satisfies a (pre-conditioned) \emph{smoothness} condition: there exists $L_f>0$ such that\begin{equation}\label{eq:smooth}
f(x+t)\leq f(x)+\inner{\nabla f(x)}{t}+\frac{L_f}{2}\|t\|_{B}^2\text{ for all }x,t\in\mathbb{R}^n,
\end{equation}
where $\|\cdot\|_B=\sqrt{\sum_{i=1}^p\|U_i^T \cdot\|_{(i)}^2}$. Finally, we will assume that $F$ is coercive, i.e. it has bounded sublevel sets. In particular, this implies that $R(x):=\sup_{y\in X^*}\|x-y\|<\infty$ where $X^* := \argmin_x f(x)$.

To define the algorithm, we assume for each $i=1,\ldots,p$, the set-valued map $\Prox^{B_i}_{\Psi_i/L_i}$ admits a selection function
\[
(x,\delta)\mapsto T_\delta^{(i)}(x)\in\Prox^{B_i}_{\frac{\Psi_i}{L_i}}\left(U_i^Tx,\frac{1}{L_i}\nabla_i f(x),\delta\right) \subseteq \mathbb{R}^{n_i} 
\] 
that ensures the monotonic decrease condition
\begin{equation}\label{eq.mondec}
F\left( x+U_i[T_\delta^{(i)}(x)-x_i] \right)\leq F(x){,}
\end{equation}
where $x_i \in \mathbb{R}^{n_i}$ is the $i$th block of $x$, that is, $x_i = U_i^T x$.\footnote{In general, not every selection function for $\Prox^{B_i}_{\Psi_i/L_i}$ satisfies the monotone decrease condition \eqref{eq.mondec}. However, computable selection functions satisfying \eqref{eq.mondec} are available for all of our applications in Section \ref{sec:numerical}.} It is now possible to introduce our I-CBPG scheme.\\

\begin{algorithm}[H]
\caption{Inexact Cyclic Block Proximal Gradient (I-CBPG) Method}\label{algo.main}
\begin{algorithmic}[1]

\Require $x^0 \in \dom(f)$, $\delta_1\geq 0$
	\For{$k=0,1,2,\ldots$}
		\State $x^{k,0}=x^k$\;
		\For{$i=1,\ldots,p$}
		\State \begin{equation}
			x^{k,i}=x^{k,i-1}+U_i[T_{\delta_{k+1}}^{(i)}(x^{k,i-1})-x^{k,i-1}_i];
		\end{equation}
		\EndFor
		\State $x^{k+1}:= x^{k,p}$
		\State Choose $\delta_{k+2}\in[0,\delta_{k+1}]$\;		
	\EndFor	
\end{algorithmic}
\end{algorithm}



    

\noindent We call $\{\delta_k\}_{k\geq 1}$ the sequence of \emph{error tolerances}. Two types of error tolerance sequences will be considered: \emph{fixed} sequences for which we merely assume that $\delta_k=\delta\geq 0$ for $k\geq 1$ and \emph{dynamically decreasing} sequences that converge to $0$ at the sublinear rate $\mathcal{O}(1/k^2)$.
 
Our analysis of I-CBPG (Algorithm \ref{algo.main}) follows a standard three step outline for proving convergence of a first-order method purposed for convex minimization. First, we prove a sufficient decrease condition (Lemma \ref{lemma:prox:block-decrease}) that relates the suboptimality gap to the norm of the inter-iterate difference, $x^{k,i}-x^{k,i-1}$. Second, using the sufficient decrease condition, we derive a recurrence inequality (Lemma \ref{lemma:prox:convex:sufficient-decrease}) satisfied by the sequence of suboptimality gaps at each iterate. Third, we prove a technical lemma (Lemma \ref{lemma:recurrence-technical}) that describes the rate of convergence of a recurrence of the form found in Lemma \ref{lemma:prox:convex:sufficient-decrease}. Finally, we deduce our desired convergence rates as a consequence of the technical lemma and the suboptimality gap recurrence inequality. The convergence rates for fixed errors are summarized in Theorem \ref{theorem:convergence-bounded-error} and Corollary \ref{cor:convergence-bounded-error} while the convergence rates for sublinearly decreasing errors are summarized in Theorem \ref{theorem:convergence-decreasing-error} and Corollary \ref{cor:convergence-decreasing-error}.

We begin by presenting the sufficient decrease inequality.

\begin{lemma}[Sufficient Decrease Inequalities]\label{lemma:prox:block-decrease}
Let $\{x^k\}_{k\geq 0}$ and $\{\delta_k\}_{k\geq 1}$ denote the sequences of iterates and error tolerances generated by I-CBPG (Algorithm \ref{algo.main}). Then we have that 
\begin{enumerate}[(i)]
	\item For all $k\geq 0$ and $1\leq i\leq p$,
\begin{equation}\label{eqn:prox:block-decrease:main-inexact}
3L_i\delta_{k+1}+F(x^{k,i})-F(x^{k,i-1})\geq\frac{L_i}{4}\left\|x^{k,i}-x^{k,i-1}\right\|_{(i)}^2.
\end{equation}
	\item For all $k\geq 0$,
	\begin{equation}
3L_{\min}p\delta_{k+1}+F(x^k)-F(x^{k+1})\geq\frac{L_{\min}}{4}\| x^k - x^{k+1}	\|_B^2.
\end{equation}
\end{enumerate}
\end{lemma}

\begin{proof} 
To streamline notation for the proofs of (i) and (ii), let us define the variables $x=x^{k,i-1}$,  $x^+=x^{k,i}$, and $\delta = \delta_{k+1}$. First, notice that the $i$-th block of $x^+$ is $x^+_i=T_\delta^{(i)}(x)$.\\

\noindent (i) By the blockwise smoothness property,
\[
f(x^+)\leq f(x)+\inner{\nabla_i f(x)}{x^+_i-x_i}+\frac{L_i}{2}\|x^+_i-x_i\|_{(i)}^2{,}
\]
so
\begin{align}
F(x^+)&\leq F(x)+\inner{\nabla_i f(x)}{x^+_i-x_i}+\frac{L_i}{2}\left\|x^+_i-x_i\right\|_{(i)}^2+\Psi_i(x^+_i)-\Psi_i(x_i).\label{eqn:prox:block-decrease:smooth1}
\end{align}
We aim to refine the right-hand side of this inequality. To this end, Theorem \ref{thm:second-prox}(iii) gives us 
\[
\inner{v-\frac{1}{L_i}\nabla_i f(x)-B_i(x^+_i-x_i)}{x_i-x^+_i}\leq \frac{\Psi_i}{L_i}(x_i)-\frac{\Psi_i}{L_i}\left(x^+_i\right)+\delta
\]
for some $v\in\mathbb{R}^n$ such that $\|v\|_{(i)}^*\leq \sqrt{2\delta}$, since $x^+_i=T_\delta^{(i)}(x)$. We rearrange this inequality to bound $\inner{\nabla_i f(x)}{x^+_i-x_i}$ according to
%
%
\begin{align}
\inner{\nabla_i f(x)}{x^+_i-x_i}&\leq-L_i\|x^+_i-x_i\|_{(i)}^2+\Psi_i(x_i)-\Psi_i\left(x^+_i\right)+L_i\inner{v}{x^+_i-x_i}+L_i\delta\notag\\
& \leq  - L_i\|x^+_i-x_i\|_{(i)}^2+\Psi_i(x_i)-\Psi_i\left(x^+_i\right)\notag\\
&\hspace{10em}+\frac{L_i}{2}\Big( 2(\|v\|_{(i)}^*)^2 + \frac{1}{2}\|x^+_i-x_i\|_{(i)}^2 \Big)+L_i\delta\label{eq:CS-app-1}\\
& \leq  - \frac{3L_i}{4}\|x^+_i-x_i\|_{(i)}^2+\Psi_i(x_i)-\Psi_i\left(x^+_i\right)+3L_i\delta{,}\label{eq:second-prox-grad-bound}
\end{align}
where on the third line we applied the AM-GM inequality to $\inner{v}{x^+_i-x_i}=\inner{\sqrt{2}v}{\frac{1}{\sqrt{2}}(x^+_i-x_i)}$. Finally, inserting \eqref{eq:second-prox-grad-bound} into \eqref{eqn:prox:block-decrease:smooth1}, we settle on a rearranged \eqref{eqn:prox:block-decrease:main-inexact},
\[
F(x^+)\leq F(x)-\frac{L_i}{4}\left\|x^+_i-x_i\right\|_{(i)}^2+3L_i\delta.
\]
\noindent (ii) Rearrange the below chain of inequalities that follows from applying (i)
\begin{align*}
\| x^k - x^{k+1}	\|_B^2=\sum_{i=1}^{p}\|x^{k,i}-x^{k,i-1}\|_{(i)}^2&\leq\sum_{i=1}^{p}\left\{\frac{4}{L_i}[F(x^{k,i-1})-F(x^{k,i})]+12\delta\right\}\\
&\leq\frac{4}{L_{\min}}[F(x^k)-F(x^{k+1})]+12p\delta.
\end{align*}
\end{proof}
We now take the second step in our analysis: deriving the main recurrence inequality.

\begin{lemma}\label{lemma:prox:convex:sufficient-decrease}
Let $\{x^k\}_{k\geq 0}$ and $\{\delta_k\}_{k\geq 1}$ denote the sequences of iterates and error tolerances generated by I-CBPG (Algorithm \ref{algo.main}). Then for all $k\geq 0$ the recurrence inequality
\[
\frac{L_{\min}}{2p(L_f+L_{\max})^2 R(x^0)^2}[F(x^{k+1})-F^*]^2\leq F(x^k)-F(x^{k+1})+ \mathcal{O}(\delta_{k+1})
\]
holds. Specifically,
\begin{multline}\label{eq:recur-main}
\frac{L_{\min}}{8p(L_f+L_{\max})^2 R(x^0)^2}[F(x^{k+1})-F^*]^2
\leq\\ F(x^k)-F(x^{k+1})+ L_{\min}\left[3p +    \frac{L^2_{\max}}{4}\left(  \frac{R(x^0)\sqrt{2} + \sqrt{p\delta_1}}{(L_f + L_{\max})R(x^0)} \right)^2\right]\delta_{k+1}.
\end{multline}
\end{lemma}

\begin{proof}
Fix $k\geq 0$ and $i\in\{1,\ldots,p\}$. Invoking the $(\delta,B)$-Second Prox Theorem (Theorem \ref{thm:second-prox}) for $T^{(i)}_{\delta_{k+1}}(x^{k,i-1})$, there exists $v^{k,i}\in\mathbb{R}^{n_i}$ with $\|v^{k,i}\|_{(i)}^*\leq \sqrt{2\delta_{k+1}}$ such that
\[
\frac{\Psi_i}{L_i}(y)-\frac{\Psi_i}{L_i}(x^{k,i}_i)+\delta_{k+1} \geq\inner{v^{k,i}-\frac{1}{L_i}\nabla_i f(x^{k,i-1})-B_iT^{(i)}_{\delta_{k+1}}(x^{k,i-1})}{y-x^{k,i}_i}
\]
for any $y\in\mathbb{R}^{n_i}$. Setting $y=x^*_i$, recognizing that $x^{k,i}_i=x^{k+1}_i$ and $T^{(i)}_{\delta_{k+1}}(x^{k,i-1})=x_i^{k+1}-x_i^k$, and multiplying both sides by $L_i$, we compute
\[
\Psi_i(x^*_i)-\Psi_i(x^{k,i}_i)+L_i\delta_{k+1}\geq L_i\inner{v^{k,i}-\frac{1}{L_i}\nabla_i f(x^{k,i-1})+B_i(x_i^k-x_i^{k+1})}{x^*_i-x^{k+1}_i}.
\]
Summing this last inequality over $i\in\{1,\ldots,p\}${,} we see
\begin{equation*}
\Psi(x^*)-\Psi(x^{k+1})+\sum_{i=1}^p L_i \delta_{k+1}\geq \sum_{i=1}^p L_i\inner{v^{k,i}-\frac{1}{L_i}\nabla_i f(x^{k,i-1})+B_i(x_i^k-x_i^{k+1})}{x^*_i-x^{k+1}_i}{,}
\end{equation*}
\noindent which we will eventually use in the rearranged form
\begin{multline}\label{eqn:prox:convex:total-nonsmooth}
\sum_{i=1}^p L_i\left(\delta_{k+1}+\inner{v^{k,i}-\frac{1}{L_i}\nabla_i f(x^{k,i-1})+B_i(x_i^k-x_i^{k+1})}{x^{k+1}_i-x^*_i}\right)\\
\geq \Psi(x^{k+1})-\Psi(x^*).
\end{multline}
The convexity of $f$ implies that 
\begin{align*}
F(x^{k+1})-F^*&=f(x^{k+1})-f(x^*)+\Psi(x^{k+1})-\Psi(x^*)\\
&\leq\inner{\nabla f(x^{k+1})}{x^{k+1}-x^*}+\Psi(x^{k+1})-\Psi(x^*)\\
&\leq\sum_{i=1}^p \inner{\nabla_i f(x^{k+1})}{x^{k+1}_i-x^*_i}+\Psi(x^{k+1})-\Psi(x^*){,}
\end{align*}
which we combine with \eqref{eqn:prox:convex:total-nonsmooth} to yield
\begin{multline}
F(x^{k+1})-F^*\leq\sum_{i=1}^p \inner{\nabla_i f(x^{k+1})}{x^{k+1}_i-x^*_i}\\
+\sum_{i=1}^p L_i\left(\delta_{k+1}+\inner{v^{k,i}-\frac{1}{L_i}\nabla_i f(x^{k,i-1})+B_i(x_i^k-x_i^{k+1})}{x^{k+1}_i-x^*_i}\right)\\
=\sum_{i=1}^p \left(L_i\delta_{k+1}+\inner{\nabla_i f(x^{k+1})-\nabla_i f(x^{k,i-1})+L_iB_i(x^k_i-x^{k+1}_i)+L_i v^{k,i}}{x^{k+1}_i-x^*_i}\right).
\end{multline}

We further compute
\begin{multline}\label{eq:sq-subopt-gap-1}
F(x^{k+1})-F^*\leq\sum_{i=1}^p \bigg[\|\nabla_i f(x^{k+1})-\nabla_i f(x^{k,i-1})\|^*_{(i)}+L_i\|x^k_i-x^{k+1}_i\|_{(i)}\\
+L_i\|v^{k,i}\|_{(i)}^*\bigg]\cdot\|x^{k+1}_i-x^*_i\|_{(i)}+pL_{\max}\delta_{k+1}\\
\leq \sum_{i=1}^p \left[L_f\|x^{k+1}-x^{k,i-1}\|_B+L_{\max}\|x^k_i-x^{k+1}_i\|_{(i)}+L_i\sqrt{2\delta_{k+1}}\right]\cdot\|x^{k+1}_i-x^*_i\|_{(i)}\\
\hspace{18em}+pL_{\max}\delta_{k+1}\\
=(L_f+L_{\max})\|x^k-x^{k+1}\|_B\cdot\sum_{i=1}^p \|x^{k+1}_i-x^*_i\|_{(i)}+L_{\max}\cdot\left(\sum_{i=1}^p \|x^{k+1}_i-x^*_i\|_{(i)}\right)\\
\cdot\sqrt{2\delta_{k+1}}+pL_{\max}\delta_{k+1}{,}
\end{multline}
where we use the Cauchy-Schwarz, triangle, block smoothness \eqref{eq:smooth:block}, and preconditioned smoothness \eqref{eq:smooth} inequalities along with the norm bounds  on the $v^{k,i}$ terms on the first line.
The norm equivalence bound $\|\cdot\|_1\leq p^{1/2}\|\cdot\|_2$ on $\mathbb{R}^p$ along with the coercivity assumption implies
\[
\sum_{i=1}^p \|x^{k+1}_i-x^*_i\|_{(i)}\leq p^{1/2}\sqrt{\sum_{i=1}^p \|x^{k+1}_i-x^*_i\|_{(i)}^2}=p^{1/2}\|x^{k+1}-x^*\|_B\leq p^{1/2}R\left(x^0\right){,}
\]
so we may refine \eqref{eq:sq-subopt-gap-1} to
\[
F(x^{k+1})-F^*\leq p^{1/2}(L_f+L_{\max})R(x^0)\|x^k-x^{k+1}\|_B+p^{1/2}L_{\max}R(x^0)\sqrt{2\delta_{k+1}}+pL_{\max}\delta_{k+1}.
\]
Now, observe that by squaring and applying the Cauchy-Schwarz inequality and monotonicity of the sequence $\{\delta_k\}_{k\geq 1}$, we get
\begin{eqnarray}
[F(x^{k+1})-F^*]^2 \leq \Big[ p^{1/2}(L_f+L_{\max})R(x^0)\|x^k-x^{k+1}\|_B+p^{1/2}L_{\max}R(x^0)\sqrt{2\delta_{k+1}}+pL_{\max}\delta_{k+1} \Big]^2\notag\\
\leq 2p(L_f+L_{\max})^2R(x^0)^2\|x^k-x^{k+1}\|_B^2  + 2\delta_{k+1} (\sqrt{2} p^{1/2}L_{\max}R(x^0) + pL_{\max}\sqrt{\delta_{k+1}}  )^2\label{eq:CS-app-2}\\
\leq 2p(L_f+L_{\max})^2R(x^0)^2\|x^k-x^{k+1}\|_B^2  + 2\delta_{k+1} (\sqrt{2} p^{1/2}L_{\max}R(x^0) + pL_{\max}\sqrt{\delta_1}  )^2.\notag
\end{eqnarray}

At this point,  we multiply both sides of the inequality by $\frac{L_{\min}}{8p(L_f+L_{\max})^2 R(x^0)^2}$ and apply the bound of Lemma \ref{lemma:prox:block-decrease} in a straightforward fashion to obtain
\begin{multline}
\frac{L_{\min}}{8p(L_f+L_{\max})^2  R(x^0)^2 }[F(x^{k+1})-F^*]^2 \leq \frac{L_{\min}}{4}\| x^k-x^{k+1}\|_B^2  \\
\hspace{12em}+\left( \frac{L_{\min}(\sqrt{2} p^{1/2}L_{\max}R(x^0) + pL_{\max}\sqrt{\delta_1}  )^2}{4p(L_f + L_{\max})^2R(x^0)^2} \right)\delta_{k+1} \\
\leq  F(x^k)-F(x^{k+1}) + 3L_{\min}p\delta_{k+1} +  \frac{L_{\min}L^2_{\max}}{4}\Bigg(   \frac{R(x^0)\sqrt{2} + \sqrt{p\delta_1}}{(L_f + L_{\max})R(x^0)} \Bigg)^2 \delta_{k+1}.
\end{multline}
\end{proof}

As promised, we see that stating and proving our main convergence results hinges upon determining the convergence rate of a sequence satisfying a certain recurrence inequality. The following technical lemma, whose proof we defer to this paper's singular appendix, accomplishes this task.

\begin{lemma}\label{lemma:recurrence-technical}
If $\{A_\ell\}_{\ell\geq 0}$ and $\{\Delta_\ell\}_{\ell\geq 1}$ are non-negative, non-increasing sequences of real numbers satisfying the recurrence inequality
\begin{equation}\label{eq:recur-ineq}
\frac{1}{\gamma}A_{\ell+1}^2\leq A_\ell-A_{\ell+1}+\Delta_{\ell+1}
\end{equation}
for some $\gamma\geq 1${,} then the following hold:
\begin{enumerate}[(i)]
\item If $\{\Delta_{\ell}\}_{\ell\geq 1}$ is a constant sequence such that $\Delta_\ell = \Delta\geq 0$ for all $\ell\geq 1$,  then for $u = \sqrt{\Delta\gamma}$, we have that 
\[
A_k\leq\max\left\{\frac{4\gamma (A_0 - u)}{(k-1)(A_0+3u)} +u ,\left(\frac{1}{2}\right)^{(k-1)/2}A_0\right\}
\]
for $k\geq 2$.
	\item If $\{\Delta_\ell\}_{\ell\geq 1}$ shrinks at the sublinear rate $\mathcal{O}(1/k^2)$, i.e. there exists $D>0$ such that $\Delta_\ell\leq D/\ell^2$ for $\ell\geq 1$, then
	\[
A_k\leq \max\left\{\frac{16\gamma}{k-3},\frac{8\sqrt{D\gamma}}{k-3},\left(\frac{1}{2}\right)^{(k-1)/2}A_0\right\}
\]
for $k\geq 4$.

\end{enumerate}
\end{lemma}
%
%

Below we present convergence rates for I-CBPG and thus complete our theoretical developments. Our first two results, Theorem \ref{theorem:convergence-bounded-error} and Corollary \ref{cor:convergence-bounded-error}, cover merely fixed errors. A reader familiar with the analyses of cyclic BPG methods in \cite{Beck13, Beck17} will notice that the constant $\gamma$ in Theorem \ref{theorem:convergence-bounded-error} differs by a factor of $4$ from that in \cite[Theorem 11.18]{Beck17}. This constant, and thus the rate in the exact computation setting, is recoverable from our analysis with minor modification. Namely, by replacing the Cauchy-Schwarz derived bounds in equations \eqref{eq:CS-app-1} and \eqref{eq:CS-app-2}, we can recover the aforementioned constant. The cost, however, is that the dependence in Lemma \ref{lemma:prox:convex:sufficient-decrease} on $\{\delta_k\}_{k\geq 1}$ deteriorates from $\mathcal{O}(\delta_k)$ to $\mathcal{O}(\delta_k^{1/2})$.

\begin{theorem}[Convergence of I-CBPG: Fixed Error Case]
\label{theorem:convergence-bounded-error}
Let $\{x^k\}_{k\geq 0}$ and $\{\delta_k\}_{k\geq 1}$ denote the sequences of iterates and error tolerances generated by I-CBPG (Algorithm \ref{algo.main}). If the error tolerance sequence is fixed ($\delta_\ell=\delta\geq 0$ for $\ell\geq 1$) then for any $k \ge 2$,
\begin{equation}\label{eqn:convergence-bounded-error}
F(x^k) - F^* \le \max\left\{  \left(\frac{1}{2} \right)^{(k-1)/2} \left( F(x^0)  - F^*  \right ) 		, \;\; \frac{4\gamma \left( F(x^0)  - F^*-u \right ) }{(k-1)(F(x^0)  - F^*  + 3u)} + u	\right\} ,
\end{equation}
where
\[
\gamma= \frac{8p(L_f + L_{\max})^2R(x^0)^2}{L_{\min}} , \quad u  = \sqrt{  L_{\min}\left[3p +    \frac{L^2_{\max}}{4}\left(  \frac{R(x^0)\sqrt{2} + \sqrt{p\delta}}{(L_f + L_{\max})R(x^0)} \right)^2\right]\delta \gamma}\;.
\]
\end{theorem}

\begin{proof}
The result is immediate upon invoking the technical recurrence lemma (Lemma \ref{lemma:recurrence-technical}(i)) with $A_k = F(x^k) - F^*$, $\gamma$ and $u$ as in the statement of the theorem, and 
\[
\Delta =   L_{\min}\left[3p +    \frac{L^2_{\max}}{4}\left(  \frac{R(x^0)\sqrt{2} + \sqrt{p\delta}}{(L_f + L_{\max})R(x^0)} \right)^2\right]\delta .
\]
\end{proof}

\begin{corollary}[Convergence of I-CBPG: Fixed Error Case (Restated)]\label{cor:convergence-bounded-error}
Under the same assumptions and definitions of $u$ and $\gamma$ in Theorem \ref{theorem:convergence-bounded-error}, if $\epsilon > u$ the iterates of I-CBPG (Algorithm \ref{algo.main}) achieve $F(x^k) - F^*\leq\epsilon$ for $k\geq K${,} where
\begin{equation}
K = 1+\left\lceil\max \left\{\frac{2}{\log 2}\cdot\log \frac{F(x^0)-F^*}{\epsilon},  \frac{4\gamma(F(x^0)-F^*-u)}{(\epsilon-u)(F(x^0)-F^*+3u)}   \right\} \right\rceil.
\end{equation}

\end{corollary}

\begin{proof}
Clearly, the expression for $K$ is the smallest $k\geq 2$ ensuring the right-hand side of \eqref{eqn:convergence-bounded-error} from Theorem \ref{theorem:convergence-bounded-error} is less than or equal $\epsilon$.
\end{proof}

Below we present the convergence rate when the error tolerance sequence $\{\delta_k\}_{k\geq 1}$ decreases at the sublinear rate $\mathcal{O}(1/k^2)$. It is appropriate to reiterate that we are not aware of any works on inexact cyclic coordinate descent type methods that establish a rate of decrease on the error sequence that preserves standard convergence rates with exception to \cite{Hua16} which considers linearly convergent schemes for strongly convex minimization. Section \ref{sec:numerical}'s numerical experiments strikingly illustrate the benefits of such a decreasing error tolerance sequence and  imply the potential for immense computational savings by permitting higher error tolerances during early iterations.  

\begin{theorem}[Convergence of I-CBPG: Decreasing Error Case]
\label{theorem:convergence-decreasing-error}
Let $\{x^k\}_{k\geq 0}$ and $\{\delta_k\}_{k\geq 1}$ denote the sequences of iterates and error tolerances generated by I-CBPG (Algorithm \ref{algo.main}). If the error tolerance sequence dynamically decreases at the sublinear rate $\mathcal{O}(1/k^2)${,} then for any $k \ge 4$, 
\begin{equation}\label{eqn:convergence-decreasing-error}
F(x^k) - F^* \leq \max\left\{\left(\frac{1}{2}\right)^{(k-1)/2}[F(x^0)-F^*],\frac{16\gamma}{k-3},\frac{8\sqrt{D\gamma}}{k-3}\right\}{,}
\end{equation}
where
\[
\gamma = \frac{8p(L_f+L_{\max})^2R(x^0)^2}{L_{\min}}{,}\quad D =   \tilde{D}L_{\min}\left[3p +    \frac{L^2_{\max}}{4}\left(  \frac{R(x^0)\sqrt{2} + \sqrt{p\delta_1}}{(L_f + L_{\max})R(x^0)} \right)^2\right]{,}
\]
and $\tilde{D}>0$ is a constant satisfying $\delta_k\leq\frac{\tilde{D}}{k^2}$  for all $k\geq 1$.
\end{theorem}

\begin{proof}
As in the proof of Theorem \ref{theorem:convergence-bounded-error}, the result rests on appropriate identification of the sequence $\{A_\ell\}_{\ell\geq 0}$ and the constants $\gamma$, $D$, and $\lambda$ in Lemma \ref{lemma:recurrence-technical}(ii). The identification here is more straightforward than in the fixed error case. Clearly, a quick examination of \eqref{eq:recur-main} from Lemma \ref{lemma:prox:convex:sufficient-decrease} shows we should choose $\gamma$ as given in the theorem statement, and $\Delta_\ell=\frac{D}{\ell^2}$ to ensure
\[
\frac{1}{\gamma}A_{k+1}^2 \leq A_{k} - A_{k+1} +\frac{D}{k^2}.
\]
Invoking Lemma \ref{lemma:recurrence-technical}, then, we achieve
\[
F(x^k) - F^*=A_k\leq\max\left\{\left(\frac{1}{2}\right)^{(k-1)/2}[F(x^0)-F^*],\frac{16\gamma}{k-3},\frac{8\sqrt{D\gamma}}{k-3}\right\}
\]
for $k\geq 4$.
\end{proof}

\begin{corollary}[Convergence of I-CBPG: Decreasing Error Case (Restated)]\label{cor:convergence-decreasing-error}
Under the same assumptions and definitions of $\gamma$ and $D$ in Theorem \ref{theorem:convergence-decreasing-error}, the iterates of I-CBPG (Algorithm \ref{algo.main}) achieve $F(x^k) - F^* \leq\epsilon$ for $k\geq K${,} where
\[
K =  \left\lceil \max \left\{1 + \frac{2}{\log 2}\cdot\log\left( \frac{F(x^0)-F^{*}}{\epsilon}\right),3+\frac{16\gamma}{\epsilon},  3 + \frac{8\sqrt{D\gamma}}{\epsilon}\right\} \right\rceil.
\]
\end{corollary}

\begin{proof}
Clearly, the expression for $K$ is the smallest $k\geq 2$ ensuring the right-hand side of \eqref{eqn:convergence-decreasing-error} from Theorem \ref{theorem:convergence-decreasing-error} is less than or equal to $\epsilon$.
\end{proof}


%
%


\section{Numerical Experiments}\label{sec:numerical}
In this section we present two sets of numerical experiments that demonstrate I-CBPG's performance capabilities.\footnote{Data and related code for these experiments will be made available upon reasonable request.} The main lesson of our experiments is that block cyclic methods with dynamically decreasing levels of inexactness can often beat block randomized methods with fixed errors. This contrasts with the typical superiority of randomized block proximal methods relative to their cyclic counterparts when both share the same fixed error tolerance. This reversal highlights the computational advantage offered by a dynamically decreasing error tolerance. The articles \cite{Richtarik16} and \cite{Richtarik14} thoroughly explore the numerical performance of exact and inexact block proximal gradient methods. Thus, for the sake of focusing on the effects of dynamic errors, we will not fully reproduce the experiments contained in the aforementioned works. However, our experimental setups closely resemble those in \cite{Richtarik14}, which concern randomized block proximal gradient methods with fixed errors. 

In Section \ref{sec:numerical-smooth}, we present experimental results for smooth minimization, where our objective functions are (non-strongly) convex quadratics. In Section \ref{sec:numerical-non-smooth}, we show experimental results for the famous LASSO problem, a common non-smooth optimization testbed.

\subsection{Experimental Results for Smooth Convex Minimization}
\label{sec:numerical-smooth}

For our experiments concerning smooth minimization, we selected the ordinary least squares (OLS) problem
\begin{equation}\label{eq:OLS}
\min_{x\in\mathbb{R}^n}\frac{1}{2}\|Ax-b\|_2^2,
\end{equation}
where $A$ has a $10$-block angular structure specified as
\[
A:=\left[\begin{array}{c} C \\ 
\hline 
D\end{array}\right], \quad C:=\begin{bmatrix}C_1 &  & \\
& \ddots & \\
& & C_{10}\end{bmatrix}, \quad D:=\begin{bmatrix} D_1 & \hdots & D_{10}\end{bmatrix}{,}
\]
and the dimensions of each $C_i$ and $D_i$ matrix depend upon the type of experiment. Our experiments follow the notation and setup of Section 8.1 of \cite{Richtarik14}. OLS problems featuring this block angular structure abound in optimization, with well-known instances found in scheduling, planning, and optimal control. Each of the ten variable blocks correspond to the columns of one of the block matrices
\[
A_i:=\begin{bmatrix} \\ C_i \\ \\ D_i\end{bmatrix}.
\]
Setting $U_1,\ldots,U_{10}$ to be the permutation matrices corresponding to this block decomposition, we conclude \eqref{eq:OLS} fits within the template of \eqref{eq.problem} with $\Psi_1,\ldots,\Psi_{10}\equiv 0$. 

Now, we turn to the subproblem which defines the iterate updates and our conjugate-gradient method-based approach to solving it. Letting $B_i:=A_i^\top A_i$ facilitates writing our iterate defining subproblem as
\begin{equation}\label{num:subprobOLS}
\argmin_{t\in\mathbb{R}^{n_i}}f(x^{k,i-1}+U_it)=\frac{1}{2}\|A_it-\tilde{b}^{k,i}\|_2^2{,}
\end{equation}
where $\tilde{b}^{k,i}=Ax^{k,i-1}-b$. Expanding the subproblem objective function as
\[
f(x^{k,i-1}+U_it)=f(x^{k,i-1})+\inner{\nabla_i f(x^{k,i-1})}{t}+\frac{1}{2}\inner{B_i t}{t},
\]
it is immediate that the block smoothness condition \eqref{eq:smooth:block} is satisfied with $L_1,\ldots,L_{10}=1$.

Following \cite{Richtarik14}, we solve \eqref{num:subprobOLS} by solving the equivalent system of equations that arises from its first-order optimality condition, $B_i t=A_i^\top\tilde{b}^{k,i}$. Like \cite{Richtarik14}, we opt to solve this system approximately using the conjugate gradient method (CGM). When $b$ belongs to the span of $A$'s columns, then the optimal value for \eqref{num:subprobOLS} is $0$ and the suboptimality gap at $t$ is readily measurable to be $f(x^{k,i-1}+U_it)$ itself. In the more general (and more likely in practice) situation that $b$ does not equal $Ax$ for some $x$, the optimal value is not known a priori, and by extension the suboptimality gap at a given $t$ is not readily measured exactly. As we will soon discuss, this is due to the computational burden associated with calculating the smallest singular values of each $B_i$. In this more realistic setting, CGM is terminated when the residual, $\|B_i t+A_i^\top\tilde{b}^{k,i}\|_2$, or the relative residual, $\|B_i t+A_i^\top\tilde{b}^{k,i}\|_2/\|A_i^\top\tilde{b}^{k,i}\|_2$, falls below some desired threshold. For our approach to subproblem \eqref{num:subprobOLS}, we opt to terminate CGM when the following conditions hold: i) the monotonic decrease condition \eqref{eq.mondec} is satisfied, and ii) the first of the residuals above falls below a predefined numerical threshold we denote by $\tilde{\delta}_k$.

The CGM termination threshold $\tilde{\delta}_k$ described above is distinct from, but closely related to, the error tolerance $\delta_k$. Indeed, we have chosen to define $\tilde{\delta}_k$ in relation to $\delta_k$ via the decomposition  $\tilde{\delta}_k:=\delta_k\cdot2\min_{i=1,\ldots,10}\sigma_{\min}(B_i)^2$. This notational convention is driven by the inherent difficulty in bounding \eqref{num:subprobOLS}'s suboptimality gap in terms of the residual, and reveals a key benefit of the dynamic error regime. The residual's tight upper bound on the suboptimalty gap for \eqref{num:subprobOLS} is given by 
\[
f(x^{k,i-1}+U_i t)-\argmin_{t'\in\mathbb{R}^{n_i}}f(x^{k,i-1}+U_it')\leq \frac{1}{2\sigma_{\min}(B_i)^2} \|B_i t+A_i^\top\tilde{b}^{k,i}\|_2^2{,}
\]
where $\sigma_{\min}(B_i)$ is the smallest singular value of $B_i$, a quantity which may be as expensive to compute as solving the original problem outright. In light of this, if at every I-CBPG iteration we run CGM to find $t$ satisfying
\begin{equation}\label{eq:CGM-stopping}
\|B_i t+A_i^\top\tilde{b}^{k,i}\|_2\leq \tilde{\delta}_{k+1}
\end{equation}
for the OLS problem \eqref{eq:OLS}, then we get the resulting suboptimality gap bound
\[
f(x^{k,i-1}+U_i t)-\argmin_{t'\in\mathbb{R}^{n_i}}f(x^{k,i-1}+U_it')\leq\frac{\tilde{\delta}_{k+1}}{2\sigma_{\min}(B_i)^2}\leq\delta_{k+1}.
\]
Consequently, if $\tilde{\delta}_k$ (or equivalently, $\delta_k$) shrinks at the sublinear rate $\mathcal{O}(1/k^2)$, we can ensure that I-CBPG will find an $x$ with suboptimality gap less than $\epsilon$ in $\mathcal{O}(1/\epsilon)$ iterations.  Therefore, through a dynamically decreasing sequence of residual tolerances $\tilde{\delta}_k$ for terminating CGM, one can find such an $x$ at the above rate without needing to determine numerical values of $\delta_k$ or $2\sigma_{\min}(B_i)^2$. The conceptual and notational convenience offered by this property is the motivating force behind our definition of $\tilde{\delta}_k$ above, and is introduced precisely to obviate the need for separate numerical identification of $\delta_k$ and $\sigma_{\min}(B_i)^2$. By contrast, finding such an $x$ in the fixed error regime  is only possible after incurring the heavy computational cost of lower bounding the smallest singular values of the $B_i$ matrices. 

Before turning to the results of our experiments, a few final words regarding the data generation processes for our experiments are in order. In particular, our two experiments consider two structures for the matrix $A$, which we refer to as wide and tall matrices depending on whether there are more columns than rows or vice versa. We fix $N=10^5$ throughout. For the wide setting, we take $A$ to have dimensions $(N+100)\times 2N$ so each $C_i$ and $D_i$  have dimensions $N\times N/5$ ($n_i=N/5$) and $100\times N/5$, respectively. For the tall setting, $A$ has dimensions $(N+100)\times .5N$ so each $C_i$ and $D_i$ have dimensions $N\times .5N/10$ ($n_i=.5N/10$) and $100\times .5N/10$, respectively.  In each case, $C_i$ is a sparse, randomly generated matrix with entries in $[0,1]$, with roughly $20$ non-zero entries per column, that has been padded with the addition of an identity matrix. Similarly, each linking block, $D_i$, is a randomly generated matrix with entries in $[0,1]$ with roughly one-tenth of its entries non-zero. In each experiment, we set $b=Ax^*$ where $x^*$ is a randomly generated vector with entries in $[0,1]$.

With the description of our experimental setups and  subproblem solution method complete, we now explore the effect of different error tolerance levels on the runtime performance of the OLS problem \eqref{eq:OLS} for both I-CBPG and its sibling from \cite{Richtarik14}, which instead picks blocks randomly according to the uniform distribution. For the sake of symmetry, we shall refer to the randomized scheme in \cite{Richtarik14} as the \emph{Inexact Randomized Block Proximal Gradient Method (I-RBPG)}. As we foreshadowed in the section introduction, the randomized I-RBPG will often beat I-CBPG in the fixed error regime, though I-CBPG with dynamically decreasing error tolerances will generally outpace both. We compare the algorithms' behavior under a dynamic rule for I-CBPG that sets $\tilde{\delta} = \left[ F(x^0)-F^*\right] /k^2$ with I-CBPG and I-RBPG under three different constant rules that fix $\tilde{\delta}_k$ at $10^{-2}$, $10^{-4}$, or $10^{-6}$ for all cycles, respectively. As in \cite{Richtarik14}, we terminated each of the algorithms once $F(x^k)-F^*$ fell below $10^{-1}$. Plots of the difference $F\left(x^k\right) - F^*$ against both CPU time and the number of iterations $k$ are presented in Figure \ref{fig:ols-wide} and Table \ref{tab:widematrix-ols} for the first case, where $A$ is $(N + 100) \times 2N$, and in Figure \ref{fig:ols-tall} and Table \ref{tab:tallmatrix-ols} for the second case, where $A$ is $(N+100) \times 0.5N$. To fairly compare the I-CBPG and I-RBPG algorithms, an I-RBPG cycle will denote 10 block updates. 
We performed a number of simulations of both types, but limit our discussion to a single instance of each for definiteness. Owing to the random data generation procedure, we did observe some minor variation in the quantitative values of the various ratios discussed below across different  problem instances, but the qualitative results are consistent. 

Among the fixed error tolerance regimes ($\tilde{\delta}_k$, and therefore $\delta_k$, constant), it is apparent for both types of $A$ matrix and for both I-CBPG and I-RBPG that larger error tolerances (as measured by a higher value of $\tilde{\delta}_k$) translate to achieving a given level of accuracy in a shorter amount of CPU time than is possible with a more stringent (lower) value of $\tilde{\delta}_k$, although the algorithm may run through more cycles in total to achieve a given suboptimality gap when the error is larger. This property of larger fixed values of $\tilde{\delta}_k$ translating to shorter CPU time costs echoes the results of \cite{Richtarik14}.

For the wide matrix, examining the performance of I-CBPG with the dynamic rule $\tilde{\delta}_k = \mathcal{O}(1/k^2)$ in Table \ref{tab:widematrix-ols}, we find that I-CBPG with the dynamic rule uses substantially less CPU time than either I-CBPG or I-RBPG under any of the constant rules. Across the rules we have studied, the time savings range from over 60\% at the low end to more than 94\% at the high end.  

The tall $A$ setting, reflected in Table \ref{tab:tallmatrix-ols}, exhibits similar performance. Although I-RBPG with a fixed error tolerance of $\tilde{\delta}_k=10^{-2}$ achieves a modest victory of an $8\%$ time savings over I-CBPG with a dynamic error tolerance regime, the latter algorithm vastly outpaces I-RBPG with less permissive error tolerances. Specifically, I-CBPG with a dynamic error tolerance achieves convergence in $46\%$ and $76\%$ less time, respectively than I-RBPG with the constant rules $\delta_k = 10^{-4}$ and $\delta_k = 10^{-6}$. 
\newpage



\begin{figure}[htb!]
\begin{subfigure}{.5\textwidth}
  \centering
  \includegraphics[width=0.8\linewidth]{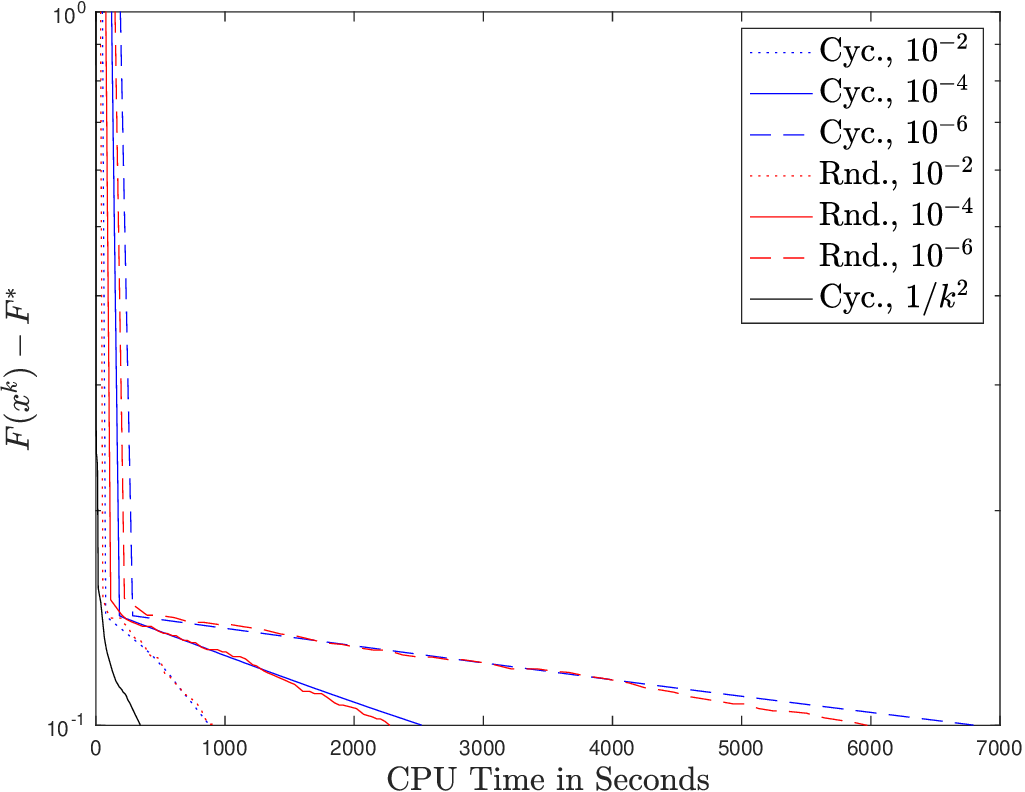}
   \caption{$F(x^k)-F^*$ vs. Total CPU Time}
  \label{fig:ols-wide-sub1}
\end{subfigure}%
\begin{subfigure}{.5\textwidth}
  \centering
  \includegraphics[width=0.8\linewidth]{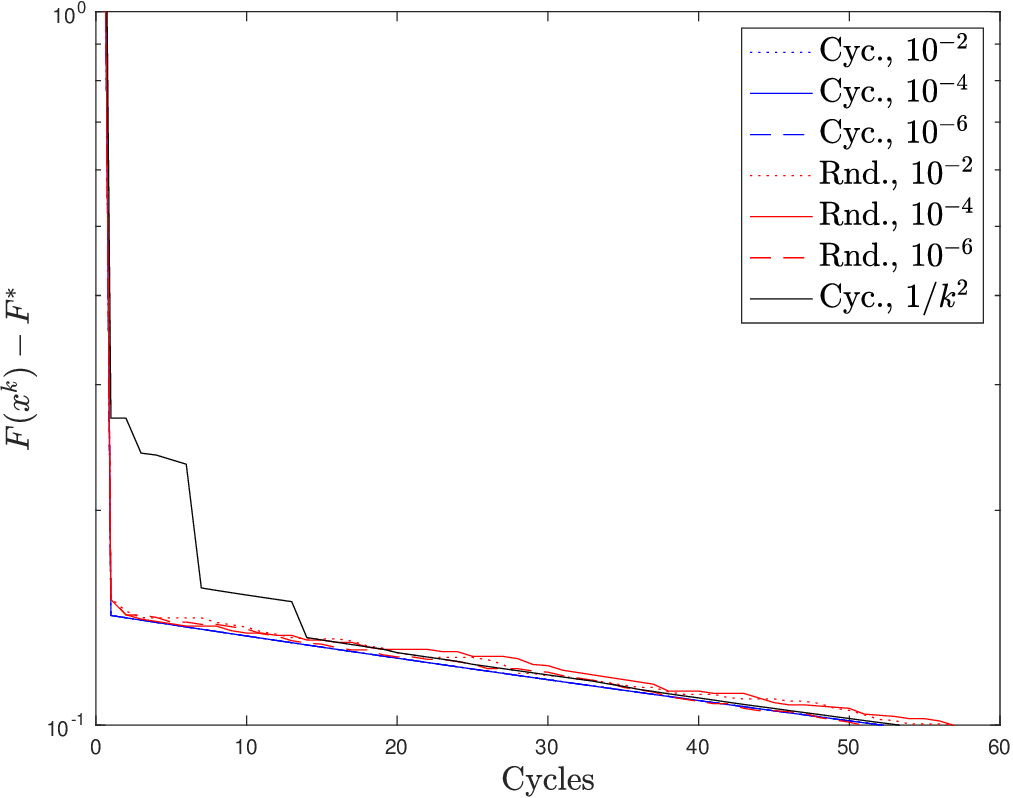}
   \caption{$F(x^k)-F^*$ vs. Cycles}
  \label{fig:ols-wide-sub2}
\end{subfigure}%

\caption{ I-CBPG performance graphs for OLS problem \eqref{eq:OLS} with wide $A$ (size $(N+100) \times 2N$, $N = 10^5$), by block scheme (cyclic or random) and $\tilde{\delta}_k$ (conjugate gradient subproblem residual tolerance rule).}
\label{fig:ols-wide}
\end{figure}



\begin{figure}[htb!]
\begin{subfigure}{.5\textwidth}
  \centering
  \includegraphics[width=0.8\linewidth]{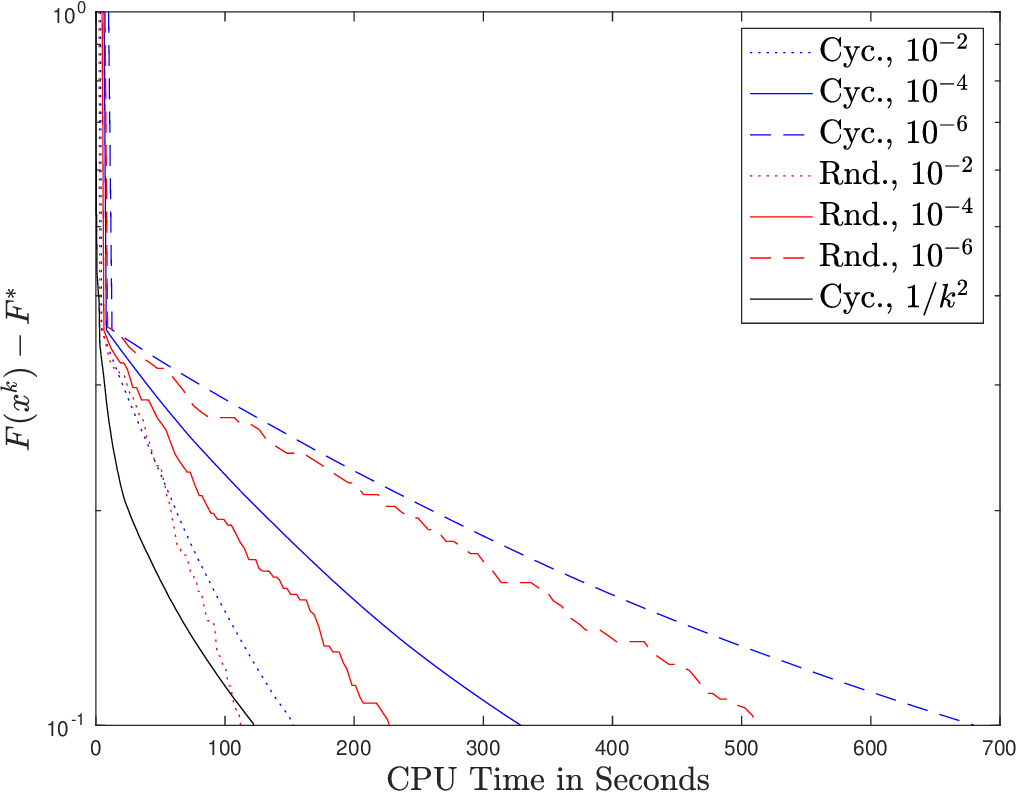}
   \caption{$F(x^k)-F^*$ vs. Total CPU Time}
  \label{fig:ols-tall-sub1}
\end{subfigure}%
\begin{subfigure}{.5\textwidth}
  \centering
  \includegraphics[width=0.8\linewidth]{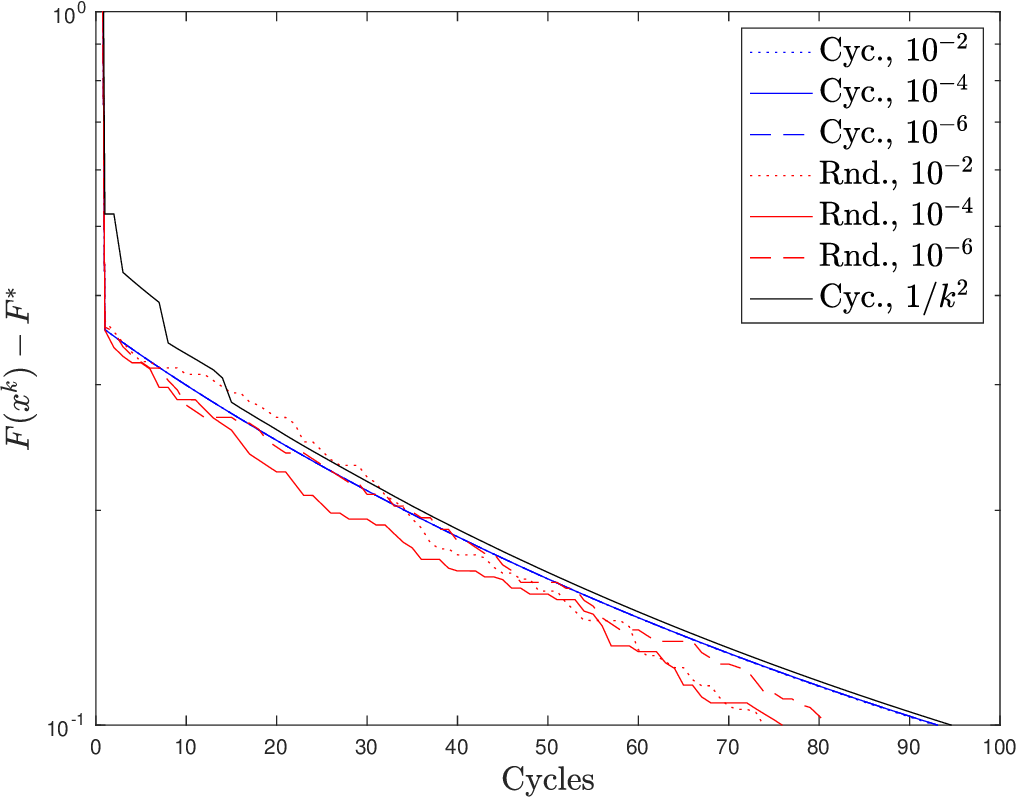}
   \caption{$F(x^k)-F^*$ vs Cycles}
  \label{fig:ols-tall-sub2}
\end{subfigure}%

\caption{ I-CBPG performance graphs for OLS problem \eqref{eq:OLS} with tall $A$ (size $(N+100) \times 2N$, $N = 10^5$) by block scheme (cyclic or random) and $\tilde{\delta}_k$ (conjugate gradient subproblem residual tolerance rule).}
\label{fig:ols-tall}
\end{figure}

\newpage

\begin{table}[htb!]
 \caption{OLS problem: total cycles, total CPU time, CPU time by cycle, and suboptimality gap by cycle for wide $A$ (size $(N+100)\times 2N$, $N = 10^5$).$^{8}$}
    \centering
    \footnotesize
    \begin{tabular}{|c|c|c|c|c|c|c|c|c|}
    \hline 
  \multicolumn{2}{| c | }{Block Scheme} & Cyclic & Cyclic & Cyclic & Random & Random & Random  & Cyclic   \\
  \multicolumn{2}{| c | }{Error Tolerance $\tilde{\delta}_k$} & $10^{-2}$ & $10^{-4}$ & $10^{-6}$ & $10^{-2}$ & $10^{-4}$ & $ 10^{-6}$ &  $k^{-2}$   \\
  \hline \multicolumn{9}{|c|}{Results at Convergence} \\
  \hline \multicolumn{2}{|c|}{Total Cycles} & 53 & 53 & 53 & 57 & 57 & 51 & 54\\
    \multicolumn{2}{|c|}{Total Time} & 878.43  &  2520.12 &  6792.83 &  903.86 & 2280.82  & 5989.25  & 343.08 \\ 
  \hline \multicolumn{9}{|c|}{Results by Cycle Number $k$} \\
  \hline

  0 &  Time (sec)  & 0  & 0  & 0  & 0  & 0  & 0  & 0 \\  
    & $F(x^k)-F^*$ & 47.89  & 47.89  & 47.89  & 47.89  & 47.89  & 47.89  & 47.89 \\  
  \hline
1 &  Time (sec)  & 77.10  & 183.28  & 285.45  & 54.53  & 115.19  & 221.67  & 1.87 \\  
    & $F(x^k)-F^*$ & 0.14  & 0.14  & 0.14  & 0.15  & 0.15  & 0.15  & 0.27 \\  
\hline
 10 &  Time (sec)  & 21.22  & 45.13  & 128.54  & 18.53  & 42.31  & 129.29  & 3.14 \\  
    & $F(x^k)-F^*$ & 0.13  & 0.13  & 0.13  & 0.14  & 0.13  & 0.14  & 0.15 \\  
\hline
 20 &  Time (sec)  & 13.19  & 43.27  & 127.12  & 14.94  & 30.34  & 130.91  & 6.37 \\  
    & $F(x^k)-F^*$ & 0.12  & 0.12  & 0.12  & 0.13  & 0.13  & 0.12  & 0.13 \\  
\hline
 30 &  Time (sec)  & 14.37  & 43.74  & 126.21  & 13.13  & 30.24  & 118.27  & 5.77 \\  
    & $F(x^k)-F^*$ & 0.12  & 0.12  & 0.12  & 0.12  & 0.12  & 0.12  & 0.12 \\  
\hline
 40 &  Time (sec)  & 12.79  & 45.83  & 127.31  & 10.01  & 26.06  & 111.15  & 5.67 \\  
    & $F(x^k)-F^*$ & 0.11  & 0.11  & 0.11  & 0.11  & 0.11  & 0.11  & 0.11 \\  
\hline
 50 &  Time (sec)  & 11.46  & 44.96  & 125.52  & 15.65  & 33.84  & 116.10  & 7.07 \\  
    & $F(x^k)-F^*$ & 0.10  & 0.10  & 0.10  & 0.10  & 0.11  & 0.10  & 0.10  

			\\          \hline
    \end{tabular}
    \normalsize
   
    \label{tab:widematrix-ols}
\end{table}


\begin{table}[htb!]
\caption{OLS problem: total cycles, total CPU time, CPU time by cycle, and suboptimality gap by cycle for tall $A$ (size $(N+100) \times .5N$, $N = 10^5$).}
    \centering
    \footnotesize
    \begin{tabular}{|c|c|c|c|c|c|c|c|c|}
    \hline 
      \multicolumn{2}{| c | }{Block Scheme} & Cyclic & Cyclic & Cyclic & Random & Random & Random  & Cyclic   \\
  \multicolumn{2}{| c | }{Error Tolerance $\tilde{\delta}_k$} & $10^{-2}$ & $10^{-4}$ & $10^{-6}$ & $10^{-2}$ & $10^{-4}$ & $ 10^{-6}$ &  $k^{-2}$   \\
  \hline \multicolumn{9}{|c|}{Results at Convergence} \\
  \hline \multicolumn{2}{|c|}{Total Cycles} & 93 & 94 & 94 & 74 & 76 & 81 & 95\\
  \multicolumn{2}{|c|}{Total Time} & 152.79 &  328.77 &  679.65 &  112.32 & 227.29  & 508.86  & 121.93 \\ 

  \hline \multicolumn{9}{|c|}{Results by Cycle Number $k$} \\
  \hline

 0 &  Time (sec)  & 0 & 0 & 0 & 0 & 0 & 0 & 0\\  
    & $F(x^k)-F^*$ & 45.82  & 45.82  & 45.82  & 45.82  & 45.82  & 45.82  & 45.82 \\  
 \hline 
  1 &  Time (sec)  & 4.47  & 8.10  & 12.51  & 3.13  & 6.22  & 9.09  & 0.32 \\  
    & $F(x^k)-F^*$ & 0.36  & 0.36  & 0.36  & 0.37  & 0.36  & 0.36  & 0.52 \\  
 \hline 
 20 &  Time (sec)  & 1.81  & 3.66  & 7.27  & 1.72  & 3.29  & 7.55  & 0.71 \\  
    & $F(x^k)-F^*$ & 0.25  & 0.25  & 0.25  & 0.27  & 0.23  & 0.25  & 0.26 \\  
 \hline 
 40 &  Time (sec)  & 1.61  & 3.44  & 7.03  & 1.34  & 3.04  & 7.24  & 1.63 \\  
    & $F(x^k)-F^*$ & 0.18  & 0.18  & 0.18  & 0.17  & 0.16  & 0.18  & 0.19 \\  
 \hline 
 60 &  Time (sec)  & 1.56  & 3.39  & 7.17  & 1.76  & 2.94  & 5.21  & 1.73 \\  
    & $F(x^k)-F^*$ & 0.14  & 0.14  & 0.14  & 0.13  & 0.13  & 0.14  & 0.14 \\  
 \hline 
 80 &  Time (sec)  & 1.46  & 3.43  & 7.20  &    &    & 6.41  & 1.68 \\  
    & $F(x^k)-F^*$ & 0.11  & 0.11  & 0.11  &    &    & 0.10  & 0.12 	\\          \hline

    \end{tabular}
    \normalsize
    
    \label{tab:tallmatrix-ols}
\end{table}



\subsection{Experimental Results for Non-smooth Convex Minimization}
\label{sec:numerical-non-smooth}

We selected a common testbed for our experiments concerning non-smooth minimization, the LASSO problem,
\begin{equation}\label{num:problem}
\min_{x\in \mathbb{R}^n} \frac{1}{2}\| Ax - b\|_2^2 + \lambda \| x\|_1{,}
\end{equation}
with $A \in \mathbb{R}^{m \times n}$, $b \in \mathbb{R}^m$, $\lambda > 0$. This problem fits within the template of \eqref{eq.problem} by recognizing $f(x) = \frac{1}{2}\| Ax - b\|_2^2 $ and $\Psi_i(U_i^Tx) = \lambda \|  U_i^T x  \|_1$ for $i = 1, \ldots, p${,} with $(U_1,\ldots,U_p)=I_n$.  

Throughout, we follow the setup of Section 8.2 of  \cite{Richtarik14JOTA} and explore two cases with $N = 10^5$: in the first, the matrix $A$ is wide with size $N \times 2N$; in the second, $A$ is tall with size $N \times 0.5N$. In both settings, $A$ is a randomly generated sparse matrix with approximately 20 nonzero entries per column. The nonzero entries are generated according to the uniform distribution on $[0,1]$. We subdivide both $A$ matrices into $p = 10$ blocks $A_i$ of equal size, and add to each block an identity matrix padded with zeros to guarantee that each $A_i$ is of full rank. With this, we have that the block smoothness condition \eqref{eq:smooth:block} is satisfied for $B_i  =A_i^TA_i$ and $L_i$ =1. We set $\lambda  = 0.01$ when $A$ is $N \times 2N$, and $\lambda = 0.1$ when  $A$ is $N \times 0.5N$. In both cases, we generated $b=\tilde{b}/[2\lambda\|A^T \tilde{b}\|_{\infty}]$, where $\tilde{b}$ is a normal random variable of appropriate dimension, to guarantee that the zero vector is an optimal solution so $F^*$ is known.

\begin{figure}[htb!]
\begin{subfigure}{.5\textwidth}
  \centering
  \includegraphics[width=0.8\linewidth]{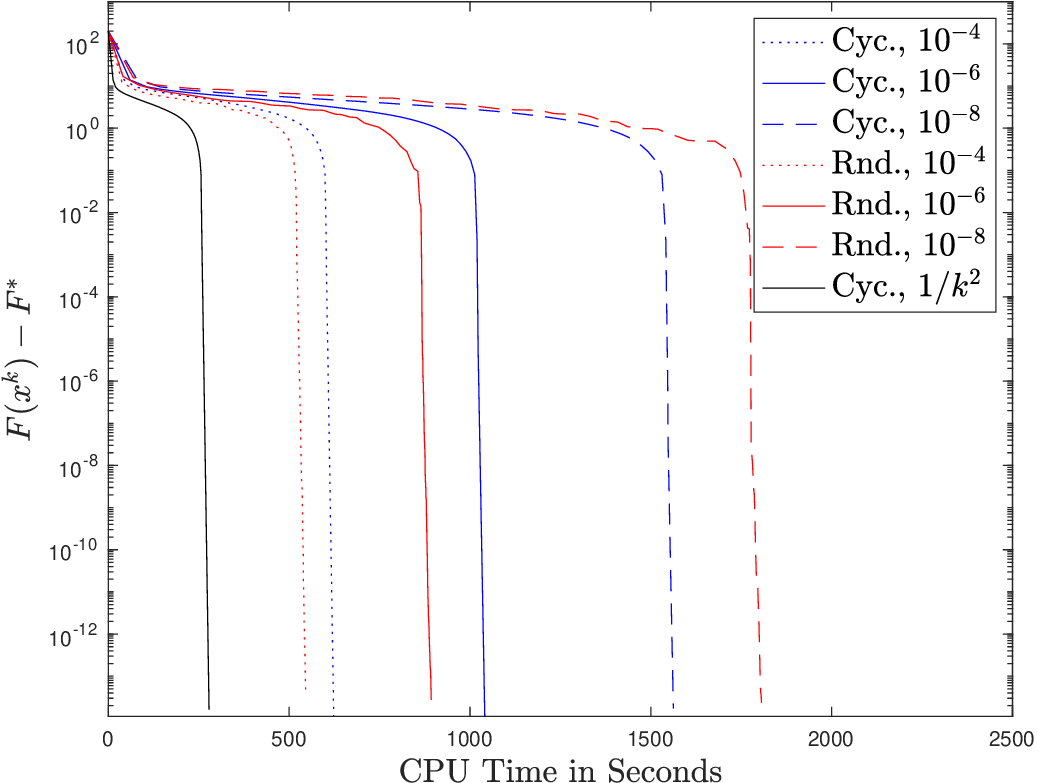}
   \caption{$F(x^k)-F^*$ vs. Total CPU Time}
  \label{fig:widematrix_sub1}
\end{subfigure}%
\begin{subfigure}{.5\textwidth}
  \centering
  \includegraphics[width=0.8\linewidth]{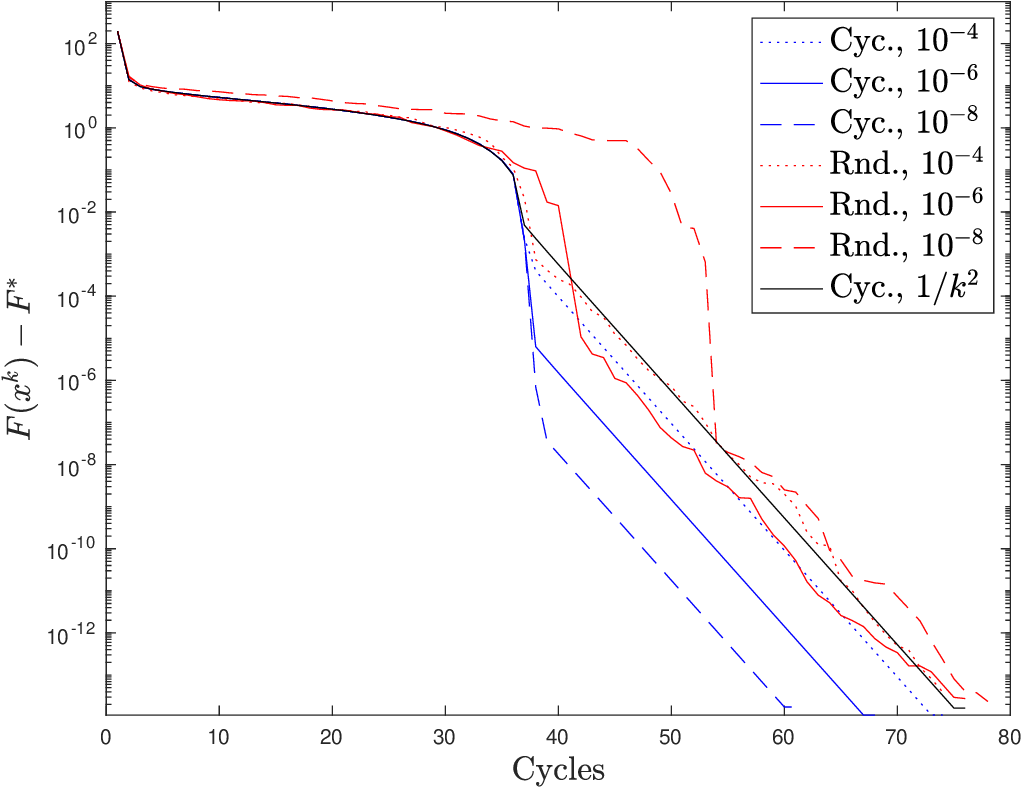}
   \caption{$F(x^k)-F^*$ vs Cycles}
  \label{fig:widematrix_sub2}
\end{subfigure}%

\caption{ I-CBPG performance graphs for LASSO problem \eqref{num:problem} with wide $A$ (size $N \times 2N$, $N = 10^5$) by block scheme (cyclic or random) and $\delta_k$ (error tolerance rule).}
\label{fig:widematrix}
\end{figure}

\begin{figure}[htb!]
\begin{subfigure}{.5\textwidth}
  \centering
  \includegraphics[width=.8\linewidth]{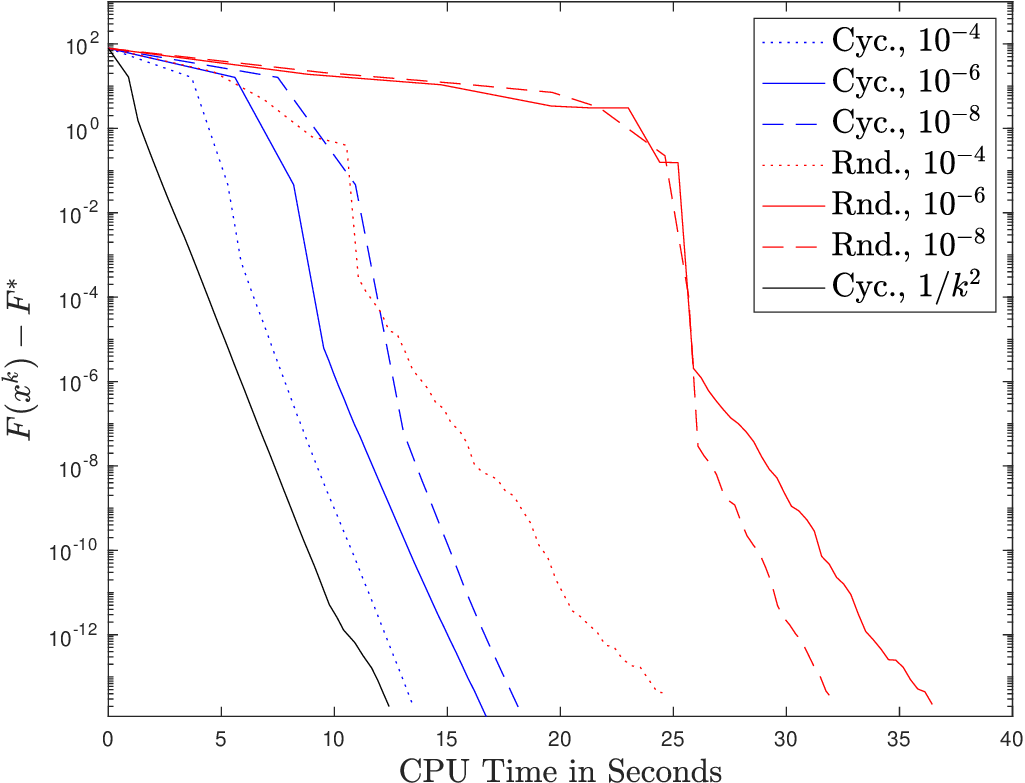}
  \caption{$F(x^k)-F^*$ vs. Total CPU Time}
  \label{fig:tallmatrix_sub1}
\end{subfigure}%
\begin{subfigure}{.5\textwidth}
  \centering
  \includegraphics[width=.8\linewidth]{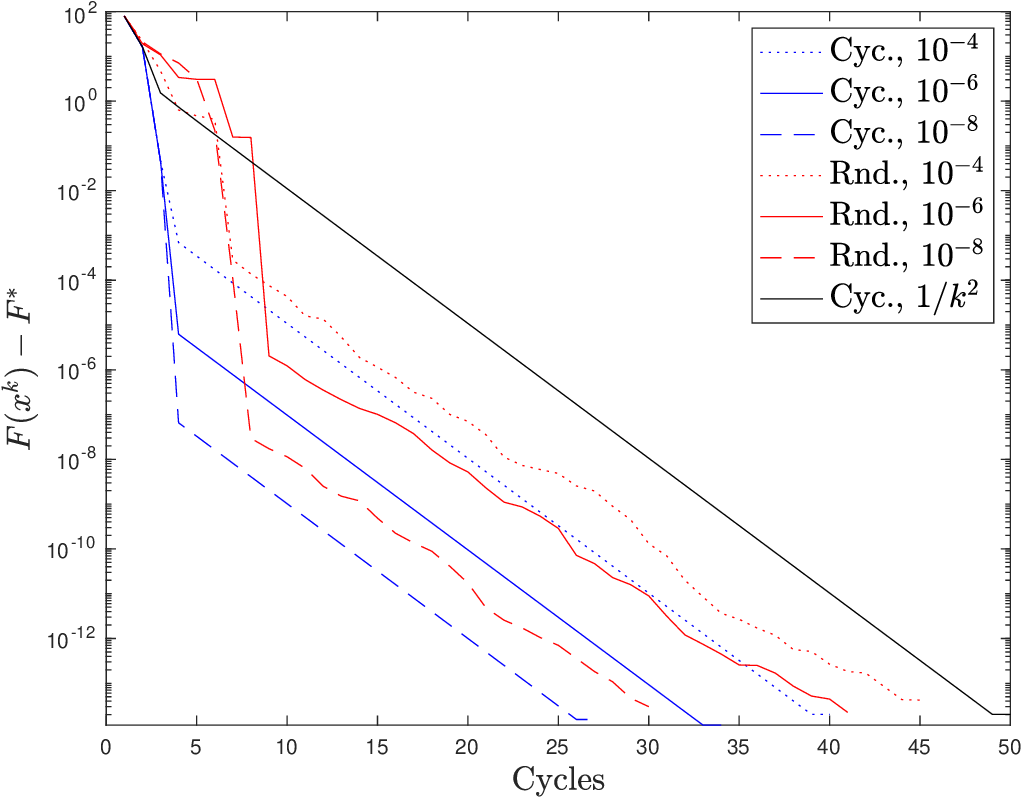}
   \caption{$F(x^k)-F^*$ vs. Cycles}
  \label{fig:tallmatrix_sub2}
\end{subfigure}%

\caption{ I-CBPG performance graphs for LASSO problem \eqref{num:problem} with tall $A$ (size $N \times 0.5N$, $N = 10^5$)  by block scheme (cyclic or random) and $\delta_k$ (error tolerance rule).}
\label{fig:tallmatrix}
\end{figure}




\begin{table}[htb!]
\caption{LASSO problem: total cycles, total elapsed CPU time, CPU time by cycle, and suboptimality gap by cycle for wide $A$ (size $N \times 2N$, $N = 10^5$).$^{12}$}
    \centering
    \footnotesize
    \begin{tabular}{|c|c|c|c|c|c|c|c|c|}
    \hline 
             \multicolumn{2}{| c | }{Block Scheme} & Cyclic & Cyclic & Cyclic & Random & Random & Random  & Cyclic   \\
  \multicolumn{2}{| c | }{Error Tolerance $\tilde{\delta}_k$} & $10^{-4}$ & $10^{-6}$ & $10^{-8}$ & $10^{-4}$ & $10^{-6}$ & $ 10^{-8}$ &  $k^{-2}$   \\
  \hline \multicolumn{9}{|c|}{Results at Convergence} \\
  \hline \multicolumn{2}{|c|}{Total Cycles} & 73& 67& 60& 73& 75& 77& 75 \\
  \multicolumn{2}{|c|}{Total Time} & 623.23& 1040.9& 1562.2& 546.14& 894.05& 1806.6&  278.45\\
  \hline \multicolumn{9}{|c|}{Results by Cycle Number $k$} \\
     \hline
         
 0 &  Time (sec)  & 0.00e+00  & 0.00e+00  & 0.00e+00  & 0.00e+00  & 0.00e+00  & 0.00e+00  & 0.00e+00 \\  
   & $F(x^k)-F^*$ & 1.98e+02  & 1.98e+02  & 1.98e+02  & 1.98e+02  & 1.98e+02  & 1.98e+02  & 1.98e+02 \\  
\hline
1 &  Time (sec)  & 4.37e+01  & 5.96e+01  & 7.99e+01  & 3.73e+01  & 4.04e+01  & 6.86e+01  & 1.31e+01 \\  
   & $F(x^k)-F^*$ & 1.40e+01  & 1.40e+01  & 1.40e+01  & 1.27e+01  & 1.69e+01  & 1.59e+01  & 1.40e+01 \\  
\hline
 5 &  Time (sec)  & 2.45e+01  & 3.87e+01  & 5.50e+01  & 2.44e+01  & 3.70e+01  & 4.34e+01  & 7.61e+00 \\  
   & $F(x^k)-F^*$ & 6.91e+00  & 6.91e+00  & 6.91e+00  & 6.11e+00  & 6.56e+00  & 8.43e+00  & 6.92e+00 \\  
\hline
20 &  Time (sec)  & 1.59e+01  & 3.20e+01  & 4.66e+01  & 1.28e+01  & 2.04e+01  & 4.00e+01  & 8.79e+00 \\  
   & $F(x^k)-F^*$ & 2.57e+00  & 2.57e+00  & 2.57e+00  & 2.64e+00  & 2.66e+00  & 3.98e+00  & 2.56e+00 \\  
\hline
40 &  Time (sec)  & 5.87e-01  & 6.30e-01  & 6.92e-01  & 6.91e-01  & 1.98e+00  & 2.69e+01  & 5.38e-01 \\  
   & $F(x^k)-F^*$ & 4.85e-05  & 7.55e-07  & 9.01e-09  & 1.93e-04  & 4.04e-04  & 7.55e-01  & 2.84e-04 \\  
\hline
60 &  Time (sec)  & 5.70e-01  & 6.38e-01  & 2.01e-02  & 6.58e-01  & 7.73e-01  & 1.21e+00  & 5.25e-01 \\  
   & $F(x^k)-F^*$ & 4.60e-11  & 7.16e-13  & 1.74e-14  & 1.02e-09  & 5.25e-11  & 2.22e-09  & 2.69e-10 

   \\ 
  \hline
    \end{tabular}
    
   \normalsize
  \label{tab:widematrix} 
\end{table} 



\begin{table}[htb!]
\caption{LASSO problem: total cycles, total elapsed CPU time, CPU time by cycle, and suboptimality gap by cycle for tall $A$ (size $2N \times N$, $N = 10^5$).$^{13}$}
    \centering
    \footnotesize
    \begin{tabular}{|c|c|c|c|c|c|c|c|c|}
    \hline 
          \multicolumn{2}{| c | }{Block Scheme} & Cyclic & Cyclic & Cyclic & Random & Random & Random  & Cyclic   \\
  \multicolumn{2}{| c | }{Error Tolerance $\tilde{\delta}_k$} & $10^{-4}$ & $10^{-6}$ & $10^{-8}$ & $10^{-4}$ & $10^{-6}$ & $ 10^{-8}$ &  $k^{-2}$   \\
  \hline \multicolumn{9}{|c|}{Results at Convergence} \\
  \hline \multicolumn{2}{|c|}{Total Cycles} & 39& 33& 26& 44& 40& 29& 49 \\
  \multicolumn{2}{|c|}{Total Time} & 13.49& 16.71& 18.22& 24.95& 36.64& 32.07&  12.42\\
  \hline \multicolumn{9}{|c|}{Results by Cycle Number $k$} \\
  \hline
         
 0 &  Time (sec)  & 0.00e+00  & 0.00e+00  & 0.00e+00  & 0.00e+00  & 0.00e+00  & 0.00e+00  & 0.00e+00 \\  
   & $F(x^k)-F^*$ & 7.89e+01  & 7.89e+01  & 7.89e+01  & 7.89e+01  & 7.89e+01  & 7.89e+01  & 7.89e+01 \\  
 \hline
1 &  Time (sec)  & 3.68e+00  & 5.60e+00  & 7.50e+00  & 4.66e+00  & 8.83e+00  & 9.50e+00  & 8.93e-01 \\  
   & $F(x^k)-F^*$ & 1.61e+01  & 1.61e+01  & 1.61e+01  & 2.18e+01  & 1.91e+01  & 2.10e+01  & 1.64e+01 \\  
\hline
10 &  Time (sec)  & 1.97e-01  & 2.62e-01  & 2.22e-01  & 3.56e-01  & 3.25e-01  & 2.82e-01  & 2.35e-01 \\  
   & $F(x^k)-F^*$ & 5.41e-06  & 4.87e-08  & 5.12e-10  & 1.60e-05  & 6.05e-07  & 6.30e-09  & 5.56e-03 \\  
\hline
20 &  Time (sec)  & 2.24e-01  & 2.45e-01  & 2.49e-01  & 3.42e-01  & 3.27e-01  & 2.78e-01  & 2.20e-01 \\  
   & $F(x^k)-F^*$ & 5.28e-09  & 4.76e-11  & 5.00e-13  & 3.53e-08  & 2.30e-09  & 4.84e-12  & 5.43e-06 \\  
\hline
30 &  Time (sec)  & 2.30e-01  & 2.87e-01  &    & 3.80e-01  & 3.22e-01  &    & 2.20e-01 \\  
   & $F(x^k)-F^*$ & 5.16e-12  & 4.64e-14  &    & 6.94e-11  & 3.14e-12  &    & 5.30e-09 \\  
\hline
40 &  Time (sec)  &    &    &    & 3.97e-01  & 3.23e-01  &    & 2.18e-01 \\  
   & $F(x^k)-F^*$ &    &    &    & 1.84e-13  & 2.24e-14  &    & 5.18e-12 

       \\  
  \hline
    \end{tabular}
    
    \normalsize
    \label{tab:tallmatrix} 
\end{table}

\newpage
At each step of the algorithm, to compute our update to the $i^{th}$ block, we find $T_{\delta_{k+1}}^{\left( i\right)}\left( x^{k,i-1} \right) $ by calculating a $\delta_{k+1}$-approximate solution (that also satisfies the monotonic decrease condition \eqref{eq.mondec}) to the smaller-dimensional (likely far smaller) problem
\begin{equation}\label{num:subproblasso}
\arg \min_{y \in \mathbb{R}^{n_i}}  \frac{1}{2} \| A_i y - \tilde{b}^{k,i}\|_2^2 + \lambda \| y\|_1{,}
\end{equation}
where $\tilde{b}^{k,i} := b-Ax^{k,i-1}+A_ix^{k,i-1}_i$. To do so, we use the box-constrained gradient projection algorithm of \cite{Broughton11} to approximately solve \eqref{num:subproblasso}. We terminate the box-constrained gradient projection algorithm when the duality gap for \eqref{num:subproblasso} falls below $\delta_{k+1}$ and the monotonic decrease condition is satisfied.

As in the smooth case, we explored the effect of different error tolerance levels on I-CBPG and I-RBPG runtime performance and cycle counts by comparing their behavior under a variety of different error tolerance regimes. Specifically, we examined I-CBPG with a dynamic rule that sets $\delta_k = 1/k^2$ as well as both I-CBPG and I-RBPG with three different constant rules that fix $\delta_k$  at $10^{-4}$, $10^{-6}$, or $10^{-8}$ for all cycles, respectively. Plots of the difference $F\left(x^k\right) - F^*$ against both CPU time and the number of iterations $k$ are presented in Figure \ref{fig:widematrix} and Table \ref{tab:widematrix} for the first case, where $A$ is $N \times 2N$, and  in Figure \ref{fig:tallmatrix} and Table \ref{tab:tallmatrix} for the second case, where $A$ is $N \times 0.5N$. As in our experiments on smooth convex minimization, we performed a number of simulations of both types, but limit our discussion to a single instance of each for definiteness. Although there was minor quantitative variation across our experiments, the qualitative relationships exhibited by this instance were representative of every single experiment.

Our key qualitative results echo those for our previous smooth experiments. In all non-smooth experiments, I-CBPG with a dynamic error tolerance rule resoundingly beats both I-CBPG and I-RBPG with fixed error tolerances regardless of the tolerance level. Additionally, relatively more permissive error tolerance regimes reliably translate to shorter convergence times, but may require more cycles to converge. We wish to call attention to the computational efficiency of I-CBPG with a dynamic error tolerance rule relative to I-RBPG with a fixed error rule. For the wide $A$ setting, I-CBPG with a dynamic error tolerance rule $\delta_k = 1/k^2$ achieves convergence in $49\%$, $68\%$, and $84\%$ less time than I-RBPG with fixed error tolerances equal to $\delta_k = 10^{-4}$, $\delta_k = 10^{-6}$, and $\delta_k = 10^{-8}$, respectively.  For the tall $A$ setting, I-CBPG with a dynamic error tolerance rule $\delta_k = 1/k^2$ achieves convergence in $50\%$, $66\%$, and $61\%$ less time than I-RBPG with fixed error tolerances equal to $\delta_k = 10^{-4}$, $\delta_k = 10^{-6}$, and $\delta_k = 10^{-8}$, respectively.

In the wide $A$ setting, Figure \ref{fig:widematrix} and Table \ref{tab:widematrix} serve especially well to illuminate the disconnect between cycle counts and CPU time while showcasing the power and performance advantage of inexact computation in general, and dynamically decreasing error tolerances in particular. While cycle times display very little variation across different $\delta$ regimes from cycle $k = 1$ onward, a greater degree of error tolerance in early cycles translates to marked improvements in speed. In particular, for the wide $A$ setting one sees that more stringent error tolerances come at significantly higher CPU time costs for early iterates. Conversely, more permissive error tolerance rules for early iterates achieve the same progress in a fraction of the time.  These time savings carry through to convergence, as shown in Tables \ref{tab:widematrix} and \ref{tab:tallmatrix}.


\section{Chapter's Conclusion}

In this chapter, we introduced inexactly computed gradients and proximal maps into the Cyclic Block Proximal Gradient scheme resulting in our I-CBPG algorithm. Our convergence analysis covers both dynamically decreasing and fixed error tolerances. Our numerical experiments show that allowing for dynamically decreasing error tolerance in coordinate descent schemes can overturn the typical superiority of randomized schemes over cyclic schemes. This invites further research into randomized coordinate descent schemes with dynamic error tolerances. In the course of our analysis, we explored a unified framework for analyzing inexact computation via inexactly computed pre-conditioned proximal maps. This framework's tools enabled us to show how the inexact proximal map subsumes inexact computation of both gradients and proximal maps.

\section{Proof of Lemma \ref{lemma:recurrence-technical}}

\begin{proof} Fix $k\geq  2$. We begin by dividing both sides of \eqref{eq:recur-ineq} by $A_\ell A_{\ell+1}$,
\[
\frac{1}{\gamma}\frac{A_{\ell+1}}{A_\ell}\leq\frac{1}{A_{\ell+1}}-\frac{1}{A_\ell}+\frac{\Delta_{\ell+1}}{A_\ell A_{\ell+1}},
\]
then rearrange and use monotonicity of $\{A_\ell\}_{\ell\geq 0}$ to simplify this to
\[
\frac{1}{A_{\ell+1}}-\frac{1}{A_\ell} \geq\frac{1}{\gamma}\frac{A_{\ell+1}}{A_\ell}-\frac{\Delta_{\ell+1}}{A_\ell A_{\ell+1}}. 
\]
This rearrangement foreshadows the important roles of $A_{\ell+1}/A_\ell$ and $\Delta_{\ell+1}/(A_\ell A_{\ell+1})$. We consider two cases, divided according to the typical size of the ratio  $A_{\ell+1}/A_\ell$ for $\ell + 1 \le k$. In the second case, when the values of $A_\ell$ fall at what one may consider a relatively slow rate over this range, we consider three subcases based on the behavior of $\{\Delta_\ell\}_{\ell\geq1}$ and the typical values of $\frac{\Delta_{\ell+1}}{A_\ell A_{\ell+1}}$.
\begin{enumerate}[(i)]
\item For at least $\lfloor k/2\rfloor$ values of $0\leq \ell\leq k-1$, we have $A_{\ell+1}/A_\ell\leq 1/2$.
\item For at least $\lfloor k/2\rfloor$ values of $0\leq \ell\leq k-1$, we have $1/2  < A_{\ell+1}/A_\ell \leq1$.
In this case, we consider three subcases based on the values of $\Delta_{\ell+1}/(A_\ell A_{\ell+1})$ and the sequence $\{\Delta_\ell\}_{\ell\geq 1}$.
\end{enumerate}

\noindent\textit{Case 1: For at least $\lfloor k/2\rfloor$ values of $0\leq \ell\leq k-1$, $\frac{A_{\ell+1}}{A_\ell}\leq\frac{1}{2}$.}\\

\noindent This is the easy case. First, assume that $k$ is even. Then we have that $A_{\ell+1}\leq\frac{1}{2}A_\ell$ for at least $k/2$ values of $0\leq \ell\leq k-1$ so 
\[
A_k\leq  \left(\frac{1}{2}\right)^{k/2}A_0,
\]
since the $A_\ell$ terms are decreasing. If $k>2$ is odd, then $k-1$ is even, so by the same logic 
\[
A_k\leq  \left(\frac{1}{2}\right)^{(k-1)/2}A_0.
\]

\noindent\textit{Case 2: For at least $\lfloor k/2\rfloor$ values of $0\leq \ell\leq k-1$, $\frac{1}{2} <  \frac{A_{\ell+1}}{A_\ell} \leq 1$.}\\

\noindent We examine the following three subcases in turn:
\begin{enumerate}[(i)]
	\item $\Delta_{\ell} = \Delta\geq 0 $ for all $\ell$.
	\item The sequence $\{\Delta_\ell\}_{\ell\geq1}$ shrinks at the sublinear rate $\mathcal{O}(1/\ell^2)$ and for at least $\lfloor k/4\rfloor$ of the values for which $\frac{1}{2}< \frac{A_{\ell+1}}{A_\ell}\leq 1$ it also holds that $\frac{1}{4\gamma}>\frac{\Delta_{\ell+1}}{A_\ell A_{\ell+1}}$.
	\item The sequence $\{\Delta_\ell\}_{\ell\geq1}$ shrinks at the sublinear rate $\mathcal{O}(1/\ell^2)$ and for at least $\lfloor k/4\rfloor$ of the values for which $\frac{1}{2}<\frac{A_{\ell+1}}{A_\ell}\leq 1$ it also holds that $\frac{1}{4\gamma} \leq \frac{\Delta_{\ell+1}}{A_\ell A_{\ell+1}}$.
\end{enumerate}

\noindent\textit{Case 2, Subcase i: $\Delta_\ell = \Delta\geq 0$ for all $\ell$}.\\

Assume for now that $k$ is even. Define $u=\sqrt{\Delta\gamma}$, and let $\tilde{A}_{\ell}=A_{\ell}-u$.  Then the recurrence   \eqref{eq:recur-ineq} implies that $ \frac{1}{\gamma}A_{\ell+1}^2\leq A_\ell-A_{\ell+1}+\Delta_{\ell+1} $, which we may express as 
\[
\frac{1}{\gamma}(\tilde{A}_{\ell+1}+u)^2=\frac{1}{\gamma}A_{\ell+1}^2\leq  A_\ell-A_{\ell+1}+\Delta=\tilde{A}_\ell-\tilde{A}_{\ell+1}+\Delta.
\]
Expanding the square on the left, using the definition of $u$, and rearranging we see
\[
\frac{1}{\gamma}\tilde{A}_{\ell+1}^2\leq \tilde{A}_\ell-\left(1+\frac{2u}{\gamma}\right)\tilde{A}_{\ell+1}.
\]
If $\tilde{A}_{k} \leq 0$, the result is immediate, so suppose that $\tilde{A}_{k} > 0$, from which it follows that the earlier $\tilde{A}_{\ell}$ terms are also positive.  Then, for any $\ell$ with $0 \le \ell \le k-1$, we may divide the recurrence inequality by the product $\tilde{A}_{\ell+1}\tilde{A}_\ell$ to obtain
\[
 \frac{1}{\tilde{A}_{\ell+1}} - \left( 1 + \frac{2u}{\gamma}\right)\frac{1}{\tilde{A}_\ell} \ge \frac{1}{\gamma}\frac{\tilde{A}_{\ell+1}}{\tilde{A}_\ell}.
\]
Now, by hypothesis, for at least $k/2$ indices in the range $0 \leq \ell \leq k -1$
\[
 \frac{1}{\tilde{A}_{\ell+1}} - \frac{1}{\tilde{A}_\ell} \geq \frac{1}{\gamma}\frac{\tilde{A}_{\ell+1}}{\tilde{A}_\ell}  +  \frac{2u}{\gamma}  \frac{1}{\tilde{A}_\ell}  \geq \frac{1}{\gamma}\frac{1 }{2}  +  \frac{2u}{\gamma}  \frac{1}{\tilde{A}_0}.
\]
Iterating backward, one obtains
\[
 \frac{1}{\tilde{A}_k} \geq \frac{1}{\tilde{A}_{k}} - \frac{1}{\tilde{A}_0}  \geq \frac{k}{2} \left( \frac{1}{2\gamma}  +  \frac{2u}{\gamma}  \frac{1}{\tilde{A}_0} \right){,}
\]
 which gives $\tilde{A}_k \leq 4\gamma \tilde{A}_0/[k (\tilde{A}_0+4u)]$. The result follows from noting that $k-1$ is even if $k$ is odd, so we may replace $k$ with $k-1$ above to obtain a generic bound. \\



\noindent\textit{Case 2, Subcase ii: The sequence $\{\Delta_\ell\}_{\ell\geq1}$ shrinks at the sublinear rate $\mathcal{O}(1/\ell^2)$ and for at least $\lfloor k/4\rfloor$ of the values for which  $\frac{1}{2} <  \frac{A_{\ell+1}}{A_\ell} \leq 1$, it also holds that$\frac{\Delta_{\ell+1}}{A_\ell A_{\ell+1}}<  \frac{1}{4\gamma}$.}\\

Our reasoning follows the same idea as when $\Delta_\ell=\Delta\geq 0$ for all $\ell\geq 1$ (Case 2, Subcase i). First, assume that $k$ is divisible by $4$. We have for $k/4$ values of $0\leq\ell\leq k-1$ that 
\[
\frac{1}{A_{\ell+1}}-\frac{1}{A_\ell}\geq\frac{1}{2\gamma}-\frac{\Delta_{\ell+1}}{A_\ell A_{\ell+1}}\geq\frac{1}{2\gamma}-\frac{1}{4\gamma}=\frac{1}{4\gamma}.
\]
This inequality iterated backward, plus monotonicity and non-negativity of the sequence $\{A_\ell\}_{\ell\geq 0}$, implies that
\[
\frac{1}{A_k}\geq\frac{1}{A_k}-\frac{1}{A_0}\geq \frac{k}{4}\left[\frac{1}{4\gamma}\right]=\frac{k}{16\gamma }.
\]
Rearranging, we have that $A_k \leq 16\gamma / k$. If $k > 4$ is not divisible by $4$, then $k-1$, $k-2$, or $k-3$ must be, so in the worst case $A_k \leq 16\gamma/(k-3)$.

\noindent\textit{Case 2, Subcase iii: The sequence $\{\Delta_\ell\}_{\ell\geq1}$ shrinks at the sublinear rate $\mathcal{O}(1/\ell^2)$ and for at least $\lfloor k/4\rfloor$ of the values for which $\frac{1}{2} <  \frac{A_{\ell+1}}{A_\ell} \leq 1$, it also holds that $\frac{\Delta_{\ell+1}}{A_\ell A_{\ell+1}}\geq\frac{1}{4\gamma}$.}\\

First, suppose $k$ is divisible by $4$. Let $\ell^*$ denote the largest  $\ell\in\{0,\ldots,k-1\}$ for which $\frac{\Delta_{\ell+1}}{A_\ell A_{\ell+1}}\geq\frac{1}{4\gamma}$ holds. By hypothesis, $\ell^*$ must be at least as big as $\frac{k}{4} - 1$, and $\Delta^2_\ell \leq D/\ell^2$, so 
\[
\frac{1}{4\gamma}\cdot A_k^2\leq \frac{1}{4\gamma}\cdot A_k A_{k-1}\leq \frac{1}{4\gamma}\cdot A_{\ell^*+1} A_{\ell^*}\leq \Delta_{\ell^*+1}\leq\Delta_{k/4}\leq\frac{D}{(k/4)^2}.
\]
Dividing by $1/4\gamma$ and taking square roots, we have $A_{k}\leq\frac{8\sqrt{\gamma D}}{k}$. If $k>4$ is not divisible by $4$, then one of $k-1$, $k-2$, or $k-3$ are, so at worst $A_k\leq  \frac{8\sqrt{\gamma D}}{k-3}$.

Having completed our analysis, we may now combine the results from Case 1 with the appropriate Subcase(s) of Case 2 to establish the results in Lemma \ref{lemma:recurrence-technical}.

\end{proof}

%
%
%
%
%


\pagestyle{plain} 


\chapter{\MakeUppercase{The Randomized Block Coordinate Descent Method in the H{\"o}lder Smooth Setting}$^{*}$}\label{chapter:holder}

\renewcommand{\thefootnote}{\fnsymbol{footnote}} 
\footnotetext[1]{Reprinted with permission from \cite{maia2024randomizedblockcoordinatedescent}, Copyright 2025 by Springer Nature.}
\renewcommand{\thefootnote}{\arabic{footnote}} 


\section{Overview}

This chapter provides the first convergence analysis for the Randomized Block Coordinate Descent method for minimizing a function that is both H\"older smooth and block H\"older smooth. The analysis applies to objective functions that are non-convex, convex, and strongly convex. For  non-convex functions, it is showed that the expected gradient norm reduces at an $\mathcal{O}\left(k^{\frac{\gamma}{1+\gamma}}\right)$ rate, where $k$ is the iteration count and $\gamma$ is the H\"older exponent. For convex functions, it is showed that the expected suboptimality gap reduces at the rate $\mathcal{O}\left(k^{-\gamma}\right)$. In the strongly convex setting, we show this rate for the expected suboptimality gap improves to $\mathcal{O}\left(k^{-\frac{2\gamma}{1-\gamma}}\right)$ when $\gamma>1$ and to a linear rate when $\gamma=1$. Notably, these new convergence rates coincide with those furnished in the existing literature for the Lipschitz smooth setting.

\section{Introduction}

\noindent
In this chapter, we provide non-asymptotic convergence rates for the Randomized Block Coordinate Descent (RBCD) method when applied to the problem
\begin{equation}\label{eq:problem}
f^{*} := \min_{x\in\mathbb{R}^d} f(x) {,}
\end{equation}
where the objective function $f:\mathbb{R}^d\to\mathbb{R}$ is H\"older smooth, a generalization of the standard (Lipschitz) smoothness, and block H\"older smooth. Formally, the continuously differentiable function, $f$, is said to be \emph{H\"older smooth} when its gradient, $\nabla f$, is H\"older continuous, i.e. there exist $L>0$ and $\gamma\in (0,1]$ guaranteeing
\begin{equation}\label{eq:holder}
\|\nabla f(y)-\nabla f(x)\|\leq L\|y-x\|^\gamma\quad\text{for all }x,y\in\mathbb{R}^d.
\end{equation}

The popularity of block coordinate methods owes to their fitness for large-scale optimization problems emerging from applications in machine learning and statistics. Essentially, randomized block coordinate descent is a (random) block-wise adaptation of gradient descent. Instead of updating all coordinates simultaneously, the randomized block coordinate descent method updates a single, randomly selected coordinate block using only that block's partial gradient. The computational economy of these block gradient updates, relative to full gradient updates, are what make the randomized block coordinate  descent method especially attractive for large-scale problems.

{G}iven an initial point $x^0$, this cheap iterate update rule is somewhat more generally realized as
\begin{equation}\label{eq:BCD}
x^k=x^{k-1}-t_k\cdot P_{i_k}\nabla f(x^{k-1}), \quad k=1,2,\ldots {,}
\end{equation}
where $t_k>0$, $i_k$ is selected randomly from $\{1,\ldots,m\}$, and $P_1,\ldots,P_m\in\mathbb{R}^{d\times d}$ are orthogonal projection matrices onto orthogonal subspaces that sum to $\mathbb{R}^d$. {The} ``block coordinate" name originates from the archetypal choice for the orthogonal subspaces projected onto: spans of collections of coordinate vectors.

For coordinate descent methods, and indeed a preponderance of first-order methods, the intimate relationship between the selection of step-sizes and $\nabla f$'s regularity determines their convergence rates \cite{Yashtini16,Sh-Dvu-Gasn, NestUnivGradMethod, SmoothBanditOptGenHoldSpa, Nemi-Yur, Nemi-Ark-1983, MultScaZOOptSmFunc,GenHolSmooConvRatFollSpecRatAssGroBoun, Gutman22}. Both Bredies \cite{Bredies08} and {Yashtini} \cite{Yashtini16} study the interplay between step-size selection and convergence for gradient descent applied to \eqref{eq:problem} in the H\"older smooth regime. Bredies \cite{Bredies08} established a $O\left(1/k^\gamma\right)$ convergence rate of the proximal gradient method, a generalization of the gradient descent method, for convex composite minimization. On the other hand, {Yashtini} \cite{Yashtini16} established that, given an appropriate step-size selection, gradient descent converges at a  $O\left(1/k^{\frac{\gamma}{1+\gamma}}\right)$ for non-convex, H\"older smooth objective functions. 

We are unaware of any studies of the randomized block coordinate descent method that assume H\"older smoothness or its block-wise adaptation, block H\"older smoothness.  We say that the continuously differentiable function, $f$, is \emph{block H\"older smooth} if for each $i=1,\ldots,m$, there exists $L_i>0$ such that
\begin{equation}\label{eq:Holder-block}
\|\nabla f(x+P_i u)-\nabla f(x)\|\leq L_i\|P_i u\|^\gamma\text{ for all }u\in\mathbb{R}^d.
\end{equation}
The seminal articles \cite{Nesterov12,Richtarik14} studying the randomized block coordinate descent method all make the more restrictive assumption that the gradient is Lipschitz continuous. 
Recently, inspired by the work of both Bredies \cite{Bredies08} and {Yashtini} \cite{Yashtini16}, Gutman and Ho-Nguyen \cite{Gutman22} produced a convergence analysis for the cyclic block coordinate descent method assuming H\"older and block H\"older smoothness in both the convex and non-convex settings. Thus, the goal of this paper is to extend this analysis to the more popular randomized block coordinate descent method in the non-convex, convex, and even strongly convex settings. 

\subsection{Chapter's Organization}

We present an outline of the chapter that includes a high-level description of each of our main contributions. This chapter is structured into four primary sections:
\begin{itemize}
	\item \textit{Section} \ref{sec:algorithm}: In this section, we introduce our RBCD step-size selection for H\"older smooth objective functions as well as the attendant notation. We also introduce two key lemmata (Lemmata \ref{lemma:duality} and \ref{lemma:block_descent}) that support our analyses.
	\item \textit{Section} \ref{sec:nonconvex}: In this section, we present our convergence analysis for general, possibly non-convex objective functions satisfying H\"older and block H\"older smoothness conditions.  For these objectives, our proposed step-size ensures RBCD shrinks the expected gradient norm at a $\mathcal{O}\left(1/k^{\frac{\gamma}{1+\gamma}}\right)$ rate (Theorem \ref{thm.conv.rate.general}). 
	\item \textit{Section} \ref{sec:convex}: In this section, we present our convergence analysis under the further assumption that the objective function is convex. In this setting, RBCD with our step-size shrinks the expected suboptimality gap at a $O\left(1/k^{\gamma}\right)$ rate for non-strongly convex objective functions (Theorem \ref{thm.conv.rate.convex}). Notably, our rates for these objective functions coincide with those of \cite{Nesterov12} when $\gamma=1$, or equivalently, when the objective is $L$-smooth. 
	\item \textit{Section} \ref{sec:strong-convex}: In this section, we present our analysis under the further assumption that the objective function is strongly convex. This analysis depends upon the value of the H\"older exponent, $\gamma$. When $\gamma=1$, we show RBCD converges at a linear rate (Theorem \ref{thm.conv.rate.strong-convex}). When $\gamma \in (0,1)$, we obtain a $\mathcal{O}\left( 1/k^{\frac{2\gamma}{1-\gamma}} \right)$ rate of convergence (Theorem \ref{thm.conv.rate.strong-convex}). Moreover, we show that our sublinear rates converge to our linear rates as $\gamma\to 1$ (Corollary \ref{cor:str-cvx.interpolation}). As for convex objectives, our rates for strongly convex objectives coincide with those of \cite{Nesterov12} when the objective is $L$-smooth. 
\end{itemize}
\section{Notation and Step-size Selection for RBCD Under H\"older Smoothness}
\label{sec:algorithm}

This short section introduces the notation necessary for all of this article's developments, and  the H\"older smoothness-based step-size selection for the RBCD method. It also exhibits two lemmata, Lemmata \ref{lemma:duality} and \ref{lemma:block_descent}, that are used throughout the paper to aid the convergence analysis of the proposed method.

Our step-size selection is an adaptation of that used for the cyclic block setting from \cite{Gutman22} to the immensely more popular randomized block setting. Thus, our notation is a synthesis of that article's notation and the notation of \cite{Nesterov12}, one of the canonical works on randomized coordinate descent. We let $\tilde{L}:=\{L_1,\ldots,L_m\}$ denote the set of the block H\"older smoothness constants. For $\alpha\in\mathbb{R}$, we define the new constant, $S_{\alpha}(f)$, as
\[
S_{\alpha}(f) = \sum_{i=1}^{m}L_i^{\alpha}.
\]
When $f$ is clear from context, we will simply write $S_\alpha$. For the sake of concision, we set $\nabla_i f(x):=P_i \nabla f(x)$ for all $x\in\mathbb{R}^d$ and $1\leq i\leq m$. We adopt the notation, $\nu:=\frac{1+\gamma}{\gamma}>1$, because the quantity $\frac{1+\gamma}{\gamma}$  frequently appears in our analysis.

Much of our analysis is framed in terms of $\tilde{L}$-weighted $q$-norms on $\mathbb{R}^d$, $\|\cdot\|_{\alpha,q}$. Given $\alpha\in\mathbb{R}$ and $q\geq 1$, we let 
  
When $\alpha=0$, $\|\cdot\|_{\alpha,q}$ reduces to the standard $q$-norm, which we write as $\|\cdot\|_q$. For simplicity, we let $\|\cdot\|:=\|\cdot\|_2$.
We bare three important notes about these weighted norms. First, $\|\cdot\|_{\alpha,q}$ generalizes the norm
\[
x\mapsto\left[\sum_{i=1}^{m}L_{i}^\alpha\|P_i x\|^2  \right]^{1/2},
\]
which plays a starring role throughout Nesterov's classical analysis of randomized coordinate descent methods from \cite{Nesterov12} in the block Lipschitz smooth setting. The flexibility provided by changing the exponents $2$ and $1/2$ to $q$ and $1/q$, respectively, is critical to capturing our more general H{\"o}lderian convergence rates.  Additionally, the parameter $\alpha$ permits us to simultaneously achieve RBCD's convergence rates for two different, common random block selection schemes:
\begin{itemize}
\item[i)] $\alpha=0$ corresponds to selecting the blocks uniformly at random;
\item[ii)] $\alpha=1$ corresponds to selecting the $i$-th block with probability $L_i/\sum_{i=1}^m L_i$.
\end{itemize}
Finally, these weighted norms possess natural duality relationships and equivalences to the Euclidean norm, which we liberally use throughout our analysis and summarize in the below lemma.
\begin{lemma}[$(\alpha,q)$-Norm Duality and Equivalences]\label{lemma:duality}
Let $\alpha\in\mathbb{R}$, $p\in[1,\infty]$, and $q$ be the H{\"o}lder conjugate of $p$, i.e. $q:=\frac{p}{p-1}$. The following hold for $\|\cdot\|_{\alpha,p}$:
\begin{enumerate}
	\item The Cauchy-Schwarz inequality
\begin{equation}\label{eq:cauchy-schwarz}
|\inner{x}{y}|\leq \|x\|_{\alpha,q}\|y\|_{-\alpha\frac{p}{q},q}
\end{equation}
holds for all $x,y\in\mathbb{R}^d$. Equality is obtained if and only if $x=0$ or 
\[
y=c\cdot\sum_{i=1}^m L_i^\alpha \|P_i x\|^{p-2}P_i x
\] 
for some $c\in\mathbb{R}$. Consequently, $\|\cdot\|_{-\alpha\frac{p}{q},q}$ is the dual norm of $\|\cdot\|_{\alpha,q}$.
	\item If $p\geq 2$ and $\alpha,\beta\in\mathbb{R}$ then the norms $\|\cdot\|_{\alpha,p}$ and $\|\cdot\|_{\beta,2}$ satisfy 
\[
\left( \max_{1\leq i\leq m}L_i^{\frac{\alpha}{p}-\frac{\beta}{2}}\right)\cdot  \|x\|_{\beta,2} \geq \|x\|_{\alpha,p} \geq \left( m^{\frac{1}{p}-\frac{1}{2}}\cdot \min_{1\leq i\leq m}L_i^{\frac{\alpha}{p}-\frac{\beta}{2}}\right) \cdot \|x\|_{\beta,2} 
\]
for all $x\in\mathbb{R}^d$.
\end{enumerate} 
\end{lemma}

\noindent We defer the proof of this lemma to the appendix (Appendix \ref{app:duality}) to maintain the focus of our exposition. 

With all of the article's requisite notation in hand, we may introduce our main algorithm (Algorithm \ref{algoHolder.main}), and describe an associated descent lemma (Lemma \ref{lemma:block_descent}). We note that our step-size, $-\|\nabla_{i}f(x^k)\|^{\nu-2}/L_i^{\nu-1}$, coincides with that proposed in \cite{Nesterov12} when $\gamma=1$. Thus, we may view it as a generalization that accounts for the use of block H\"older smoothness in the place of standard block smoothness.

\begin{algorithm}[H]
\caption{Randomized Block Coordinate Descent Method (RBCD)}\label{algoHolder.main}
\begin{algorithmic}[1]
\Require $x^0 \in \dom(f)$, $\alpha\in[0,1]$
	\For{$k=0,1,2,\ldots$}
	\State Choose
		\[
		i_k\sim \left(p_1,\ldots,p_m\right):=\left(\frac{L_{1}^{\alpha}}{\sum_{j=1}^{m}L_{j}^{\alpha}},\ldots,\frac{L_{m}^{\alpha}}{\sum_{j=1}^{m}L_{j}^{\alpha}}\right);
	\]
	\State Update block $i_k$ of $x^k$ according to
		\begin{equation}\label{eq:main_update}
			x^{k+1}:=x^{k}-\frac{\|\nabla_{i}f(x^k)\|^{\nu-2}}{L_i^{\nu-1}}\cdot \nabla_i f(x^k).
		\end{equation}
	\EndFor
\end{algorithmic}
\end{algorithm}

A special case of the main descent lemma of \cite{Bredies08}, derived in \cite{Gutman22}, plays the same role in our analysis that it played for the cyclic block analysis in \cite{Gutman22}. We directly quote this special case from \cite{Gutman22} below.

\begin{lemma}
\label{lemma:block_descent}[Block H{\"o}lder Descent Lemma, \cite{Gutman22}, Lemma 1] Let $f:\mathbb{R}^{d}\rightarrow R$ be a function that satisfies the block H\"older smoothness condition. For any $i$, $1\leq i \leq m$,
\begin{equation}\label{eq:upper-model}
    f(x+U_iy) \leq f(x) + \inner{\nabla_i f(x)}{U_iy} + \frac{L_i}{1+\gamma}\|U_iy\|^{1+\gamma}_2.
\end{equation}
Moreover, if $x^+$ is the minimizer of the right-hand side of \eqref{eq:upper-model}, i.e.
\[
x^+=x-\frac{\|\nabla_{i}f(x)\|^{\nu-2}}{L_i^{\nu-1}}\cdot \nabla_i f(x),
\]
then
\[
f(x)-f(x^+)\geq\frac{1}{\nu L_i^{\nu-1}}\|\nabla_i f(x)\|^\nu.
\]
\end{lemma}

\section{Convergence Analysis: General Objectives}
\label{sec:nonconvex}

In this section, we layout our convergence rate analysis for non-convex objectives satisfying H\"older smoothness \eqref{eq:holder} and block H\"older smoothness \eqref{eq:Holder-block}. We will present the main convergence theorem (Theorem \ref{thm.conv.rate.general}) after we elaborate our key Sufficient Decrease Lemma (Lemma \ref{lemma:sufficient-decrease}). This lemma facilitates all of our convergence analyses.

\begin{lemma}\label{lemma:sufficient-decrease}
\emph{(Sufficient Decrease)}  Let $\{x_n\}_{n=0}^{\infty}$ be the sequence generated by RBCD (Algorithm \ref{algoHolder.main}). If $f$ satisfies our H{\"o}lder smoothness \eqref{eq:holder} and block H\"older smoothness \eqref{eq:Holder-block} assumptions, then
\begin{equation}\label{eq:sufficient-decrease}
\frac{1}{\nu S_\alpha(f)}\|\nabla f(x_k)\|_{\alpha+1-\nu,\nu}^\nu\leq f(x_k)-\mathbb{E}\left[f(x^{k+1})\bigg| x^k\right]
\end{equation}
holds for all $k\geq 0$.
\end{lemma}

\begin{proof}
Expanding the expectation-defining sum, and applying the block descent lemma (Lemma \ref{lemma:block_descent}), we compute
\begin{multline*}
f(x_k)-\mathbb{E}\left[f(x^{k+1})\bigg| x^k\right]=\mathbb{E}\left[f(x_k)-f(x^{k+1})\bigg| x^k\right]\\
=\sum_{i=1}^{m}\left(\frac{L_i^{\alpha}}{\sum_{j=1}^{m}L_{j}^{\alpha}}\right)\cdot\left[ f(x_k)-f\left(x^{k}-\frac{\|\nabla_{i}f(x^k)\|^{\nu-2}}{L_i^{\nu-1}}\cdot \nabla_i f(x^k)\right)\right]\\
\stackrel{\text{Lemma }\ref{lemma:block_descent}}\geq \frac{1}{\nu S_\alpha}\sum_{i=1}^{m}L_i^{\alpha+1-\nu}\|\nabla_{i}f(x^k)\|^\nu=\frac{1}{\nu S_\alpha}\|\nabla f(x_k)\|_{\alpha+1-\nu,\nu}^\nu.
\end{multline*}
Rearranging the inequality and taking total expectations completes the proof.
\end{proof}

Next, we present the centerpiece of this section, our main convergence theorem for non-convex objective functions.

\begin{theorem}[RBCD Convergence: General Objective Functions]\label{thm.conv.rate.general}
Let $\{x_n\}_{n=0}^{\infty}$ be the sequence generated by RBCD (Algorithm \ref{algoHolder.main}). If $f$ satisfies our H{\"o}lder smoothness \eqref{eq:holder} and block H\"older smoothness \eqref{eq:Holder-block} assumptions, then
\[
\min_{0\leq j \leq k}\mathbb{E}\left[\|\nabla f(x^j)\|_{1+\alpha-\nu,\nu}\right] \leq  \left(\nu S_\alpha(f)\right)^{\frac{1}{\nu}} \cdot\left(\frac{f(x^0)-f^*}{k+1}\right)^{\frac{1}{\nu}}= \mathcal{O}\left( k^{-\frac{1}{\nu}}\right)
\]
holds for all $k\geq 0$. Consequently, we have the convergence rate measured in the norm $\|\cdot\|_{\beta,2}$,
\[
\min_{0\leq j \leq k}\mathbb{E}\left[\|\nabla f(x^j)\|_{\beta,2}\right] \leq\left(\frac{\displaystyle \max_{1\leq i\leq m} L_i^{\frac{\beta}{2}-\frac{1+\alpha-\nu}{\nu}}}{m^{\frac{\nu-2}{2\nu}}}\right)\cdot\left(\nu S_\alpha(f)\right)^{\frac{1}{\nu}} \cdot\left(\frac{f(x^0)-f^*}{k+1}\right)^{\frac{1}{\nu}}= \mathcal{O}\left( k^{-\frac{1}{\nu}}\right),
\]
holds for all $k\geq 0$.
\end{theorem}

\begin{proof}
For each $k\geq 0$, observe that
\begin{align}
\min_{0\leq j\leq k}\mathbb{E}\left[\|\nabla f(x^j)\|_{\alpha+1-\nu,\nu}\right]^\nu
&\leq \frac{1}{(k+1)}\cdot\sum_{j=0}^k\mathbb{E}\left[\|\nabla f(x^j)\|_{\alpha+1-\nu,\nu}\right]^\nu \nonumber \\
&\leq \frac{1}{(k+1)}\cdot\sum_{j=0}^k\mathbb{E}\left[\|\nabla f(x^j)\|_{\alpha+1-\nu,\nu}^\nu\right] \label{thm:rbcd.conv.proof1}\\
&\leq \nu S_\alpha\cdot \frac{1}{(k+1)}\cdot\sum_{j=0}^k  \left(\mathbb{E}\left[f(x^j)\right]-\mathbb{E}\left[f(x^{j+1})\right]\right) \label{thm:rbcd.conv.proof2}\\
&=\nu S_\alpha\cdot\frac{f(x^0)-\mathbb{E}[f(x^{k+1})]}{k+1}\leq\nu S_\alpha\cdot\frac{f(x^0)-f^*}{k+1}{,} \nonumber
\end{align} 
{where we apply Jensen's inequality to the expectation operator for the convex function $x\mapsto x^\nu$ in \eqref{thm:rbcd.conv.proof1}, and Lemma \ref{lemma:sufficient-decrease} in \eqref{thm:rbcd.conv.proof2}. Taking $\nu$-th roots of both sides of the resultant inequality above, produces our first result.}

The result in terms of the $\|\cdot\|_{\beta,2}$ follows immediately from Lemma \ref{lemma:duality}.
\end{proof}

\section{Convergence Analysis: Convex Objectives}
\label{sec:convex}

In this section, we forward our convergence analysis of RBCD (Theorem \ref{thm.conv.rate.convex}) for convex objective functions that are both H\"older \eqref{eq:holder} and block H\"older \eqref{eq:Holder-block} smooth. First, we  present a Technical Recurrence Lemma (Lemma \ref{lemma_poly}) that helps produce our convergence rates in this section, and a subset of the convergence rates for strongly convex objective functions in the sequel. Next, we exhibit a techical lemma (Lemma \ref{lemma:nesterov-holder-cvx-str}) that permits us to express our rates in terms of the diameter of the initial sublevel set. Finally, the section concludes with our main convergence theorem (Theorem \ref{thm.conv.rate.convex}) and a comparison of these rates to those furnished for smooth and convex functions in \cite{Nesterov12}. 

As promised, we begin this section with a Technical Recurrence Lemma that supports the derivation of our convergence rates.

\begin{lemma}
\label{lemma_poly}(Technical Recurrence, \cite[Chapter~2, Lemma~6]{Polyak_book})\label{lemma:recurrence}
If $\{A_k\}_{k\geq 0}$ is a non-negative sequence of real numbers satisfying the recurrence
\[
A_{k+1}\leq A_k - \theta A_{k}^r 
\]
for some $\theta \geq 0$ and $r> 1$, then
\[
A_k \leq \frac{A_0}{(1+(r-1)\theta A_0^{r-1}k)^{\frac{1}{r-1}}} {.}
\]
\end{lemma}

The following lemma permits us to express our convergence rates here and in the sequel section in terms of the initial sublevel set's diameter.

\begin{lemma}\label{lemma:nesterov-holder-cvx-str}
Under the Block H\"older Smoothness assumption \eqref{eq:Holder-block} and coercivity of $f$, $f$ satisfies
\begin{align*}
f(x)-f^* \leq\left(\frac{\nu S_\alpha(f)}{2}\right)^{\frac{1}{\nu-1}}\cdot\left(\frac{\nu-1}{\nu}\right)\cdot R(x)_{(1+\alpha-\nu)(1-\nu),\frac{\nu}{\nu-1}}^{\frac{\nu}{\nu-1}}
\end{align*}
for all $x\in\mathbb{R}^d$, where $R_{\beta,q}(x):=\max\left\{\|y-x^*\|_{\beta,q}:f(x^*)=f^*,f(y)\leq f(x)\right\}<\infty$.
\end{lemma}

\noindent We defer the proof of this lemma to the appendix (Appendix \ref{app:level-set}).

Finally, equipped with these tools, we present and prove the theorem that establishes RBCD's convergence rate for convex functions. Afterward, we explain its relationship to its analogue for smooth and convex functions in \cite{Nesterov12}.

\begin{theorem}[RBCD Convergence: Convex Objective Functions]\label{thm.conv.rate.convex}
Let $\{x_n\}_{n=1}^{\infty}$ be the sequence generated by RBCD (Algorithm \ref{algoHolder.main}). If $f$ is a convex and coercive function that satisfies our H{\"o}lder smoothness \eqref{eq:holder} and block H\"older smoothness \eqref{eq:Holder-block} assumptions, then
\[
\mathbb{E}[f(x^{k})]- f^*\leq\frac{\left(\nu S_\alpha(f) R_{(1+\alpha-\nu)(1-\nu),\frac{\nu}{\nu-1}}(x^0)^\nu\right)^{\frac{1}{\nu-1}}(\nu-1)}{\left[2\nu^{\nu-1}+(\nu-1)^\nu k\right]^{\frac{1}{\nu-1}}}=\mathcal{O}\left(k^{-\frac{1}{\nu-1}}\right) {,}
\]
where $R_{\beta,q}(x^0):=\max_{y}\left\{\|y-x^*\|_{\beta,q}:f(x^*)=f^*,f(y)\leq f(x^0)\right\}<\infty${.}
\end{theorem}

\begin{proof}
The bulk of this proof centers on an application of the Technical Recurrence Lemma (Lemma \ref{lemma:recurrence}). In the context of that lemma, we let $A_i = \mathbb{E}[f(x^{i})]- f^*$ for each $i\geq 0$. By definition, $A_i\geq 0$ for each $i\geq 0$. To simplify notation, we let $R:=R_{(1+\alpha-\nu)(1-\nu),\frac{\nu}{\nu-1}}(x^0)$.

With this notation, we may restate the sufficient decrease inequality \eqref{eq:sufficient-decrease} of Lemma \ref{lemma:sufficient-decrease} as
\begin{equation*}
\mathbb{E}\left[\|\nabla f(x_k)\|_{\alpha+1-\nu,\nu}^\nu\right]\leq \nu S_\alpha\cdot\left(A_k-A_{k+1}\right),
\end{equation*}
or, equivalently,
\begin{equation}\label{eq:cvx.recur}
A_{k+1}\leq A_k-\frac{1}{\nu S_\alpha}\mathbb{E}\left[\|\nabla f(x_k)\|_{\alpha+1-\nu,\nu}^\nu\right].
\end{equation}
Thus, to apply the Technical Recurrence Lemma (Lemma \ref{lemma:recurrence}) we need only bound the expectation on the right below by $A_k^\nu$. By the Cauchy-Schwarz inequality (Lemma \ref{lemma:duality}) for $\|\cdot\|_{1+\alpha-\nu,\nu}$ and its dual $\|\cdot\|_{(1+\alpha-\nu)(1-\nu),\frac{\nu}{\nu-1}}$, we achieve for any optimum $x^*$, that
\begin{multline*}
f(x^k)-f^*\leq \inner{x^k-x^*}{\nabla f(x^k)}\\
\leq\|x^k-x^*\|_{(1+\alpha-\nu)(1-\nu),\frac{\nu}{\nu-1}}\|\nabla f(x^k)\|_{1+\alpha-\nu,\nu}\leq R\|\nabla f(x^k)\|_{1+\alpha-\nu,\nu}.
\end{multline*}
Raising each side of the above inequality to the power $\nu$, taking expectations, and applying  Jensen's inequality to the convex function $x\mapsto x^{\nu/2}$, we conclude
\begin{equation*}
A_k^\nu=\left(\mathbb{E}[f(x^k)]-f^*\right)^\nu\leq \mathbb{E}\left[\left(f(x^k)-f^*\right)^\nu\right]\\
\leq R^\nu\mathbb{E}\left[\|\nabla f(x^k)\|_{1+\alpha-\nu,\nu}^\nu\right].
\end{equation*}
Stringing together our work above, equation \eqref{eq:cvx.recur} yields the recurrence
\[
A_{k+1}\leq A_k-\frac{1}{\nu S_\alpha R^\nu} A_k^\nu
\]
for each $k\geq 0$.
We are now permitted to apply the Technical Recurrence Lemma (Lemma \ref{lemma:recurrence}) with  $r=\nu$ and $\theta=\frac{1}{\nu S_\alpha R^\nu}$ to produce
\begin{equation*}
\mathbb{E}[f(x^{k})]- f^*\leq\frac{f(x^{0})- f^*}{\left[1+(\nu-1)\cdot\theta\cdot\left(f(x^{0})- f^*\right)^{\nu-1}\cdot k\right]^{\frac{1}{\nu-1}}}.
\end{equation*}
We dedicate the remainder of this proof to simplifying this convergence bound. By factoring $f(x^{0})- f^*$ out of both the numerator and denominator, we may equivalently write the right-hand side of this bound as
 \begin{equation*}
\frac{1}{\left[\frac{1}{\left(f(x^{0})- f^*\right)^{\nu-1}}+(\nu-1)\cdot\theta\cdot k\right]^{\frac{1}{\nu-1}}}.
\end{equation*}

By considering $x=x^0$ in Lemma~\ref{lemma:nesterov-holder-cvx-str}, raising both sides to the power $\nu-1$ and applying the norm equivalence inequality from Lemma \ref{lemma:duality}, we further see
\begin{align*}
\left[f(x^0)-f(x^*)\right]^{\nu-1}&\leq \left(\frac{\nu S_\alpha}{2}\right)\cdot\left(\frac{\nu-1}{\nu}\right)^{\nu-1}\cdot R^\nu
\end{align*}
so the right-hand side of our bound simplifies to
\[
\left[\frac{1}{\left(\frac{\nu S_\alpha}{2}\right)\cdot\left(\frac{\nu-1}{\nu}\right)^{\nu-1}\cdot R^\nu}+(\nu-1)\cdot\left(\frac{1}{\nu S_\alpha R^\nu}\right)\cdot k\right]^{-\frac{1}{\nu-1}}=\frac{\left(\nu S_\alpha R^\nu\right)^{\frac{1}{\nu-1}}(\nu-1)}{\left(2\nu^{\nu-1}+(\nu-1)^\nu k\right)^{\frac{1}{\nu-1}}},
\]
which concludes the proof.
\end{proof}
Notably, our rate matches that provided by {Nesterov \cite{Nesterov12}} in the standard block smooth setting, i.e. when $\nu=2$ we recover the convergence rate,
\[
\mathbb{E}[A_{k+1}] \leq \frac{2}{k+4}S_\alpha (f) R^2(x_1) = \mathcal{O}(k^{-1}),
\]
from \cite{Nesterov12}.

\section{Convergence Analysis: Strongly Convex Objectives}
\label{sec:strong-convex}

In this final section, we conclude the paper with a convergence analysis of RBCD (Algorithm \ref{algoHolder.main}) for strongly convex objective functions that are both H\"older \eqref{eq:holder} and block H\"older \eqref{eq:Holder-block} smooth. We say that $f:\mathbb{R}^d\to \mathbb{R}$ is $\sigma$-strongly convex with respect to the norm $\|\cdot\|_{1-\alpha,2}$, where $\sigma>0$, if
\begin{equation}\label{eq:str-cvx}
f(y) \geq f(x) +\inner{\nabla f(x)}{y-x} + \frac{1}{2}\sigma \|x-y\|_{1-\alpha,2}^2
\end{equation}
for all $x,y\in\mathbb{R}^d$. The section begins with our main theorem (Theorem \ref{thm.conv.rate.strong-convex}), which provides rates in both the $L$-smooth and H\"older smooth settings. Next, we compare these rates with those in the previous section and \cite{Nesterov12}. Finally, we show that the smooth setting's linear rate is achieved in the limit as $\nu\to 2$, or equivalently, $\gamma\to 1$ (Corollary~\ref{cor:str-cvx.interpolation}). 

Without further ado, we present our main convergence theorem for strongly convex objective functions.

\begin{theorem}[RBCD Convergence: Strongly Convex Objective Functions]\label{thm.conv.rate.strong-convex}
Let $\{x_n\}_{n=1}^{\infty}$ be the sequence generated by RBCD (Algorithm \ref{algoHolder.main}). Suppose that $f:\mathbb{R}^d\to\mathbb{R}$ is $\sigma$-strongly convex and satisfies both the H{\"o}lder and block H{\"o}lder  smoothness assumptions \eqref{eq:holder} and \eqref{eq:Holder-block}. The following hold:
\begin{enumerate}
\item (Linear Rate - Smooth Setting) If $\nu=2$, i.e. $f$ is smooth, then
\begin{equation}\label{eq:linear-rate}
\mathbb{E}[f(x_k)] -f^* \leq \left(1-\frac{\sigma}{S_\alpha(f)} \right)^k \cdot \frac{S_\alpha(f)^{\frac{1}{\nu-1}}(\nu-1)R(x^0)^{\frac{\nu}{\nu-1}}}{\nu^{\frac{\nu-2}{\nu}}2^{\frac{1}{\nu-1}}}=\mathcal{O}\left(\exp\left(-\frac{\sigma}{S_\alpha(f)}k\right)\right).
\end{equation}
\item (Sublinear Rate - H{\"o}lder Smooth Setting) If $\nu>2$, i.e. $f$ is H{\"o}lder smooth but not smooth, then
\begin{equation*}
 \mathbb{E}[f(x^{k})]- f^* \leq \frac{C_0}{\left(C_1+C_2 k\right)^{\frac{2}{\nu-2}}}=\mathcal{O}\left(k^{-\frac{2}{\nu-2}}\right),
\end{equation*}
where
\begin{equation*}
\begin{gathered}
C_0=(2\nu S_{\alpha}(f))^{\frac{2}{\nu-2}}m^{\frac{1}{\nu}}(\nu-1)R(x^0)^{\frac{\nu}{\nu-1}},\quad C_1=2^{\frac{\nu-2}{2(\nu-1)}}m^{\frac{\nu-2}{2\nu}} S_{\alpha}(f)^{\frac{\nu-1}{\nu}}\nu^{\frac{\nu^2-2\nu+4}{2\nu}}\\
C_2=R(x^0)^{\frac{\nu(\nu-2)}{2(\nu-1)}}(\nu-1)^{\frac{\nu-2}{2}}(\nu-2)(2\sigma)^{\frac{\nu}{2}}\displaystyle \min_{1\leq i\leq m} L_i^{\frac{(\alpha+1)(2-\nu)}{2\nu}}
\end{gathered}
\end{equation*}
\end{enumerate}
\end{theorem}

\begin{proof}
Let $R:=R_{(1+\alpha-\nu)(1-\nu),\frac{\nu}{\nu-1}}(x^0)$ to simplify notation. Both parts of the theorem speedily follow from the recurrence
\begin{equation}\label{eq:recur-str-cvx}
A_{k+1}\leq A_{k} - A_{k}^{\frac{\nu}{2}} \cdot \left( \frac{(2\sigma)^{\frac{\nu}{2}}\displaystyle \min_{1\leq i\leq m} L_i^{\frac{(\alpha+1)(2-\nu)}{2\nu}}}{\nu S_{\alpha}m^{\frac{1}{2}-\frac{1}{\nu}}}\right){,}
\end{equation}
where $A_i = \mathbb{E}[f(x^{i})]- f^*$ for each $i\geq 0$. After establishing \eqref{eq:recur-str-cvx}, we will separately show how each of the Theorem's two parts result from it. 

As in the proof of convergence for non-strongly convex functions, the sufficient decrease inequality \eqref{eq:sufficient-decrease} of Lemma \ref{lemma:sufficient-decrease} implies \eqref{eq:cvx.recur}, which we recall is
\[
A_{k+1}\leq A_k-\frac{1}{\nu S_\alpha}\mathbb{E}\left[\|\nabla f(x_k)\|_{\alpha+1-\nu,\nu}^\nu\right].
\]
Glancing at \eqref{eq:recur-str-cvx} and this latest inequality, it becomes immediately clear that we ought to bound $\mathbb{E}\left[\|\nabla f(x_k)\|_{\alpha+1-\nu,\nu}^\nu\right]$ below by $A_{k}^{\frac{\nu}{2}}=\mathbb{E}\left[ f(x_k)-f^*\right]^{\nu/2}$, appropriately scaled. To this end, strong convexity now makes it's main appearance. Using the standard argument of fixing $x\in\mathbb{R}^d$ in $\sigma$-strong convexity's defining inequality \eqref{eq:str-cvx} and minimizing it over $y\in\mathbb{R}^d$, we achieve the Polyak-\L ojasiewicz (PL) inequality
\[
\frac{1}{2\sigma} \left(\|\nabla f(x)\|_{1-\alpha,2}^*\right)^2 \geq f(x)-f^* .
\]
Setting $x=x_k$, raising both sides to the power $\nu/2$, and then taking expectations, we see 
\begin{equation*}
\frac{1}{(2\sigma)^{\frac{\nu}{2}}} \mathbb{E}\left[\left(\|\nabla f(x_k)\|_{1-\alpha,2}^*\right)^\nu \right] \geq \mathbb{E}\left[(f(x_k)-f^*)^{\nu/2}\right],
\end{equation*}
which, by Jensen's inequality applied to the convex function $x\mapsto x^{\nu/2}$, produces
\begin{equation}\label{eq:strcv.Jens}
\frac{1}{(2\sigma)^{\frac{\nu}{2}}} \mathbb{E}\left[\left(\|\nabla f(x_k)\|_{1-\alpha,2}^*\right)^\nu \right] \geq \left(\mathbb{E}\left[ f(x_k)-f^*\right] \right)^{\nu/2} = A_{k}^{\frac{\nu}{2}}
\end{equation}
The main bound \eqref{eq:recur-str-cvx} is secured by twice applying the $(\alpha,q)$-Norm Duality Equivalence Lemma (Lemma \ref{lemma:duality}) to connect the recurrence inequality \eqref{eq:cvx.recur} and the PL-derived bound \eqref{eq:strcv.Jens},
\begin{align*}
A_k- A_{k+1}&\stackrel{\eqref{eq:cvx.recur}}\geq \frac{1}{\nu S_\alpha}\mathbb{E}\left[\|\nabla f(x_k)  \|_{\alpha+1-\nu,\nu}^\nu\right] \\
&\stackrel{\text{Lemma }\ref{lemma:duality}}\geq \frac{1}{\nu S_\alpha}\cdot \left(\frac{\displaystyle \min_{1\leq i\leq m}L_i^{\frac{\alpha+1-\nu}{\nu}-\frac{\alpha-1}{2}}}{m^{\frac{1}{2}-\frac{1}{\nu}}} \right)\cdot \mathbb{E}\left[\|\nabla f(x_k)\|_{\alpha-1,2}^\nu\right]\\
&\stackrel{\text{Lemma }\ref{lemma:duality}}=\frac{1}{\nu S_\alpha}\cdot \left(\frac{\displaystyle \min_{1\leq i\leq m}L_i^{\frac{(\alpha+1)(2-\nu)}{2\nu}}}{m^{\frac{1}{2}-\frac{1}{\nu}}} \right)\cdot \mathbb{E}\left[\left(\|\nabla f(x_k)\|_{1-\alpha,2}^*\right)^\nu\right]\\
&\stackrel{\eqref{eq:strcv.Jens}}\geq \frac{1}{\nu S_\alpha}\cdot \left(\frac{\displaystyle \min_{1\leq i\leq m}L_i^{\frac{(\alpha+1)(2-\nu)}{2\nu}}}{m^{\frac{1}{2}-\frac{1}{\nu}}} \right)\cdot \left[(2\sigma)^{\frac{\nu}{2}}A_{k}^{\frac{\nu}{2}}\right].
\end{align*}
Now, we are prepared to prove the theorem's two constituent parts.\\

\noindent 1. If $\nu=2$, then the main recurrence inequality \eqref{eq:recur-str-cvx} becomes
\[
A_{k+1}\leq A_{k} - A_k^{} \cdot \left( \frac{\sigma}{S_{\alpha}}\right) = A_k\left(1-\frac{\sigma}{S_\alpha} \right),
\]
which by backward induction is equivalent to our desired bound,
\begin{align}\label{eq:str-cvx.nu=2}
\mathbb{E}[f(x_k)] -f^* = A_k\leq \left(1-\frac{\sigma}{S_\alpha} \right)^k\cdot  A_0&=\left(1-\frac{\sigma}{S_\alpha} \right)^k \cdot \left[f(x_0) -f^* \right]\\
&\leq \left(1-\frac{\sigma}{S_\alpha} \right)^k \cdot \frac{S_\alpha^{\frac{1}{\nu-1}}(\nu-1)R^{\frac{\nu}{\nu-1}}}{\nu^{\frac{\nu-2}{\nu}}2^{\frac{1}{\nu-1}}},\nonumber
\end{align}
where we have applied Lemma \ref{lemma:nesterov-holder-cvx-str} in the first line.\\

\noindent 2. The $\nu>2$ result requires a verification that is as straightforward as, but more tedious than, that of 1. Applying the Technical Recurrence Lemma (Lemma \ref{lemma:recurrence}) with  $r=\nu/2$, and\\ $\theta=\frac{(2\sigma)^{\frac{\nu}{2}}}{\nu S_\alpha m^{\frac{\nu-2}{2\nu}}}\cdot \displaystyle \min_{1\leq i\leq m} L_i^{\frac{(\alpha+1)(2-\nu)}{2\nu}}$ we see
\begin{equation}\label{eq:rate-for-interpolation}
\mathbb{E}[f(x^{k})]- f^*\leq \frac{f(x^{0})- f^*}{\left[1+\left(\frac{\nu-2}{2}\right)\cdot \frac{(2\sigma)^{\frac{\nu}{2}}}{\nu S_\alpha m^{\frac{\nu-2}{2\nu}}}\cdot \displaystyle \min_{1\leq i\leq m} L_i^{\frac{(\alpha+1)(2-\nu)}{2\nu}}\cdot\left(f(x^{0})- f^*\right)^{\frac{\nu-2}{2}}\cdot k\right]^{\frac{2}{\nu-2}}}.
\end{equation}
This intermediate form of our convergence rate will facilitate the proof of our later convergence rate interpolation result (Corollary \ref{cor:str-cvx.interpolation}) so we have labeled it. 

For now though, we focus on processing this expression of the rate into its final form. The first step is to simply re-arrange this to
\[
\mathbb{E}[f(x^{k})]- f^*\leq \frac{(2\nu S_{\alpha})^{\frac{2}{\nu-2}}m^{\frac{1}{\nu}}(f(x^0)-f^*)}{\left[2\nu S_\alpha m^{\frac{\nu-2}{2\nu}}+ (\nu-2)(2\sigma)^{\frac{\nu}{2}}\displaystyle \min_{1\leq i\leq m} L_i^{\frac{(\alpha+1)(2-\nu)}{2\nu}}\cdot\left(f(x^{0})- f^*\right)^{\frac{\nu-2}{2}}k\right]^{\frac{2}{\nu-2}}}.
\]
By factoring $f(x^{0})- f^*$ out of both the numerator and denominator, and by applying Lemma \ref{lemma:nesterov-holder-cvx-str} to the previous expression, its right-hand side simplifies to
\begin{align*}\label{eq:rate-for-interpolation}
 &= \frac{(2\nu S_{\alpha})^{\frac{2}{\nu-2}}m^{\frac{1}{\nu}}}{\left[\frac{\nu S_\alpha m^{\frac{\nu-2}{2\nu}}}{(f(x^0)-f^*)^{\frac{\nu-2}{2}}}+ (\nu-2)(2\sigma)^{\frac{\nu}{2}}\displaystyle \min_{1\leq i\leq m} L_i^{\frac{(\alpha+1)(2-\nu)}{2\nu}}k\right]^{\frac{2}{\nu-2}}}\\
&\stackrel{\text{Lemma }\ref{lemma:nesterov-holder-cvx-str}}\leq \frac{(2\nu S_{\alpha})^{\frac{2}{\nu-2}}m^{\frac{1}{\nu}}}{\left[\frac{\nu S_\alpha m^{\frac{\nu-2}{2\nu}}}{\left( \frac{S_{\alpha}^{\frac{\nu-2}{2(\nu-1)}} (\nu-1)^{\frac{\nu-2}{2}} R^{\frac{\nu(\nu-2)}{2(\nu-1)}}}{\nu^{\frac{(\nu-2)^2}{2\nu}}2^{\frac{\nu-2}{2(\nu-1)}}}\right)}+ (\nu-2)(2\sigma)^{\frac{\nu}{2}}\displaystyle \min_{1\leq i\leq m} L_i^{\frac{(\alpha+1)(2-\nu)}{2\nu}}k\right]^{\frac{2}{\nu-2}}}\\
&= \frac{(2\nu S_{\alpha})^{\frac{2}{\nu-2}}m^{\frac{1}{\nu}}}{\left[\frac{2^{\frac{\nu-2}{2(\nu-1)}}m^{\frac{\nu-2}{2\nu}} S_{\alpha}^{\frac{\nu-1}{\nu}}\nu^{\frac{\nu^2-2\nu+4}{2\nu}}         
}{(\nu-1)^{\frac{\nu-2}{2}}R^{\frac{\nu(\nu-2)}{2(\nu-1)}}}+ (\nu-2)(2\sigma)^{\frac{\nu}{2}}\displaystyle \min_{1\leq i\leq m} L_i^{\frac{(\alpha+1)(2-\nu)}{2\nu}}k\right]^{\frac{2}{\nu-2}}}\\
&= \frac{(2\nu S_{\alpha})^{\frac{2}{\nu-2}}m^{\frac{1}{\nu}}(\nu-1)R^{\frac{\nu}{\nu-1}}}{\left[2^{\frac{\nu-2}{2(\nu-1)}}m^{\frac{\nu-2}{2\nu}} S_{\alpha}^{\frac{\nu-1}{\nu}}\nu^{\frac{\nu^2-2\nu+4}{2\nu}}+ R^{\frac{\nu(\nu-2)}{2(\nu-1)}}(\nu-1)^{\frac{\nu-2}{2}}(\nu-2)(2\sigma)^{\frac{\nu}{2}}\displaystyle \min_{1\leq i\leq m} L_i^{\frac{(\alpha+1)(2-\nu)}{2\nu}}k\right]^{\frac{2}{\nu-2}}}
\end{align*}

\end{proof}

We now make two crucial comparisons for the rates above. First, it is noteworthy that when $\nu=2$ in the strongly-convex regime, we recover the same linear rate as that in \cite{Nesterov12}. Second, our strongly convex sublinear rate in the $\nu>2$ setting, $\mathcal{O}(k^{-\frac{2}{\nu-2}})$, is indeed faster than the $\mathcal{O}(k^{-\frac{1}{\nu-1}})$ rate occurring in the merely convex case. 

To conclude this article, we demonstrate that, in the strongly convex case, when $\nu\to 2$ the intermediate form \eqref{eq:rate-for-interpolation} of the strongly convex sublinear rate of Theorem \ref{thm.conv.rate.strong-convex} converges to its $\nu=2$ linear rate of convergence.

\begin{corollary}\label{cor:str-cvx.interpolation}
\emph{(Interpolation of Linear and Sublinear Rate)} In the strongly convex setting of Theorem \ref{thm.conv.rate.strong-convex}, if $\nu\to 2${,} then the sublinear rate converges to a linear rate of convergence. More formally, for the convergence bound
\[
\mathbb{E}[f(x^{k})]- f^*\leq \frac{f(x^{0})- f^*}{\left[1+\left(\frac{\nu-2}{2}\right)\cdot \frac{(2\sigma)^{\frac{\nu}{2}}}{\nu S_\alpha m^{\frac{\nu-2}{2\nu}}}\cdot \displaystyle \min_{1\leq i\leq m} L_i^{\frac{(\alpha+1)(2-\nu)}{2\nu}}\cdot\left(f(x^{0})- f^*\right)^{\frac{\nu-2}{2}}\cdot k\right]^{\frac{2}{\nu-2}}},
\]
for $k\geq 0$, we observe the limiting result
\begin{multline*}
\lim_{\nu\to 2^+}\frac{f(x^{0})- f^*}{\left[1+\left(\frac{\nu-2}{2}\right)\cdot \frac{(2\sigma)^{\frac{\nu}{2}}}{\nu S_\alpha m^{\frac{\nu-2}{2\nu}}}\cdot \displaystyle \min_{1\leq i\leq m} L_i^{\frac{(\alpha+1)(2-\nu)}{2\nu}}\cdot\left(f(x^{0})- f^*\right)^{\frac{\nu-2}{2}}\cdot k\right]^{\frac{2}{\nu-2}}}\\
\leq\left(1-\frac{\sigma}{2S_\alpha} \right)^k\cdot \left(f(x^{0})- f^*\right)
\end{multline*}
\end{corollary}

\begin{proof}
The form of the sublinear convergence bound here was established by the immediately preceeding theorem in equation \eqref{eq:rate-for-interpolation}. The use of said rate in the proof of this corollary was foreshadowed there. 

To prove our main limit result, suppose for a moment that
\begin{equation}\label{eq:interp.limit}
\lim_{x \to 0^+} \left[1+g(x)\cdot x\right]^{-\frac{1}{x}}= e^{-g(0)}
\end{equation}
holds for any continuously differentiable $g:[0,\infty)\to[0,\infty)$ such that $g(0)>0$. Restating our sublinear convergence rate, we see that
\[
\frac{\mathbb{E}[f(x^{k})]- f^*}{f(x^0)-f^*}\leq 
 \left[ 1+g\left(\frac{\nu-2}{2}\right)\cdot \left(\frac{\nu-2}{2}\right)\right]^{\frac{2}{2-\nu}}{,}
\]
where 
\[
g(x) = \frac{(2\sigma)^{x+1}}{(2x+2)S_\alpha m^{\frac{x}{2x+2}}}\cdot \min_{1\leq i\leq m}L_i^{-\frac{(\alpha+1)x}{2x+2}} \cdot A_0^{x}\cdot k.
\]
Observe that $g$ is continuous differentiable and $g(0)=\sigma/S_\alpha>0$. Thus, it follows that
\begin{align*}
\frac{\mathbb{E}[f(x^{k})]- f^*}{f(x^0)-f^*}\leq & 
\lim_{\nu\to 2^+} \left[ 1+g\left(\frac{\nu-2}{2}\right)\cdot \left(\frac{\nu-2}{2}\right)\right]^{\frac{2}{2-\nu}}\\
\stackrel{\eqref{eq:interp.limit}}= &e^{-g(0)}=\left(e^{-\frac{\sigma}{S_\alpha}}\right)^k \\
\leq & \left(\frac{1}{1+\frac{\sigma}{S_\alpha}}\right)^k \leq \left(1-\frac{\sigma}{2S_\alpha} \right)^k{,}
\end{align*}
where we applied the standard inequalities $e^{x}\geq 1+x$ and $(1+x)^{-1}<(1-x/2)$ for all $x>0$ in the last two lines. Thus, we only need to prove \eqref{eq:interp.limit} to finish the proof.

The proof is a simply straightforward computation:
\begin{align*}
\lim_{x\to 0^+}\left[1+g(x)\cdot x\right]^{-\frac{1}{x}}&=\lim_{x\to 0^+}\exp\left(-\frac{\ln \left[1+g(x)\cdot x\right]}{x}\right)\\
&=\exp\left(-\lim_{x\to 0^+}\frac{\ln \left[1+g(x)\cdot x\right]}{x}\right)\\
&=\exp\left(-\lim_{x\to 0^+}\frac{g(x)+x\cdot g'(x)}{1+x\cdot g(x)}\right)\\
&=\exp(-g(0)){,}
\end{align*}
where continuity of $x\mapsto e^x$ is used in the second line, l'H\^opital's rule is used in the third line, and the definition of $g$ is used in the final line.
\end{proof}

\section{Proof of Lemma \ref{lemma:duality}}\label{app:duality}

In this section of the technical appendix, we prove Lemma \ref{lemma:duality}.\\

\noindent 1. We begin by choosing $x,y\in\mathbb{R}^d$. The inequality is trivial if $x=0$ so we assume $x\neq 0$. By the standard and $p$-norm versions of the Cauchy-Schwarz inequality, we compute

\begin{multline*}
\inner{x}{y}=\sum_{i=1}^m\inner{P_i x}{P_i y}\leq\sum_{i=1}^m \|P_i x\|\|P_i y\|
=\sum_{i=1}^m \left(L_i^{\alpha/p}\|P_i x\|\right)\left(\frac{\|P_i y\|}{L_i^{\alpha/p}}\right)\\
=\inner{\left(L_1^{\alpha/p}\|P_1 x\|,\ldots,L_1^{\alpha/p}\|P_1 x\|\right)}{\left(\frac{\|P_1 y\|}{L_1^{\alpha/p}},\ldots,\frac{\|P_m y\|}{L_m^{\alpha/p}}\right)}\\
\leq\left\|\left(L_1^{\alpha/p}\|P_1 x\|,\ldots,L_m^{\alpha/p}\|P_m x\|\right)\right\|_p\left\|\left(\frac{\|P_1 y\|}{L_1^{\alpha/p}},\ldots,\frac{\|P_m y\|}{L_m^{\alpha/p}}\right)\right\|_q=\|x\|_{\alpha,p}\|y\|_{-\alpha\frac{p}{q},q}
\end{multline*}
By the standard Cauchy-Schwarz inequality, the first inequality is obtained with equality if and only if one of $P_i x$ and $P_i y$ is a scalar multiple of the other for each $i=1,\ldots,m$. By the Cauchy-Schwarz inequality for $p$-norms,  and our assumption that $x\neq 0$, the second inequality obtains equality if and only if there is some $c\in\mathbb{R}$ such that
\[
c^q\cdot \left(L_1^{\alpha}\|P_1 x\|^p,\ldots,L_m^{\alpha}\|P_m x\|^p\right)=\left(\frac{\|P_1 y\|^q}{L_1^{\alpha q/p}},\ldots,\frac{\|P_m y\|^q}{L_m^{\alpha q/p}}\right).
\]
Assuming both inequalities hold then, we conclude  there are $c_1,\ldots,c_m,c\in\mathbb{R}$ such that $P_iy=c_i \cdot P_ix$ and $c^q\cdot L_i^{\alpha}\|P_i x\|^p=\frac{\|P_i y\|^q}{L_i^{\alpha q/p}}$ for $i=1,\ldots,m$. Fixing $1\leq i\leq m$, and combining the equalities, we see that $c^q\cdot L_i^{\alpha}\|P_i x\|^p=c_i^q\cdot \frac{\|P_i x\|^q}{L_i^{\alpha q/p}}$, so
\[
c_i=c\cdot L_i^{\alpha\left(1-\frac{q}{p}\right)}\|P_i x\|^{\frac{p}{q}-1}=c\cdot L_i^\alpha \|P_i x\|^{p-2}
\]
where we use the definition of $q$ as $p$'s H{\"o}lder conjugate to produce the second equality. This completes the proof of 1.\\

\noindent 2. Given $x\in\mathbb{R}^d$, consider the vector $(\|P_j x\|_j)_{j=1}^{m}$. For any $p\geq 2$, the norm equivalence inequality yields
\[
\|x\|_{\alpha,p} = \left(\sum_{j=1}^m L_j^\alpha\|P_jx\|_2^p\right)^{\frac{1}{p}} = \left(\sum_{j=1}^m \|L_j^{\alpha/p}\cdot P_jx\|_2^p\right)^{\frac{1}{p}} \leq \left(\sum_{j=1}^m \|L_j^{\alpha/p}\cdot P_jx\|_2^2\right)^{\frac{1}{2}} 
\]
But,
\[
\left(\sum_{j=1}^m \|L_j^{\alpha/p}\cdot P_jx\|_2^2\right)^{\frac{1}{2}} = \left(\sum_{j=1}^m L_j^{2\alpha/p} \|P_jx\|_2^2\right)^{\frac{1}{2}} \leq \|x\|_{\beta,2} \cdot \max_{1\leq j\leq m}L_m^{\frac{\alpha}{p}-\frac{\beta}{2}} 
\]
as, for any $i$, with $1\leq i\leq m$, $L_i^{2\alpha/p} \leq L^{\beta}\left(\max_{1\leq j\leq m} L_m^{\frac{\alpha}{p}-\frac{\beta}{2}}\right)^2$ and we complete the first part of the inequality. The second part is quite similar. By the equivalence norm inequality, 
\[
\|x\|_{\beta,2} = \left(\sum_{j=1}^m L_j^{\beta} \|P_jx\|_2^2\right)^{\frac{1}{2}}  = \left(\sum_{j=1}^m  \|L_j^{\beta/2} \cdot P_jx\|_2^2\right)^{\frac{1}{2}} \leq m^{\frac{1}{2}-\frac{1}{p}}\cdot \left(\sum_{j=1}^m  \|L_j^{\beta/2} \cdot P_jx\|_2^p\right)^{\frac{1}{p}}
\]
But
\[
\left(\sum_{j=1}^m  \|L_j^{\beta/2} \cdot P_jx\|_2^p\right)^{\frac{1}{p}} = \left(\sum_{j=1}^m  L_j^{\beta p/2}\|P_jx\|_2^p\right)^{\frac{1}{p}} \leq \|x\|_{\alpha,p}\cdot \max_{1\leq j\leq m}L_{j}^{\frac{\beta}{2}-\frac{\alpha}{p}} = \frac{\|x\|_{\alpha,p}}{\displaystyle \min_{1\leq j\leq m}L_{j}^{\frac{\alpha}{p}-\frac{\beta}{2}}}
\]
and we are done.

\section{Proof of Lemma \ref{lemma:nesterov-holder-cvx-str}}\label{app:level-set}

Suppose we are able to prove for all $x,y\in\mathbb{R}^d$ that 
\begin{equation}\label{eq:holder-modified}
f(y)\leq f(x)+\inner{\nabla f(x)}{y-x}+\left(\frac{\nu S_\alpha}{2}\right)^{\frac{1}{\nu-1}}\cdot\left(\frac{\nu-1}{\nu}\right)\cdot\|u\|_{(1+\alpha-\nu)(1-\nu),\frac{\nu}{\nu-1}}^{\frac{\nu}{\nu-1}}
\end{equation}
Then $x^*$, the first-order condition $\nabla f(x^*)$ naturally holds so \eqref{eq:holder-modified} implies
\begin{align*}
f(x)-f(x^*)&\leq \inner{\nabla f(x^*)}{x-x^*}+\left(\frac{\nu S_\alpha}{2}\right)^{\frac{1}{\nu-1}}\cdot\left(\frac{\nu-1}{\nu}\right)\cdot\|x-x^*\|_{(1+\alpha-\nu)(1-\nu),\frac{\nu}{\nu-1}}^{\frac{\nu}{\nu-1}}\\
&=\left(\frac{\nu S_\alpha}{2}\right)^{\frac{1}{\nu-1}}\cdot\left(\frac{\nu-1}{\nu}\right)\cdot\|x-x^*\|_{(1+\alpha-\nu)(1-\nu),\frac{\nu}{\nu-1}}^{\frac{\nu}{\nu-1}}\\
&\leq\left(\frac{\nu S_\alpha}{2}\right)^{\frac{1}{\nu-1}}\cdot\left(\frac{\nu-1}{\nu}\right)\cdot R(x)_{(1+\alpha-\nu)(1-\nu),\frac{\nu}{\nu-1}}^{\frac{\nu}{\nu-1}}
\end{align*}
and this completes the proof. Thus, it suffices to prove \eqref{eq:holder-modified}. Suppose that 
\begin{equation}\label{eq:holder-cocoercive}
\frac{2}{\nu S_\alpha}\|\nabla f(x)-\nabla f(y)\|_{1+\alpha-\nu,\nu}^{\nu-1}\leq \|x-y\|_{(1+\alpha-\nu)(1-\nu),\frac{\nu}{\nu-1}}
\end{equation}
holds for all $x,y\in\mathbb{R}^d$. Then, given $u\in\mathbb{R}^d$, the integral formulation of the mean value theorem states
\[
f(x+u)-f(x)=\int_0^1\inner{\nabla f(x+tu)}{u}dt,
\]
Thus, the Cauchy-Schwartz inequality and our previous inequality imply
\begin{align*}
f(x+u)-f(x)-\inner{\nabla f(x)}{u}&=\int_0^1\inner{\nabla f(x+tu)-\nabla f(x)}{u}dt\\
&\leq\int_0^1 \|\nabla f(x+tu)-\nabla f(x)\|_{1+\alpha-\nu,\nu}\|u\|_{(1+\alpha-\nu)(1-\nu),\frac{\nu}{\nu-1}}dt\\
&\leq\int_0^1 \left(\frac{\nu S_\alpha}{2}\|tu\|_{(1+\alpha-\nu)(1-\nu),\frac{\nu}{\nu-1}}\right)^{\frac{1}{\nu-1}}\|u\|_{(1+\alpha-\nu)(1-\nu),\frac{\nu}{\nu-1}}dt\\
&=\left(\frac{\nu S_\alpha}{2}\right)^{\frac{1}{\nu-1}}\|u\|_{(1+\alpha-\nu)(1-\nu),\frac{\nu}{\nu-1}}^{\frac{\nu}{\nu-1}}\int_0^1 t^{\frac{1}{\nu-1}}dt\\
&=\left(\frac{\nu S_\alpha}{2}\right)^{\frac{1}{\nu-1}}\cdot\left(\frac{\nu-1}{\nu}\right)\cdot\|u\|_{(1+\alpha-\nu)(1-\nu),\frac{\nu}{\nu-1}}^{\frac{\nu}{\nu-1}},
\end{align*}
so taking $u=y-x$ in \eqref{eq:holder-cocoercive} completes the proof. Consequently, we now need only prove the H{\"o}lderian co-coercivity condition \eqref{eq:holder-cocoercive}, which we do presently.

Given $y\in\mathbb{R}^d$, the function $x\mapsto\phi(x):=f(x)-f(y)-\inner{\nabla f(y)}{x-y}$ is readily seen to be H{\"o}lder block smooth because $f$. Moreover, $\phi$ has the same block H{\"o}lder smoothness constants and $\nabla \phi(x)=\nabla f(x)-\nabla f(y)$. Thus, by the Block H{\"o}lder Descent Lemma (Lemma \ref{lemma:block_descent}),
\begin{multline*}
f(x)-f(y)-\inner{\nabla f(y)}{y-x}= \phi(y)-\min_{z\in\mathbb{R}^d}\phi(z)\geq\max_{1\leq i\leq m}\left[\frac{1}{\nu L_i^{\nu-1}}\|\nabla_i f(x)-\nabla_i f(y)\|^\nu\right]\\
\geq \frac{1}{\nu S_{\alpha}}\sum_{i=1}^m\frac{L_i^\alpha}{L_i^{\nu-1}}\|\nabla_i f(x)-\nabla_i f(y)\|^\nu
= \frac{1}{\nu S_{\alpha}}\|\nabla f(x)-\nabla f(y)\|_{1+\alpha-\nu,\nu}^\nu,
\end{multline*}
which we may restate as
\[
f(x)\geq f(y)+\inner{\nabla f(y)}{y-x}+\frac{1}{\nu S_{\alpha}}\|\nabla f(x)-\nabla f(y)\|_{1+\alpha-\nu,\nu}^\nu.
\]
Adding this inequality to its analogue with the roles of $x$ and $y$ reversed, we see produce
\[
\frac{2}{\nu S_\alpha}\|\nabla f(x)-\nabla f(y)\|_{1+\alpha-\nu,\nu}^\nu\leq\inner{\nabla f(x)-\nabla f(y)}{x-y}.
\]
By Cauchy-Schwarz, we then see that
\[
\frac{2}{\nu S_\alpha}\|\nabla f(x)-\nabla f(y)\|_{1+\alpha-\nu,\nu}^\nu\leq\|\nabla f(x)-\nabla f(y)\|_{1+\alpha-\nu,\nu}\|x-y\|_{(1+\alpha-\nu)(1-\nu),\frac{\nu}{\nu-1}},
\]
or equivalently
\[
\frac{2}{\nu S_\alpha}\|\nabla f(x)-\nabla f(y)\|_{1+\alpha-\nu,\nu}^{\nu-1}\leq \|x-y\|_{(1+\alpha-\nu)(1-\nu),\frac{\nu}{\nu-1}}.
\]
Given $u\in\mathbb{R}^d$, the integral formulation of the mean value theorem states
\[
f(x+u)-f(x)=\int_0^1\inner{\nabla f(x+tu)}{u}dt,
\]
so the Cauchy-Schwarz inequality and the previous inequality imply
\begin{align*}
f(x+u)-f(x)-\inner{\nabla f(x)}{u}&=\int_0^1\inner{\nabla f(x+tu)-\nabla f(x)}{u}dt\\
&\leq\int_0^1 \|\nabla f(x+tu)-\nabla f(x)\|_{1+\alpha-\nu,\nu}\|u\|_{(1+\alpha-\nu)(1-\nu),\frac{\nu}{\nu-1}}dt\\
&\leq\int_0^1 \left(\frac{\nu S_\alpha}{2}\|tu\|_{(1+\alpha-\nu)(1-\nu),\frac{\nu}{\nu-1}}\right)^{\frac{1}{\nu-1}}\|u\|_{(1+\alpha-\nu)(1-\nu),\frac{\nu}{\nu-1}}dt\\
&=\left(\frac{\nu S_\alpha}{2}\right)^{\frac{1}{\nu-1}}\|u\|_{(1+\alpha-\nu)(1-\nu),\frac{\nu}{\nu-1}}^{\frac{\nu}{\nu-1}}\int_0^1 t^{\frac{1}{\nu-1}}dt\\
&=\left(\frac{\nu S_\alpha}{2}\right)^{\frac{1}{\nu-1}}\cdot\left(\frac{\nu-1}{\nu}\right)\cdot\|u\|_{(1+\alpha-\nu)(1-\nu),\frac{\nu}{\nu-1}}^{\frac{\nu}{\nu-1}}.
\end{align*}
Taking $u=y-x$ completes the proof.





\let\oldbibitem\bibitem
\renewcommand{\bibitem}{\setlength{\itemsep}{0pt}\oldbibitem}
\bibliographystyle{ieeetr}

\phantomsection
\addcontentsline{toc}{chapter}{REFERENCES}

\renewcommand{\bibname}{{\normalsize\rm REFERENCES}}

\bibliography{data/myReference.bib}


\end{document}